\documentclass{amsart}
\usepackage[utf8]{inputenc}

\usepackage{fullpage}
\usepackage[utf8]{inputenc}
\usepackage[T1]{fontenc}
\usepackage{graphicx,subfigure} 
\usepackage{amsmath,amsfonts,amssymb,mathtools}
\usepackage{stmaryrd}
\usepackage{mathrsfs,dsfont}
\usepackage{amsthm}
\usepackage{mathabx}
\usepackage{tabularx}
\usepackage{float}
\usepackage[usenames]{color}
\allowdisplaybreaks
\usepackage{amscd}
\usepackage{ulem}

\usepackage{hyperref}

\newtheorem{theo}{Theorem}
\newtheorem{lemma}{Lemma}[section]
\newtheorem{assum}{Assumption}[section]
\newtheorem{rem}{Remark}[section]
\newtheorem{propo}{Proposition}[section]

\newcommand{\N}{\mathbb{N}}
\newcommand{\R}{\mathbb{R}}
\newcommand{\E}{\mathbb{E}}

\newcommand{\IL}{\Lambda}

\newcommand{\HH}{\mathcal{H}}

\newcommand{\IW}{\mathcal{W}}
\newcommand{\Lf}{{\rm L}_f}
\newcommand{\IZ}{\mathcal{Z}}
\newcommand{\IP}{\mathcal{P}}

\title{Strong and weak rates of convergence in the Smoluchowski--Kramers approximation for stochastic partial differential equations}

\author{Charles-Edouard Br\'ehier}
\address{Universite de Pau et des Pays de l'Adour, E2S UPPA, CNRS, LMAP, Pau, France}
\email{charles-edouard.brehier@univ-pau.fr}

\author{Ziyi Lei}
\address{State Key Laboratory of Mathematical Sciences, Academy of Mathematics and Systems Science, Chinese Academy of Sciences, Beijing 100190, China}
\email{ziyilei@lsec.cc.ac.cn}

\date{\today}

\begin{document}

\begin{abstract}
We consider a class of stochastic damped semilinear wave equations, in the small-mass limit.  While it is well established that the solutions converge in probability to the solution of a stochastic semilinear heat equation, explicit convergence rates have not been obtained in the literature so far. In this work, we establish strong and weak rates of convergence for this Smoluchowski-Kramers approximation, which depend on the regularity of the driving Wiener process. Specifically, for the trace-class noise, both the strong and weak orders of convergence are $1$, whereas for space-time white noise (in dimension $1$) the strong and weak rates of convergence are $1/2$ and $1$ respectively. The proofs require appropriate moment bounds and regularity properties, where the dependence with respect to the small parameter is studied carefully. In addition, the proof of the weak convergence result requires original regularity results on solutions to infinite dimensional Kolmogorov equations, as well as an auxiliary spectral Galerkin approximation and Malliavin calculus techniques.

\end{abstract}

\maketitle

\section{Introduction}\label{sec:intro}

In this work, we are interested in establishing rates of convergence in the Smoluchowski--Kramers approximation for the following class of stochastic partial differential equations in the small $\epsilon$ regime:
\begin{equation}\label{eq:spde-intro-eps}
\epsilon^2\partial_{tt}u^\epsilon(t,x)=\Delta_{x}u^\epsilon(t,x)-\partial_t u^{\epsilon}(t,x)+f(u^\epsilon(t,x))+\dot{W}^{Q}(t,x),
\end{equation}
where $0\le t\le T$ and $x\in\mathcal{D}\subset \R^d$ ($d\in\N$), with homogeneous Dirichlet boundary conditions, and initial values $u^\epsilon(0,x)$ and $v^\epsilon(0,x)=\partial_tu^\epsilon(0,x)$. The nonlinearity $f$ is assumed to be a globally Lipschitz continuous mapping, and the evolution is driven by Gaussian noise which is white in time and possibly colored in space, with a covariance operator denoted by $Q$. As explained below, we treat \eqref{eq:spde-intro-eps} and related stochastic partial differential equations as stochastic evolution equations taking values in Hilbert spaces. In this rigorous framework (detailed in Section~\ref{sec:SPDE}, e.g., the formulation \eqref{eq:sdwe} of~\eqref{eq:spde-intro-eps}), the evolution is driven by a $Q$-Wiener process. We assume that covariance operator $Q$ commutes with the linear second-order operator driving the evolution ($\Delta_x$ with homogeneous Dirichlet boundary conditions for~\eqref{eq:spde-intro-eps} and~\eqref{eq:spde-intro-limit}). A more precise description of the driving noise and the assumptions on coefficient $f$ will be given in Section \ref{sec:setting}.

Equations of type~\eqref{eq:spde-intro-eps} can be used in modeling the evolution of the displacement field of particles suspended in a random continuous medium subject to deterministic and stochastic external forces, where $\epsilon^2$ represents the mass of the particles. The Smoluchowski--Kramers approximation investigates the small mass limit, that is, the limiting behavior of the solution $u^{\epsilon}$, as the mass $\epsilon$ vanishes. As demonstrated in \cite{MR2240691,MR2267703,MR3583470}, letting $\epsilon\to 0$ provides convergence in probability of $u^\epsilon$ to the solution $u^0$ of the semilinear heat equation
\begin{equation}\label{eq:spde-intro-limit}
\partial_tu^0(t,x)=\Delta_{x}u^0(t,x)+f(u^0(t,x))+\dot{W}^Q(t,x),
\end{equation}
for $0\le t\le T$, $x \in \mathcal{D}$, with homogeneous Dirichlet boundary conditions, and initial value
$u^0(0,\cdot)$ determined by the limit of $u^\epsilon(0,\cdot)$ as $\epsilon \to 0$. This means that one has for any $T\in(0,\infty)$, for all $\eta\in(0,1)$ one has
\[
\underset{\epsilon\to 0}\lim~\mathbb{P}\Big(\sup_{0\le t\le T}\int_{\mathcal D}|u^{\epsilon}(t,x)-u^0(t,x)|^2\,dx>\eta\Big)=0.
\]
This result reflects that the inertial effects of the particles become negligible in the small mass limit, allowing the system's displacement field to be fully characterized by a simpler first order reaction-diffusion model. Recent works have further investigated the generalization of this observation in the presence of a magnetic field (cf. \cite{
MR3575542,MR4142946}), for non-constant friction (cf. \cite{MR4993799,MR4759624,MR4657218,MR4413207}), which follows the works~\cite{MR3324144,MR2916095,MR3014105} for stochastic differential equations, and under various constraints (cf. \cite{MR4942808,MR4865023}). The literature features extensive research on these convergence properties. We refer the reader to~\cite{MR4235249,MR4056993,
 MR3531676,MR3245077,MR3916264,MR4908783} and the recent preprint \cite{GV25} for further discussions on the Smoluchowski--Kramers approximation in the context of stochastic partial differential equations and infinite-dimensional stochastic systems. Additionally, for a comprehensive overview of recent efforts examining the validity of this approximation across various types of stochastic differential equations, we refer to the articles~\cite{MR4632567,MR4713516,
MR4419570,MR4054345,MR4102448} and preprints~
\cite{SW24,SLW24,LQW24,ZW24,LLX26,ZLW25}.

To understand more precisely how solutions to equations~ \eqref{eq:spde-intro-eps} and to the limit equation~ \eqref{eq:spde-intro-limit} differ, it is of fundamental importance to establish explicit convergence rates in the Smoluchowski--Kramers approximation as $\epsilon\to 0$, which, to the best of our knowledge, have not been exhibited so far. To achieve this, we explore novel techniques to effectively control the highly oscillatory behavior of the velocity component as the mass parameter vanishes. We establish the different strong and weak rates of convergence in the Smoluchowski--Kramers approximation for \eqref{eq:spde-intro-eps}, employing novel techniques in particular for the weak convergence (Kolmogorov and Poisson equations, Malliavin calculus). Our main results can be described as follows. Under appropriate conditions on the initial values and on the nonlinearity, one has the following strong and weak error estimates: 
\begin{itemize}
    \item  for all $p\in[1,\infty)$, $T\in(0,\infty)$ and $\alpha\in(0,\beta)$, there exists $C_{\alpha,p}(T)\in(0,\infty)$ such that one has
\begin{equation}\label{eq:strong-intro}
\underset{t\in[0,T]}\sup~\bigl(\E[\|u^\epsilon(t,\cdot)-u^0(t,\cdot)\|_{L^2(\mathcal{D})}^p]\bigr)^{\frac1p}\le C_{\alpha,p}(T)\epsilon^\alpha.
\end{equation}
\item for all $T\in(0,\infty)$, $\alpha\in[0,\beta)$, and for function $\varphi\colon H\to\R$ of class $\mathcal{C}^4$, with bounded derivatives up to order $4$, there exists $C_{\alpha}(T,\varphi)\in(0,\infty)$ such that for all $\epsilon\in(0,1)$ one has
\begin{equation}\label{eq:weak-intro}
\underset{0\le t\le T}\sup~\bigl|\E[\varphi(u^{\epsilon}(t))]-\E[\varphi(u^{0}(t))]\bigr|\le C_{\alpha}(T,\varphi)\epsilon^{\min(2\alpha,1)}. 
\end{equation}
\end{itemize}
where the parameter $\beta\in(0,1]$ describes the regularity of the noise (see Assumption~\ref{ass:noise}). For instance, one has $\beta=1$ if the equation is driven by trace-class noise, and $\beta=1/2$ if the equation is driven by space-time white noise in dimension $d=1$. We refer to Theorems~\ref{theo:strong} and~\ref{theo:weak} in Section~\ref{sec:main} for more rigorous and general statements, using the notation and assumptions introduced in Section~\ref{sec:prelim}.
Note that the strong and weak orders coincide only if $\beta=1$, i.e. for trace-class noise. This behavior is the same as for the Smoluchowski--Kramers approximation for finite-dimensional stochastic differential equations (see for instance~\cite{MR4632567}). If $\beta\in(0,1)$, then the weak order of convergence $\min(2\beta,1)$ is strictly larger than the strong order $\beta$. Note that the weak order is limited to $1$ even for finite dimensional systems, this justifies why the weak order is not $2\beta$ in general. If $\beta\in(0,1/2]$ the weak order is twice the strong order, in particular this result holds for space-time white noise in dimension $d=1$.

Compared to the convergence in probability  between equation \eqref{eq:spde-intro-eps} and the limit equation \eqref{eq:spde-intro-limit} obtained in \cite{MR2240691,MR2267703,MR3583470}, deriving a strong convergence rate for the same framework poses significant analytical challenges. The primary difficulties lie in overcoming the coupling between parameter $\epsilon$ and the nonlinearity and the stochastic perturbation term, and establishing further the uniform bounds with respect to $\epsilon$ in the regularity estimates for the solutions associated with the stochastic damped wave equations (see e.g., Proposition~\ref{propo:sdwe-momentbounds}). In addition, a careful analysis of properties of the semigroups associated with the damped wave equation and the heat equation is needed to compute the error bounds in the strong approximation. With the help of properties and convergence on the semigroups derived in Section~\ref{sec:aux}, we obtain the rate of strong convergence in \eqref{eq:strong-intro}.

In contrast, deriving a precise rate of weak convergence (in the distributional sense) between equation \eqref{eq:spde-intro-eps} and the limit equation \eqref{eq:spde-intro-limit}  involves a more complicated and technical analysis. To obtain the optimal rate of weak convergence as shown in \eqref{eq:weak-intro}, it requires not only the careful analysis of properties of the semigroups derived in Section~\ref{sec:aux}, but also the techniques in obtaining the regularity properties for solutions to the Kolmogorv equation associated to \eqref{eq:spde-intro-eps} (see e.g., Proposition~\ref{propo:regularity-PhiepsilonN}), and dealing with the singular terms containing $\frac{1}{\epsilon}, \frac{1}{\epsilon^2}$, and the highly oscillatory stochastic convolution terms (see e.g., $E_{1,2,2,3}^{\epsilon,N}$ and $E_{1,2,3,1,3}^{\epsilon, N}$). As shown in the graph below,
\[
\begin{CD}
u^{\epsilon,N}(t)@>{\epsilon\to 0}>> u^{0,N}(t) \\
@VV{N\to\infty}V        @VV{N\to\infty}V\\
u^{\epsilon}(t)  @>{\epsilon\to 0}>> u^{0}(t),
\end{CD}
\]
 using the auxiliary spectral Galerkin approximation, we divide the proof of \eqref{eq:weak-intro} into the weak convergence analysis of the spectral Galerkin approximations for the stochastic damped wave equation and stochastic heat equation respectively: \[\sup_{t\in[0,T]}|\E[\varphi(u^{\epsilon}(t))]-\E[\varphi(u^{\epsilon,N}(t))]|,~~~\quad \sup_{t\in[0,T]}|\E[\varphi(u^{0}(t))]-\E[\varphi(u^{0,N}(t))]|,\]
and the Smoluchowski--Kramers approximation in finite dimension:
$\sup_{t\in[0,T]}|\E[\varphi(u^{\epsilon,N}(t))]-\E[\varphi(u^{0,N}(t))]|$. The proofs of the three parts are based on the Kolmogorov equations approach associated to equations~\eqref{eq:spde-intro-limit} and~\eqref{eq:spde-intro-eps}. Whereas the regularity properties for solutions to the Kolmogorov equation associated to~\eqref{eq:spde-intro-limit} are standard, in Section~\ref{sec:auxweak} we obtain original regularity properties for solutions to the Kolmogorov equation associated to~\eqref{eq:spde-intro-eps}. The proof of \eqref{eq:weak-intro}, presented in Section~\ref{sec:proof-theo-weak}, is quite technical. In particular, it requires the auxiliary Poisson equation and Malliavin calculus techniques to achieve the rate of convergence $\min(2\beta,1)$ instead of $\beta$. The techniques developed in Sections~\ref{sec:auxweak} and~\ref{sec:proof-theo-weak} may be of interest to study other questions, such as convergence of asymptotic preserving schemes in the spirit of~\cite{MR4632567}.

This manuscript is organized as follows. Section~\ref{sec:prelim} provides the necessary notation and assumptions, and contains a description of the considered stochastic evolution equations. The main result, Theorem~\ref{theo:strong}, is stated in Section~\ref{sec:main}. Section~\ref{sec:aux} presents technical properties of the semigroup associated with the heat and the damped wave equations, where the dependence with respect to the parameter $\epsilon$ is analyzed carefully. Section~\ref{sec:properties} is devoted to present moment bounds and regularity properties of solutions to the stochastic partial differential equations~\eqref{eq:spde-intro-eps} (uniformly with respect to $\epsilon$) and~\eqref{eq:spde-intro-limit}. The proof of Theorem~\ref{theo:strong} is provided in Section~\ref{sec:proof-theo-strong}. Section~\ref{sec:auxweak} then presents additional results required for the proof of Theorem~\ref{theo:weak}. The proof of Theorem~\ref{theo:weak} is then provided in Section~\ref{sec:proof-theo-weak}.

\section{Preliminaries}\label{sec:prelim}

\subsection{Setting}\label{sec:setting}


\subsubsection{Notation}\label{sec:notation}
Let $d\in\N$ be a positive integer and let $\mathcal{D}\subset\R^d$ be an open bounded domain with polygonal boundary $\partial\mathcal{D}$, for instance $\mathcal{D}=(0,1)^d$. Define the separable Hilbert space $H=L^2(\mathcal{D})$. The inner product and the norm in $H$ are denoted by $\langle \cdot,\cdot \rangle_H$ and $\|\cdot\|_H$ respectively. Owing to the Riesz theorem, the dual space of $H$ is identified with $H$.

Let us denote by $-\IL$ the Laplace operator endowed with homogeneous Dirichlet boundary conditions. Then $\IL$ is a self-adjoint unbounded linear operator from $H$ to $H$, with domain $D(\IL)=H^2(\mathcal{D})\cap H_0^1(\mathcal{D})$, and there exists a complete orthonormal system $\bigl(e_n\bigr)_{n\in\N}$ of $H$ and a nondecreasing sequence $\bigl(\lambda_n\bigr)_{n\in\N}$ of positive real numbers such that one has
\begin{equation}\label{eq:Laplace}
\IL e_n=\lambda_n e_n.
\end{equation}
Note that the eigenvalue $\lambda_1$ is positive.
In addition, there exists ${\rm c}_{\IL,d}\in(0,\infty)$ such that $\lambda_n\sim {\rm c}_{\IL,d}n^{2/d}$ when $n\to\infty$.

For any nonnegative real number $\alpha\in[0,\infty)$, define
\[
H^\alpha=\{u\in H;~\sum_{n=1}^{\infty}\lambda_n^{\alpha}\langle u,e_n\rangle_H^2<\infty\}.
\]
Moreover, for all $u_1,u_2,u\in H^\alpha$, set
\[
\langle u_1,u_2 \rangle_{H^\alpha}=\sum_{n=1}^{\infty}\lambda_n^{\alpha}\langle u_1,e_n \rangle_H \langle u_2,e_n \rangle_H,\quad
\|u\|_{H^\alpha}^2=\langle u,u \rangle_{H^\alpha}=\sum_{n=1}^{\infty}\lambda_n^{\alpha}\langle u,e_n \rangle_H^2.
\]
For any $\alpha\in[0,\infty)$, the space $H^\alpha$, equipped with the inner product $\langle \cdot,\cdot \rangle_{H^\alpha}$ and the norm $\|\cdot\|_{H^\alpha}$, is a separable Hilbert space.

For any nonnegative real number $\alpha\in[0,\infty)$, the space $H^{-\alpha}$ is defined as the closure of the space
\[
\{u\in H;~\sum_{n=1}^{\infty}\lambda_n^{-\alpha}\langle u,e_n\rangle_H^2<\infty\},
\]
with inner product $\langle\cdot,\cdot\rangle_{H^{-\alpha}}$ and norm $\|\cdot\|_{H^{-\alpha}}$ defined as follows: for all $u_1,u_2,u\in H^{-\alpha}$, set
\[
\langle u_1,u_2 \rangle_{H^{-\alpha}}=\sum_{n=1}^{\infty}\lambda_n^{-\alpha}\langle u_1,e_n \rangle_H \langle u_2,e_n \rangle_H,\quad
\|u\|_{H^{-\alpha}}^2=\langle u,u \rangle_{H^{-\alpha}}=\sum_{n=1}^{\infty}\lambda_n^{-\alpha}\langle u,e_n \rangle_H^2.
\]
For any $\alpha\in[0,\infty)$, the space $H^{-\alpha}$ equipped with the inner product $\langle \cdot,\cdot \rangle_{H^{-\alpha}}$ and the norm $\|\cdot\|_{H^{-\alpha}}$ is a separable Hilbert space. It can be identified with the dual space of $H^{\alpha}$. Note also that $H^0=H$.

For any real number $\alpha\in\mathbb{R}$, introduce the product space
\[
\HH^\alpha=H^\alpha \times H^{\alpha-1}.
\]
Moreover, for all $x_1=(u_1,v_1), x_2=(u_2,v_2),x=(u,v)\in\HH^\alpha$, define
\begin{align*}
\langle x_1,x_2 \rangle_{\HH^\alpha}&=\bigl\langle (u_1,v_1),(u_2,v_2) \bigr\rangle_{\HH^\alpha}=\langle u_1,u_2 \rangle_{H^\alpha}+\langle v_1,v_2 \rangle_{H^{\alpha-1}},\\
\|x\|_{\HH^\alpha}^2&=\langle x, x \rangle_{\HH^\alpha}=\bigl\langle (u,v),(u,v)\bigr\rangle_{\HH^\alpha}=\|u\|_{H^\alpha}^2+\|v\|_{H^{\alpha-1}}^2.
\end{align*}
For any $\alpha\in\mathbb{R}$, the space $\HH^\alpha$, equipped with the inner product $\langle \cdot,\cdot \rangle_{\HH^\alpha}$ and the norm $\|\cdot\|_{\HH^\alpha}$, is a separable Hilbert space. When $\alpha=0$, the notation $\HH=\HH^0$, $\langle \cdot,\cdot \rangle_{\HH}=\langle \cdot,\cdot \rangle_{\HH^0}$ and $\|\cdot\|_{\HH}=\|\cdot\|_{\HH^0}$ is used in the sequel.

Observe that if $\alpha_1<\alpha_2$, one has $H^{\alpha_2}\subset H^{\alpha_1}$ and $\HH^{\alpha_2}\subset \HH^{\alpha_1}$, with continuous embeddings: for all $u\in H^{\alpha_2}$ and all $x\in\HH^{\alpha_2}$, one has
\begin{equation}\label{eq:comparenorms}
\|u\|_{H^{\alpha_1}}\le \lambda_1^{-\frac{\alpha_2-\alpha_1}{2}}\|u\|_{H^{\alpha_2}}~,\quad \|x\|_{\HH^{\alpha_1}}\le \lambda_1^{-\frac{\alpha_2-\alpha_1}{2}}\|x\|_{\HH^{\alpha_2}}.
\end{equation}

Let $\alpha\in\R$ be an arbitrary real number. Define the projection operators $\Pi_u\colon \HH^\alpha\rightarrow H^\alpha$ and $\Pi_v\colon \HH^\alpha\rightarrow H^{\alpha-1}$ such that for all $(u,v)\in\HH^\alpha$ one has
\begin{equation}\label{eq:Pi_uv}
\Pi_u(u,v)=u\ \  \mbox{and} \ \ \Pi_v(u,v)=v.
\end{equation}
The operators $\Pi_u$ and $\Pi_v$ depend on $\alpha$, however the dependence is omitted to simplify notation. Note that for all $\alpha\in\R$ and all $x=(u,v)\in\HH^\alpha$, one has
\begin{equation}\label{eq:Pi-bound}
\max\bigl(\|\Pi_u x\|_{H^\alpha},\|\Pi_v x\|_{H^{\alpha-1}}\bigr)\le \|x\|_{\HH^\alpha}.
\end{equation}

For all $\alpha\in\R$, let the linear operator $\IL^\alpha$ be defined as follows: for all $u\in H^{2\alpha}$, set
\[
\IL^{\alpha}u=\sum_{n\in\N}\lambda_n^{\alpha}\langle u,e_n \rangle_H e_n.
\]
For any positive real number $\alpha\in(0,\infty)$, the linear operator $\IL^\alpha$ is unbounded from $H$ to $H$, and has domain $D(\IL^\alpha)=H^{2\alpha}$. Moreover, for all $u\in H^{2\alpha}$, one has $\|\IL^\alpha u\|_H=\|u\|_{H^{2\alpha}}$.

For any $\epsilon\in(0,\infty)$, the linear operator $A_\epsilon$ defined as follows: for all $\alpha\in\R$ and for all $x=(u,v)\in\HH^{\alpha}$, set
\begin{equation}\label{eq:Aepsilon}
A_\epsilon x=\bigl(\frac{v}{\epsilon}, -\frac{\IL u}{\epsilon}-\frac{v}{\epsilon^2}\bigr)\in \HH^{\alpha-1}.
\end{equation}
Note that the linear operator $A_\epsilon$ can be considered as an unbounded linear operator on $\HH$ with domain $\HH^1$.

\subsubsection{Semigroups}\label{sec:semigroups}

The linear operator $\IL$ generates a semigroup $\bigl(e^{-t\IL}\bigr)_{t\ge 0}$ of bounded linear operators, defined as follows:
\begin{equation}\label{eq:expL}
e^{-t\IL}u=\sum_{n\in\N}e^{-t\lambda_n}\langle u,e_n\rangle_H e_n,\qquad \forall~t\ge 0,\quad \forall~u\in H,
\end{equation}
where for all $n\in\N$ the mapping $t\mapsto u_n(t):=e^{-t\lambda_n}\langle u,e_n\rangle_H$ is the solution of the linear ordinary differential equation
\[
u_n'(t)=-\lambda_n u_n(t),\quad t\ge 0;\quad u_n(0)=\langle u,e_n\rangle.
\]

In addition, for any $\epsilon\in(0,\infty)$, the linear operator $A_\epsilon$ generates a semigroup $\bigl(e^{tA_\epsilon}\bigr)_{t\ge 0}$ of bounded linear operators, defined as follows: 
\begin{equation}\label{eq:expAeps}
e^{tA_\epsilon}x=\sum_{n\in\N}\bigl(f_n^\epsilon(t)e_n,g_n^\epsilon(t)e_n\bigr)
,\qquad \forall~t\ge 0,\quad \forall~x=(u,v)\in\HH,\end{equation}
where for all $n\in\N$ the mapping $t\mapsto \bigl(f_n^\epsilon(t),g_n^\epsilon(t)\bigr)$ is the solution of the linear two-dimensional system of ordinary differential equations
\begin{equation}\label{eq:fngn}
\frac{d}{dt}\begin{pmatrix} f^\epsilon_n(t) \\ g^\epsilon_n(t) \end{pmatrix}=\begin{pmatrix} 0 & \frac{1}{\epsilon} \\ -\frac{\lambda_n}{\epsilon} & -\frac{1}{\epsilon^2} \end{pmatrix}\begin{pmatrix} f^\epsilon_n(t) \\ g^\epsilon_n(t) \end{pmatrix}=\begin{pmatrix}\frac{1}{\epsilon}g^\epsilon_n(t) \\ -\frac{\lambda_n}{\epsilon}f^\epsilon_n(t)-\frac{1}{\epsilon^2}g^\epsilon_n(t)\end{pmatrix},~~t\ge 0~; \quad \begin{pmatrix} f^\epsilon_n(0) \\ g^\epsilon_n(0)\end{pmatrix}=\begin{pmatrix}\langle u,e_n\rangle_H \\ \langle v,e_n\rangle_H \end{pmatrix}.
\end{equation}

Properties of semigroups $\bigl(e^{-t\IL}\bigr)_{t\ge 0}$ and $\bigl(e^{tA_\epsilon}\bigr)_{t\ge 0}$ are studied in Section~\ref{sec:expL} and Section~\ref{sec:expAeps} respectively.

\subsubsection{Nonlinearity}\label{sec:nonlinearity}
The nonlinearity is given by a mapping $f:u\in H\mapsto f(u)\in H$. Given such a function $f$, the mapping $F:\HH\to\HH^1$ is then defined as follows: for all $(u,v)\in\HH$,
\begin{equation}\label{eq:F}
F(u,v)=\bigl(0,f(u)\bigr).
\end{equation}
First, the mapping $f$ is assumed to satisfy the following Lipschitz continuity condition.
\begin{assum}\label{ass:fLip}
The mapping $f$ is globally Lipschitz continuous, with Lipschitz constant $\Lf$: one has
\[
\Lf=\underset{u_1,u_2\in H,u_1\neq u_2}\sup~\frac{\|f(u_2)-f(u_1)\|_H}{\|u_2-u_1\|_H}\in[0,\infty).
\]
\end{assum}
In particular, observe that $f:H\to H$ has at most linear growth. Moreover, note that under Assumption~\ref{ass:fLip}, the mapping $F:\HH\to\HH$ is also globally Lipschitz continuous. Indeed, using the definitions of the norms $\|\cdot\|_\HH$ and $\|\cdot\|_{\HH^1}$ and the inequality~\eqref{eq:comparenorms}, for all $x_1=(u_1,v_1),x_2=(u_2,v_2)\in \HH$, one has
\[
\|F(x_2)-F(x_1)\|_{\HH}\le \lambda_1^{-\frac12}\|F(x_2)-F(x_1)\|_{\HH^{1}} \le \lambda_1^{-\frac12}\|f(u_2)-f(u_1)\|_{H}\le \frac{\Lf}{\sqrt{\lambda_1}}\|u_2-u_1\|_{H}\le \frac{\Lf}{\sqrt{\lambda_1}}\|x_2-x_1\|_{\HH}.
\]

Supplementary regularity conditions are imposed on the nonlinearity $f$ for the proof of Theorem~\ref{theo:weak}.
\begin{assum}\label{ass:mapping-f}
The mapping $f:H\to H$ is $4$ times Fr\'echet differentiable, with bounded derivatives of order $1$ to $4$.
For all $k\in\{1,2,3,4\}$, set
\[
\vvvert f \vvvert_k=\sum_{\ell=1}^{k}\underset{u\in H}\sup~\underset{h_1,\ldots,h_\ell\in H\setminus\{0\}}\sup~\frac{\|D^\ell f(u).(h_1,\ldots,h_\ell)\|_H}{\|h_1\|_H \cdots\|h_\ell\|_H}\in[0,\infty),
\]
where $D^\ell f$ denotes the Fr\'echet derivative of $f$ of order $\ell$.

In addition, it is assumed that $f(H^1)\subset H^1$ and that $f:H^1\to H^1$ has at most linear growth, i.e., there exists $C_1\in(0,\infty)$ such that for all $u\in H^1$ one has
\begin{equation}\label{eq:falpha}
\|f(u)\|_{H^1}\le C_\alpha\bigl(1+\|u\|_{H^1}\bigr),
\end{equation}
and that for all $\alpha\in(0,1)$, all $\kappa\in(0,1-\alpha] $, there exists $C_{\alpha,\kappa}\in(0,\infty)$ such that for all $u\in H^{\alpha+\kappa}$ one has
\begin{equation}\label{eq:falpha1}
\|f(u)\|_{H^\alpha}\le C_{\alpha,\kappa}\bigl(1+\|u\|_{H^{{\alpha+\kappa}}}\bigr).
\end{equation}
Finally, for all $\alpha\in[0,1/2)$, there exists $C_\alpha\in(0,\infty)$ such that for all $u_1,u_2\in H^\alpha$ one has
\begin{equation}\label{eq:Dfalpha}
\|f(u_2)-f(u_1)\|_{H^{-1}}\le C_\alpha\bigl(1+\|u_1\|_{H^\alpha}+\|u_2\|_{H^{\alpha}}\bigr)\|u_2-u_1\|_{H^{-\alpha}}
\end{equation}
and for all $u\in H^{\alpha}$ and $h\in H^{-\alpha}$ one has
\begin{equation}\label{eq:DfalphaBis}
\|Df(u).h\|_{H^{-1}}\le C_\alpha\bigl(1+\|u\|_{H^\alpha}\bigr)\|h\|_{H^{-\alpha}}.
\end{equation}

\end{assum}

Note that the condition~\eqref{eq:Dfalpha} from Assumption~\ref{ass:mapping-f} implies the following additional regularity result: if $\beta\in[\frac12,1]$, for all $u_1,u_2\in H^\beta$, one has
\begin{equation}\label{eq:Dfalpha2}
\|f(u_2)-f(u_1)\|_{H^{-1}}\le C_\beta\bigl(1+\|u_1\|_{H^\beta}+\|u_2\|_{H^{\beta}}\bigr)\|u_2-u_1\|_{H^{\beta-1}}.
\end{equation}
That result follows from applying~\eqref{eq:Dfalpha} with $\alpha=1-\beta\in[0,\frac12]$, noting that $\beta\ge 1-\beta$ and using the inequality~\eqref{eq:comparenorms}. Similarly, for all $u\in H^\beta$ and all $h\in H^{\beta-1}$ one has
\begin{equation}\label{eq:DfalphaBis2}
\|Df(u).h\|_{H^{-1}}\le C_\beta\bigl(1+\|u\|_{H^\beta}\bigr)\|h\|_{H^{\beta-1}}.
\end{equation}

\subsubsection{The Wiener process}\label{sec:Wiener}

Let $\bigl(\Omega,\mathcal{F},\mathbb{P})$ be a probability space, where the expectation operator is denoted by $\E[\cdot]$. A filtration $\bigl(\mathcal{F}_t\bigr)_{t\ge 0}$ satisfying the usual conditions is considered. Let $\bigl(\beta_n\bigr)_{n\in\N}$ be a family of independent standard real-valued Wiener processes, adapted to the filtration $\bigl(\mathcal{F}_t\bigr)_{t\ge 0}$.

Given a sequence $\bigl(q_n\bigr)_{n\in\N}$ of nonnegative real numbers, let us define for all $t\ge 0$
\begin{equation}\label{eq:WQ}
W^Q(t)=\sum_{n\in\N}\sqrt{q_n}\beta_n(t)e_n.
\end{equation}
The stochastic process $(W^Q(t))_{t\ge0}$ is then a $Q$-Wiener process, with the covariance operator $Q$ defined by
\[
Qh=\sum_{n\in\N} q_n \langle h,e_n\rangle_{H}e_n.
\]
More precisely, the following conditions are imposed.

\begin{assum}\label{ass:noise}
The sequence $\bigl(q_n\bigr)_{n\in\N}$ is bounded.
Moreover, there exists $\beta\in(0,1]$ such that, one has
\begin{equation}\label{eq:condition_beta} 
\sum_{n\in\N}q_n\lambda_n^{\alpha-1}<\infty,\quad \forall~\alpha\in[0,\beta).
\end{equation}
\end{assum}

Under Assumption~\ref{ass:noise}, the covariance operator $Q$ is a bounded linear self-adjoint operator on $H$. In addition,  its square root $Q^{\frac12}$ is defined as the bounded linear operator on $H$ given by
\[
Q^{\frac12}h=\sum_{n\in\N} \sqrt{q}_n \langle h,e_n\rangle_{H}e_n,\qquad \forall~h\in H.
\]
The condition~\eqref{eq:condition_beta} can then be rewritten as
\[
\|\IL^{\frac{\alpha-1}{2}}Q^{\frac12}\|_{\mathcal{L}_2(H)}^2=\sum_{n\in\N}\|\IL^{\frac{\alpha-1}{2}}Q^{\frac12}e_n\|_{H}^2 ~= \sum_{n\in\N}q_n\lambda_n^{\alpha-1}<\infty
\]
where $\mathcal{L}_2(H)$ denotes the space of Hilbert--Schmidt operators from $H$ to $H$, and $\|\cdot\|_{\mathcal{L}_2(H)}$ denotes the associated norm. The condition~\eqref{eq:condition_beta} ensures that the Wiener process $\bigl(W^Q(t)\bigr)_{t\ge 0}$ takes values in $H^{\alpha-1}$ for all $\alpha\in[0,\beta)$.

The $Q$-Wiener process $\bigl(W^Q(t)\bigr)_{t\ge 0}$ can be written as $W^Q(t)=Q^{\frac12}W(t)$, where $\bigl(W(t)\bigr)_{t\ge 0}$ is a cylindrical Wiener process derfined as
\begin{equation}\label{eq:W}
W(t)=\sum_{n\in\N}\beta_n(t)e_n.
\end{equation}

In addition, note that the covariance operator $Q$ commutes with the linear operator $\IL$.

For any $\epsilon>0$, the stochastic evolution equations considered in this paper are driven by the $\HH$-valued Wiener process $\bigl(\IW^Q(t)\bigr)_{t\ge 0}$ given by
\begin{equation}\label{eq:IWQ}
\IW^Q(t)=\bigl(0,W^Q(t)\bigr)=\sum_{n\in\N}\sqrt{q_n}\beta_n(t)\bigl(0,e_n\bigr),\quad \forall~t\ge 0.
\end{equation}
Under the condition~\eqref{eq:condition_beta}, the Wiener process $\bigl(\IW^Q(t)\bigr)_{t\ge 0}$ takes values in $\mathcal{H}^\alpha$ for all $\alpha\in[0,\beta)$. Moreover, for all $\alpha\in[0,\beta)$ and $p\in[1,\infty)$, there exists $C_{\alpha,p}\in(0,\infty)$ such that for all $t\ge s\ge 0$ one has
\begin{equation}\label{eq:incrementsWQ}
\E[\|\IW^Q(t)-\IW^Q(s)\|_{\HH^\alpha}^p]=C_{\alpha,p}(t-s)^{\frac{p}{2}}.
\end{equation}

For instance, one may consider the identity operator $Q=I$ in dimension $d=1$, in which case the condition~\eqref{eq:condition_beta} is satisfied with $\beta=1/2$. If $Q$ is a trace-class operator, i.e. if $\sum_{n=1}^{\infty}q_n<\infty$, then the condition~\eqref{eq:condition_beta} is satisfied with $\beta=1$.

\subsubsection{Spectral Galerkin approximation}\label{sec:Galerkin}

For any positive integer $N\in\N$, let us consider the finite dimensional spaces $H_N\subset H$ and $\HH_N\subset \HH$ defined as
\begin{equation}\label{eq:HN-HHN}
H_N={\rm span}\left\{e_1,\ldots,e_N\right\}~,\quad \HH_N={\rm span}\left\{(e_1,0),(0,e_1),\ldots,(e_N,0),(0,e_N)\right\}
\end{equation}
Observe that for all $n\in\N$ one has $\|e_n\|_H=\|(0,e_n)\|_{\HH}=1$ and $\|(0,e_n)\|_\HH=\|e_n\|_{H^{-1}}=\lambda_n^{-\frac12}$. As a result, $\{e_1,\ldots,e_N\}$ is an orthonormal system of the finite dimensional space $H_N$. Similarly, one checks that $\{(e_1,0),(0,\sqrt{\lambda_1}e_1),\ldots,(e_N,0),(0,\sqrt{\lambda_N}e_N)\}$ is an orthonormal system of the finite dimensional space $\HH_N$. Note that one has $H_N\subset H^\alpha$ and $\HH_N\subset\HH^\alpha$ for all $N\in\N$ and $\alpha\in\R$.

In addition, for any positive integer $N\in\N$, introduce the associated orthogonal projection operators denoted by $P_N$ and $\IP_N$: for all $u=\sum_{n\in\N}\langle u,e_n\rangle_H e_n\in H$, set
\begin{equation}\label{eq:PN}
P_Nu=\sum_{n=1}^{N}\langle u,e_n\rangle_H e_n,
\end{equation}
and for all $x=(u,v)\in\HH$, set
\begin{equation}\label{eq:IPN}
\IP_Nx=\bigl(P_Nu,P_Nv\bigr)=\Bigl(\sum_{n=1}^{N}\langle u,e_n\rangle_H e_n,\sum_{n=1}^{N}\langle v,e_n\rangle_H e_n\Bigr)=\sum_{n=1}^{N}\langle x,(e_n,0)\rangle_{\HH}(e_n,0)\!+\!\sum_{n=1}^{N}\lambda_n\langle x,(0,e_n)\rangle_{\HH}(0,e_n).
\end{equation}
Note that the orthogonal projection operators $P_N$ and $\IP_N$ satisfy the following property: for all $\alpha\in\R$
\begin{equation}\label{eq:boundPN}
\underset{u\in H^\alpha\setminus\{0\}}\sup~\frac{\|P_N u\|_{H^\alpha}}{\|u\|_{H^\alpha}}=1,\quad
\underset{x\in\HH^\alpha\setminus\{0\}}\sup~\frac{\|\IP_N x\|_{\HH^\alpha}}{\|x\|_{\HH^\alpha}}=1,\quad\forall~N\in\N.
\end{equation}
Moreover, one has the following result: for all $\alpha\in[0,\infty)$, for all $u\in H^\alpha$, all $x\in\HH^\alpha$ and all $N\in\N$, one has
\begin{equation}\label{eq:errorPN}
\|P_Nu-u\|_{H}\le \lambda_N^{-\frac{\alpha}{2}}\|u\|_{H^\alpha},\quad
\|\IP_Nx-x\|_{\HH}\le \lambda_N^{-\frac{\alpha}{2}}\|x\|_{\HH^\alpha}.
\end{equation}
The proof of~\eqref{eq:errorPN} is elementary, however it is provided below for completeness.

Owing to the definitions of operators $P_N$, $\IP_N$ and noting that the sequence $\bigl(\lambda_n\bigr)_{n\in\N}$ is non-decreasing, for all $u\in H^\alpha$, one has
\begin{align*}
\|P_Nu-u\|_{H}^2=\sum_{n=N+1}^{\infty}\langle u,e_n\rangle_H^2\le \lambda_N^{-\alpha}\sum_{n=N+1}^{\infty}\lambda_n^{\alpha}\langle u,e_n\rangle_H^2\le \lambda_N^{-\alpha}\|u\|_{H^\alpha}^2.
\end{align*}
Similarly, for all $x\in\HH^\alpha$, one has
\begin{align*}
\|\IP_Nx-x\|_{\HH}^2&=\Bigl\|\Bigl(\sum_{n=N+1}^{\infty}\langle u,e_n\rangle_H e_n,\sum_{n=N+1}^{\infty}\langle v,e_n\rangle_H e_n\Bigr)\Bigr\|_{\HH}^2=\sum_{n=N+1}^{\infty}\langle u,e_n\rangle_H^2+\sum_{n=N+1}^{\infty}\lambda_n^{-1}\langle v,e_n\rangle_{H}^2\\
&\le \lambda_N^{-\alpha}\Bigl(\sum_{n=N+1}^{\infty}\lambda_n^{\alpha}\langle u,e_n\rangle_H^2+\sum_{n=N+1}^{\infty}\lambda_n^{\alpha-1}\langle v,e_n\rangle_{H}^2\Bigr)\le \lambda_N^{-\alpha}(\|u\|_{H^\alpha}^2+\|v\|_{H^{\alpha-1}}^2)= \lambda_N^{-\alpha}\|x\|_{\HH^\alpha}^2.
\end{align*}

Note that combining~\eqref{eq:boundPN} and~\eqref{eq:errorPN}, for all $u\in H^\alpha$, $x\in\HH^\alpha$, and $K\ge N$, one obtains
\begin{equation}\label{eq:errorPNK}
\|P_Nu-P_Ku\|_{H}\le \lambda_N^{-\frac{\alpha}{2}}\|u\|_{H^\alpha},\quad
\|\IP_Nx-\IP_Kx\|_{\HH}\le \lambda_N^{-\frac{\alpha}{2}}\|x\|_{\HH^\alpha}.
\end{equation}

In addition, one has the following inverse inequalities: for all $\alpha\in[0,\infty)$, for all $u\in H^{-\alpha}$, all $x\in\HH^{-\alpha}$, and all $N\in\N$, one has
\begin{equation}\label{eq:inverseinequality}
\|P_Nu\|_H\le \lambda_N^{\frac{\alpha}{2}}\|u\|_{H^{-\alpha}},\quad \|\IP_Nx\|_{\HH}\le \lambda_N^{\frac{\alpha}{2}}\|x\|_{\HH^{-\alpha}}.
\end{equation}

It is worth mentioning that the linear operators $\IL$ and $e^{-t\IL}$, for all $t\ge 0$, commute with the orthogonal projection operator $P_N$ for all $N\in\N$. Similarly, for all $\epsilon\in(0,1)$, the linear operators $A_\epsilon$ and $e^{t A_\epsilon}$, for all $t\ge 0$, commute with the orthogonal projection operator $\IP_N$ for all $N\in\N$. Moreover, note that the covariance operator $Q$ of the Wiener process $\bigl(W^Q(t)\bigr)_{t\ge 0}$ given by~\eqref{eq:WQ} also commutes with the orthogonal projection operator $P_N$ for all $N\in\N$. As a result, for all $t\ge 0$ and $N\in\N$, one has
\begin{align*}
P_NW^{Q}(t)&=\sum_{n=1}^{N}\sqrt{q_n}\beta_n(t)e_n,\\
\IP_N\IW^{Q}(t)&=\sum_{n=1}^{N}\sqrt{q_n}\beta_n(t)\bigl(0,e_n\bigr).
\end{align*}

Finally, for all $N\in\N$, introduce the auxiliary mappings $f_N:H_N\to H_N$ and $F_N:\HH_N\to\HH_N$ defined by
\begin{align*}
f_N(u)&=P_Nf(u),\quad\forall~u\in H_N,\\
F_N(x)&=\IP_NF(x)=(0,f_N(u)),\quad\forall~x=(u,v)\in \HH_N.
\end{align*}

\subsubsection{Malliavin calculus}\label{sec:Malliavin}

The proof of the weak error estimates requires Malliavin calculus techniques. We thus recall basic definitions and properties from Malliavin calculus. For a comprehensive presentation we refer to the classical monographs~\cite{D.2006The,DE2018}.

For any $\vartheta\in L^2((0,T);H)$, define 
\[
W(\vartheta)=\int_{0}^{T}\langle \vartheta(s),dW(s)\rangle_H,
\]
where $\bigl(W(t)\bigr)_{t\ge 0}$ is the cylindrical Wiener process given by~\eqref{eq:W}.

Let $\mathcal{S}(\R)$ be the set of smooth real-valued cylindrical random variables defined as
\[
\mathcal{S}(\R)=\left\{{G}=g\bigl(W(\vartheta_1),\ldots,W(\vartheta_L));~{g}\in \mathcal{C}^\infty_0(\R^{L};\R),~\vartheta_\ell\in L^2((0,T);H), 1\le \ell\le L,~L\in\N \right\}
\]
For any ${G}=g\bigl(W(\vartheta_1),\ldots,W(\vartheta_L)\bigr)\in \mathcal{S}(\R)$, its Malliavin derivative $\mathcal{D}G$ is defined as
\[
\mathcal{D}_s{G}\!=\!\sum_{\ell=1}^L \partial_{\ell} {g}(W(\vartheta_1),\ldots,W(\vartheta_L))\vartheta_\ell(s),\quad \forall~s\in(0,T).
\]
The process $\bigl(\mathcal{D}_s{G}\bigr)_{s\in(0,T)}$ takes values in $L^2((0,T);H)$. For all $h\in H$ and all $s\in(0,T)$, set ${D}_s^h\mathcal{G}=\langle {D}_s\mathcal{G},h\rangle_{H}$.

For any ${G}\in\mathcal{S}(\R)$, set
\[
\|{G}\|_{\mathbb{D}^{1,2}(\R)}^{2}=\E[|{G}|^2]+\int_{0}^{T}\E[\|\mathcal{D}_tG\|_H^2]dt.
\]
The definition of the Malliavin derivative can be extended for $G\in\mathbb{D}^{1,2}(\R)$, where $\mathbb{D}^{1,2}(\R)$ is the closure of $\mathcal{S}(\R)$ with respect to the norm $\|\cdot\|_{\mathbb{D}^{1,2}(\R)}$.

For all $N\in\N$, the definition of the Malliavin derivative is next extended to $H_N$-valued random variables. Let
\[
\mathbb{D}^{1,2}(H_N)=\left\{\mathcal{G}=\sum_{n=1}^{N}{G}_ne_n;~{G}_n\in\mathbb{D}^{1,2}(\R), 1\le n\le N \right\}.
\]

For all $\mathcal{G}=\sum_{n=1}^{N}{G}_ne_n\in\mathbb{D}^{1,2}(H_N)$, set
\[
\mathcal{D}_t^h\mathcal{G}=\sum_{n=1}^{N}\mathcal{D}_t^h{G}_ne_n,\quad \forall~t\in(0,T),~h\in H.
\]

In the sequel, the following properties of the Malliavin derivative are used. First one has the chain rule: if $\varphi\in \mathcal{C}_b^1(H_N;\R)$, for all $\mathcal{G}\in\mathbb{D}^{1,2}(H_N)$, one has $\varphi(\mathcal{G})\in\mathbb{D}^{1,2}(\R)$ and
\begin{equation}\label{eq:chainrule}
\mathcal{D}_t^h\bigl(\varphi(\mathcal{G})\bigr)=D\varphi(\mathcal{G}).\mathcal{D}_t^h\mathcal{G},\quad\forall~t\in(0,T),~h\in H.
\end{equation}

Let $\mathcal{L}(H,H_N)$ be the set of bounded linear operators from $H$ to $H_N$. Assume that $\bigl(\Theta(t)\bigr)_{t\in[0,T]}$ is a predictable square-integrable process with values in $\mathcal{L}(H,H_N)$. First, the Malliavin derivative of the It\^o integral $\int_0^t\Theta(r)dW(r)$ satisfies
\begin{align}\label{eq:stochasticintegral-W}
\mathcal{D}_s^h\int_0^t\Theta(r)dW(r)=\int_s^t\mathcal{D}_s^h\Theta(r)dW(r)+\Theta(s)h,\qquad \forall t\ge s.
\end{align}
Finally, one has the following integration by parts formula: if $\mathcal{G}\in \mathbb{D}^{1,2}(H_N)$, $\varphi\in\mathcal{C}_b^2(H_N;\R)$, then one has
\begin{equation}\label{prop:Malliavinderivative-W}
\E\Big[\Big\langle D\varphi(\mathcal{G}), \int_0^T\Theta(t)dW(t)\Big\rangle_H\Big]=\sum_{n\in\N}\E\Big[\int_0^T D^2\varphi(\mathcal{G}).\big(\mathcal D^{e_n}_t\mathcal{G},  \Theta(t)e_n\big)\,dt\Big].
\end{equation}

The properties above can be adapted when the $Q$-Wiener process $\bigl(W^Q(t)\bigr)_{t\ge 0}$ is considered, writing $W^Q(t)=Q^{\frac12}W(t)$.

\subsection{Stochastic partial differential equations}\label{sec:SPDE}

The aim of this article is to study the behavior when $\epsilon\to 0$ of the solution of the stochastic semilinear damped wave equation driven by additive noise
\begin{equation}\label{eq:sdwe}
	\left\lbrace
	\begin{aligned}
	      &du^\epsilon (t)=\frac{v^\epsilon(t)}{\epsilon}dt,~t\ge 0,\\
	      &dv^\epsilon(t)=\Bigl(-\frac{\IL u^\epsilon(t)}{\epsilon}-\frac{v^\epsilon(t)}{\epsilon^2}\Bigr)dt+\frac{f(u^\epsilon(t))}{\epsilon}dt+\frac{1}{\epsilon}dW^Q(t),~t\ge 0,\\
	      &u^\epsilon(0)=u_0^\epsilon,\quad v^\epsilon(0)=v_0^\epsilon.
   \end{aligned}
   \right.
\end{equation}
The linear operator $\IL$, the mapping $f$ and the Wiener process $\bigl(W^Q(t)\bigr)_{t\ge 0}$ are described in Section~\ref{sec:setting}. For all $\epsilon\in(0,1)$, the unknowns $(u^\epsilon(t))_{t\ge 0}$ and $(v^\epsilon(t))_{t\ge 0}$ are $H$ and $H^{-1}$-valued continuous stochastic processes, adapted to the filtration $\bigl(\mathcal{F}_t\bigr)_{t\ge 0}$. The initial values $u_0^\epsilon\in H$ and $v_0^\epsilon\in H^{-1}$ are $\mathcal{F}_0$-measurable random variables and are thus independent of the Wiener process $\bigl(W^Q(t)\bigr)_{t\ge 0}$. They are allowed to depend on $\epsilon$, more precise conditions are given below.

Employing the notation introduced in Section~\ref{sec:setting}, the stochastic evolution system~\eqref{eq:sdwe} can be equivalently rewritten as the stochastic evolution equation for the unknown $X^\epsilon(t)=(u^\epsilon(t),v^\epsilon(t))\in\HH$
\begin{equation}\label{eq:sdwe1}
\left\lbrace
     \begin{aligned}
&dX^\epsilon(t)=A_\epsilon X^\epsilon(t)dt+ \frac{1}{\epsilon}F(X^\epsilon(t))dt+\frac{1}{\epsilon}d\IW^Q(t),~t\ge 0,\\
&X^\epsilon(0)=x_0^\epsilon,  
     \end{aligned}
\right.
\end{equation}
with $A_\epsilon$ given by~\eqref{eq:Aepsilon}, $F$ given by~\eqref{eq:F} and the Wiener process $\bigl(\IW^Q(t)\bigr)_{t\ge 0}$ given by~\eqref{eq:IWQ}. For all $\epsilon\in(0,1)$, the initial value is given by $x_0^\epsilon=(u_0^\epsilon,v_0^\epsilon)\in\HH$. For all $\epsilon\in(0,1)$, the unknown $(X^\epsilon(t))_{t\ge 0}$ is a $\HH$-valued continuous stochastic process, adapted to the filtration $\bigl(\mathcal{F}_t\bigr)_{t\ge 0}$.

Recall that a $\HH$-valued stochastic process $\bigl(X^\epsilon(t)\bigr)_{t\ge 0}$ (adapted to the filtration $\bigl(\mathcal{F}_t\bigr)_{t\ge 0}$ and with continuous trajectories) is a mild solution of the stochastic evolution equation~\eqref{eq:sdwe1} with initial value $X^\epsilon(0)=x^\epsilon_0$, if for all $t\ge 0$ almost surely one has
\begin{equation}\label{eq:mild-sdwe}
X^\epsilon(t)=e^{tA_{\epsilon}}x_0^\epsilon+\frac{1}{\epsilon}\int_0^te^{(t-s)A_{\epsilon}}F(X^\epsilon(s))\,ds+\frac{1}{\epsilon}\int_0^te^{(t-s)A_{\epsilon}}\,d\IW^Q(s),
\end{equation}
where $(e^{tA_{\epsilon}})_{t\ge0}$ is the semigroup generated by the linear operator $A_\epsilon$, given by~\eqref{eq:expAeps}, see Section~\ref{sec:semigroups}.

For all $\epsilon\in(0,1)$ and all $t\ge 0$, set
\begin{equation}\label{eq:IZ}
\IZ^\epsilon(t)=\frac{1}{\epsilon}\int_0^te^{(t-s)A_{\epsilon}}\,d\IW^Q(s).
\end{equation}
The process $\bigl(\IZ^\epsilon(t)\bigr)_{t\ge 0}$ is called the stochastic convolution. Under Assumption~\ref{ass:noise}, it will be seen below that this process is a Gaussian continuous $\HH$-valued stochastic process. In addition, the mapping $F:\HH\to\HH$ is globally Lipschitz continuous, therefore by a standard fixed point procedure, it is straightforward to prove that for all $\epsilon\in(0,1)$, there exists a unique global mild solution~\eqref{eq:mild-sdwe} of the stochastic evolution equation~\eqref{eq:sdwe1}. More details, in particular regularity properties and moment bounds, are provided in Section~\ref{sec:properties} below. It is worth mentioning that properties which are uniform with respect to $\epsilon\in(0,1)$ are required, and obtaining them requires to employ the results presented in Section~\ref{sec:aux}.

When $\epsilon\to 0$, assuming that $u_0^\epsilon\to u_0^0$ in $H$, then for all $t\ge 0$, $u^\epsilon(t)=\Pi_uX^\epsilon(t)$ converges to the solution $u^0(t)$ at time $t$ of the stochastic semilinear heat equation driven by additive noise
\begin{equation}\label{eq:she}
	\left\lbrace
	\begin{aligned}
		&du^0(t)=-\IL u^0(t)dt+f(u^0(t))dt+dW^Q(t),\\
		&u^0(0)=u_0^0.
	\end{aligned}
   \right.
\end{equation}
The objective of this article is to provide strong error estimates, see Section~\ref{sec:main} for a precise statement.

Recall that an $H$-valued stochastic process $\bigl(u^0(t)\bigr)_{t\ge 0}$ (adapted to the filtration $\bigl(\mathcal{F}_t\bigr)_{t\ge 0}$ and with continuous trajectories) is a mild solution of the stochastic evolution equation~\eqref{eq:she} with initial value $u^0(0)=u^0_0$, if for all $t\ge 0$ almost surely one has
\begin{equation}\label{eq:mild-she}
u^0(t)=e^{-t\IL}u_0^0+\int_0^te^{-(t-s)\IL}f(u^0(s))\,ds+\int_0^te^{-(t-s)\IL}\,dW^Q(s),
\end{equation}
where $(e^{-t\IL})_{t\ge0}$ is the semigroup generated by the linear operator $-\IL$, given by~\eqref{eq:expL}, see Section~\ref{sec:semigroups}.

For all $t\ge 0$, set
\begin{equation}\label{eq:Z}
Z(t)=\int_0^te^{-(t-s)\IL}\,dW^Q(s).
\end{equation}
The process $\bigl(Z(t)\bigr)_{t\ge 0}$ is called the stochastic convolution. Under Assumption~\ref{ass:noise}, it will be seen below that this process is a Gaussian continuous $H$-valued stochastic process. In addition, the mapping $f:H\to H$ is globally Lipschitz continuous (see Assumption~\ref{ass:fLip}), therefore by a standard fixed point procedure, it is straightforward to prove that there exists a unique global mild solution~\eqref{eq:mild-she} of the stochastic evolution equation~\eqref{eq:she}. We refer to Section~\ref{sec:properties} for additional details, in particular on regularity properties and moment bounds.

To conclude this section, let us apply the spectral Galerkin approximation procedure introduced in Section~\ref{sec:Galerkin}. Let $N\in\N$. On the one hand, from~\eqref{eq:sdwe1} one obtains the following stochastic evolution equation
\begin{equation}\label{eq:sdwe1-galerkin}
\left\lbrace
\begin{aligned}
	dX^{\epsilon,N}(t)&=\big(A_{\epsilon}X^{\epsilon,N}(t)+\frac{1}{\epsilon}F_N(X^{\epsilon,N}(t))\big)dt+\frac{1}{\epsilon}\mathcal{P}_Nd\IW^Q(t),~t\ge 0,\\
X^{\epsilon,N}(0)&=\mathcal{P}_Nx_0^\epsilon,
\end{aligned}
\right.
\end{equation}
with the unknown $t\ge 0\mapsto X^{\epsilon,N}(t)$ taking value in $\HH_N$, where the orthogonal projection operator $\IP_N$ is given by~\eqref{eq:IPN}. Equivalently, the stochastic evolution equation~\eqref{eq:sdwe1-galerkin} can be written as the stochastic evolution system for $(u^{\epsilon,N}(t),v^{\epsilon,N}(t))=(\Pi_uX^{\epsilon,N}(t),\Pi_vX^{\epsilon,N}(t))\in \HH_N$
\begin{equation}\label{eq:sdwe-galerkin}
\left\lbrace
\begin{aligned}
	&du^{\epsilon,N}(t)=\frac{v^{\epsilon,N}(t)}{\epsilon}dt,~t\ge 0,\\
	&dv^{\epsilon,N}(t)=\Bigl(-\frac{\IL u^{\epsilon,N}(t)}{\epsilon}-\frac{v^{\epsilon,N}(t)}{\epsilon^2}\Bigr)dt+\frac{f_N(u^{\epsilon,N}(t))}{\epsilon}dt+\frac{1}{\epsilon}P_NdW^Q(t),~t\ge 0,\\
	&u^{\epsilon,N}(0)=P_Nu_0^\epsilon,\quad v^{\epsilon,N}(0)=P_Nv_0^\epsilon,
\end{aligned}
\right.
\end{equation}
with the unknowns $t\ge 0\mapsto u^{\epsilon,N}(t)$ and $t\ge 0\mapsto v^{\epsilon,N}(t)$ taking value in $H_N$, where the orthogonal projection operator $P_N$ is given by~\eqref{eq:PN}. For all $N\in\N$ and $\epsilon\in(0,1)$, there exists a unique solution $\bigl(X^{\epsilon,N}(t)\bigr)_{t\ge0}$ of~\eqref{eq:sdwe1-galerkin}, which is given by the mild formulation
\begin{equation}\label{eq:mild-sdwe-galerkin} 
    X^{\epsilon,N}(t)=e^{tA_{\epsilon}}\IP_Nx^\epsilon_0+\frac{1}{\epsilon}\int_0^te^{(t-s)A_{\epsilon}}F_N(X^{\epsilon,N}(s))\,ds+\frac{1}{\epsilon}\int_0^te^{(t-s)A_{\epsilon}}\IP_N\,d\IW^Q(s),\quad\forall~t\ge 0.
\end{equation}

On the other hand, from~\eqref{eq:she} one obtains the following stochastic evolution equation
\begin{equation}\label{eq:she-galerkin}
\left\lbrace
\begin{aligned}
     &du^{0,N}(t)=-\IL u^{0,N}(t)dt+f_N(u^{0,N}(t))dt+P_NdW^Q(t),~t\ge 0,\\
     &u^{0,N}(0)=P_Nu_0^0,
\end{aligned}
\right.
\end{equation}
with the unknown $t\ge 0\mapsto u^{0,N}(t)$ taking values in $H_N$, where the orthogonal projection operator $P_N$ is given by~\eqref{eq:PN}. For all $N\in\N$, there exists a unique solution $\bigl(u^{0,N}(t)\bigr)_{t\ge 0}$ of~\eqref{eq:she-galerkin}, which is given by the mild formulation
\begin{equation}\label{eq:mild-she-galerkin}
    u^{0,N}(t)=e^{-t\IL}P_Nu^0_0+\int_0^t e^{-(t-s)\IL}f_N(u^{0,N}(s))\,ds+\int_0^t e^{-(t-s)\IL}P_N\,dW^Q(s),\quad\forall~t\ge 0.
\end{equation}

It is well-known that the spectral Galerkin approximation converges when $N\to\infty$, for instance in the following sense: for all $\epsilon\in(0,1)$ one has
\[
\underset{t\in[0,T]}\sup~\E[\|X^{\epsilon,N}(t)-X^{\epsilon}(t)\|_{\HH}^2]\underset{N\to\infty}\to 0,
\]
and
\[
\underset{t\in[0,T]}\sup~\E[\|u^{0,N}(t)-u^{0}(t)\|_{H}^2]\underset{N\to\infty}\to 0.
\]

\section{Main results}\label{sec:main}

The statements of the main results require to impose some  assumptions in the initial values $x_0^\epsilon=(u_0^\epsilon,v_0^\epsilon)$.

For the strong error estimates given in Theorem~\ref{theo:strong}, Assumption~\ref{ass:init-strong} is needed.
\begin{assum}\label{ass:init-strong}
For all $\alpha\in[0,\beta)$, one has $x_0^\epsilon\in \HH^\alpha$ and $u_0^0\in H^\alpha$ almost surely, and for all $p\in[1,\infty)$ there exists $M_{p,\alpha}\in(0,\infty)$ such that
\begin{align}
&\underset{\epsilon\in(0,1)}\sup~\E[\|x_0^\epsilon\|_{\HH^\alpha}^p]+\E[\|u_0^0\|_{H^\alpha}^p]\le M_{p,\alpha},\label{eq:init-strong-bound}\\
&\bigl(\E[\|u_0^\epsilon-u^0_0\|_{H}^p]\bigr)^{\frac{1}{p}}\le M_\alpha \epsilon^{\alpha},\quad\forall~\epsilon\in(0,1).\label{eq:init-strong-error}
\end{align}
\end{assum}

\begin{theo}\label{theo:strong}
Let Assumptions~\ref{ass:fLip},~\ref{ass:noise} and~\ref{ass:init-strong} be satisfied.

For all $p\in[1,\infty)$, $T\in(0,\infty)$ and $\alpha\in[0,\beta)$ (where $\beta\in(0,1]$ is given by the condition~\eqref{eq:condition_beta} from Assumption~\ref{ass:noise}), there exists $C_{p,\alpha}(T)\in(0,\infty)$ such that for all $\epsilon\in(0,1)$ one has  
\begin{equation}\label{eq:theo-strong}
\underset{0\le t\le T}\sup~\bigl(\E[\|u^{\epsilon}(t)-u^{0}(t)\|_H^{p}]\bigr)^{\frac{1}{p}}\le C_{p,\alpha}(T) \epsilon^\alpha.
\end{equation}
\end{theo}

For the weak error estimates given in Theorem~\ref{theo:weak}, Assumption~\ref{ass:init-weak} is needed.
\begin{assum}\label{ass:init-weak}
For all $\alpha\in[0,\beta)$, one has $x_0^\epsilon\in\HH^{2\alpha}$ and $u_0^0\in H^{2\alpha}$ almost surely, and there exists $M_\alpha\in(0,\infty)$ such that
\begin{align}
&\underset{\epsilon\in(0,1)}\sup~\E[\|x_0^\epsilon\|_{\HH^{2\alpha}}^3]+\E[\|u_0^0\|_{H^{2\alpha}}^3]\le M_\alpha,\label{eq:init-weak-bound}\\
&\E[\|u_0^\epsilon-u^0_0\|_{H}^2]\le M_\alpha \epsilon^{2\min(2\alpha,1)},\quad\forall~\epsilon\in(0,1).\label{eq:init-weak-error}
\end{align}
\end{assum}

\begin{theo}\label{theo:weak}
Let Assumptions~\ref{ass:fLip}, \ref{ass:mapping-f},~\ref{ass:noise} and \ref{ass:init-weak} be satisfied.

Then for all $T\in(0,\infty)$, $\alpha\in[0,\beta)$ (where $\beta\in(0,1]$ is given by the condition~\eqref{eq:condition_beta} from Assumption~\ref{ass:noise}), and all mapping $\varphi\colon H\to\R$ of class $\mathcal{C}^4$, with bounded derivatives up to order $4$, there exists $C_\alpha(T,\varphi)\in(0,\infty)$ such that for all $\epsilon\in(0,1)$ one has
\begin{equation}\label{eq:theo-weak}
\underset{0\le t\le T}\sup~\bigl|\E[\varphi(u^{\epsilon}(t))]-\E[\varphi(u^{0}(t))]\bigr|\le C_{\alpha}(T,\varphi)\epsilon^{\min(2\alpha,1)}. 
\end{equation}
\end{theo}

Comparing the strong and weak error estimates from Theorems~\ref{theo:strong} and~\ref{theo:weak}, one observes that in general the weak order of convergence is strictly larger than the strong order of convergence. The only case where the two orders of convergence coincide is the trace-class noise case, when $\beta=1$: in that case both orders of convergence are equal to $1$, which coincides with the strong and weak orders of convergence for finite dimensional stochastic differential equations. The weak order of convergence is not always the double of the strong order of convergence since the orders of convergence cannot be larger than $1$. This is why the weak order of convergence is equal to $\min(2\beta,1)$ in~\eqref{eq:theo-weak}. If $\beta\in(0,1/2)$, then the strong and weak orders of convergence are $\beta$ and $2\beta$ respectively. This encompasses in particular the space-time white noise case in dimension $d=1$, for which $\beta=1/2$, thus for which the strong and weak orders of convergence are $1/2$ and $1$ respectively.

The proofs of Theorems~\ref{theo:strong} and~\ref{theo:weak} are postponed to Sections~\ref{sec:proof-theo-strong} and~\ref{sec:proof-theo-weak} respectively. Sections~\ref{sec:aux} and~\ref{sec:properties} provide auxiliary results which are used in both proofs.

\section{Auxiliary results}\label{sec:aux}

This section presents several auxiliary results which play important roles in the error analysis.

\subsection{Properties of the heat semigroup}\label{sec:expL}

Let us first recall standard properties concerning the heat semigroup $\bigl(e^{-t\IL}\bigr)_{t\ge 0}$ defined by~\eqref{eq:expL} in Section~\ref{sec:semigroups}.

\begin{lemma}\label{lem:expL}
For all $\alpha\in\R$, all $u\in H^\alpha$ and all $t\ge 0$, one has
\begin{equation}\label{eq:lem_expL-bound}
\|e^{-t\IL}u\|_{H^\alpha}\le e^{-t\lambda_1}\|u\|_{H^\alpha}.
\end{equation}
In addition, the following smoothing property is satisfied: for all $\alpha\ge 0$, there exists $C_\alpha\in(0,\infty)$ such that for all $t>0$ and all $u\in H$ one has
\begin{equation}\label{eq:lem_expL-smoothing}
\|e^{-t\IL}u\|_{H^\alpha}\le C_\alpha t^{-\frac{\alpha}{2}}\|u\|_H.
\end{equation}
\end{lemma}

The results stated in Lemma~\ref{lem:expL} are standard and the elements of proof below are given for completeness.

\begin{proof}[Proof of Lemma~\ref{lem:expL}]
First, owing to the definitions~\eqref{eq:expL} of $e^{-t\IL}u$ and of the norm $\|\cdot\|_{H^\alpha}$ on $H^\alpha$ (see Section~\ref{sec:notation}), one obtains
\[
\|e^{-t\IL}u\|_{H^\alpha}^2=\sum_{n=1}^\infty \lambda_n^{\alpha}e^{-2t\lambda_n}\langle u, e_n\rangle_H^2\le e^{-2t\lambda_1}\sum_{n=1}^\infty \lambda_n^{\alpha}\langle u, e_n\rangle_H^2=e^{-2t\lambda_1}\|u\|_{H^\alpha}^2.
\]
This concludes the proof of~\eqref{eq:lem_expL-bound}. In order to prove~\eqref{eq:lem_expL-smoothing}, note that for all $\alpha\ge 0$ the mapping $x\ge 0\mapsto x^\alpha e^{-x}$ is bounded. As a result, there exists $C_\alpha\in(0,\infty)$ such that for all $u\in H$ and all $t>0$ one has
\[
\|e^{-t\IL}u\|_{H^\alpha}^2=t^{-\alpha}\sum_{n=1}^\infty t^{\alpha}\lambda_n^{\alpha}e^{-2t\lambda_n}\langle u, e_n\rangle_H^2\le C_\alpha^2 t^{-\alpha}\sum_{n=1}^\infty\langle u, e_n\rangle_H^2=C_\alpha^2 t^{-\alpha}\|u\|_H^2.
\]
This concludes the proof of~\eqref{eq:lem_expL-smoothing}. The proof of Lemma~\ref{lem:expL} is thus completed.
\end{proof}

\subsection{Properties of the damped wave semigroup}\label{sec:expAeps}

The objective of this section is to study properties of the damped wave semigroup $\bigl(e^{tA_\epsilon}\bigr)_{t\ge 0}$ given by~\eqref{eq:expAeps} in Section~\ref{sec:semigroups}. The main difficulty is to obtain upper bounds which are either uniform with respect to the parameter $\epsilon\in(0,1)$, or which contain terms which vanish when $\epsilon\to 0$. We obtain variants of the smoothing property~\eqref{eq:lem_expL-smoothing}.

First, the semigroup satisfies the following bounds which are uniform with respect to $\epsilon\in(0,1)$ and $t\in[0,\infty)$.
\begin{lemma}\label{lem:expAeps-bound}
For all $\alpha\in\R$ and all $x\in\HH^\alpha$, one has
\begin{equation}\label{eq:lem_expAeps-bound}
\underset{\epsilon\in(0,1)}\sup~\underset{t\ge 0}\sup~\|e^{tA_\epsilon}x\|_{\HH^\alpha}\le \|x\|_{\HH^\alpha}. 
\end{equation}
\end{lemma}

\begin{proof}[Proof of Lemma~\ref{lem:expAeps-bound}]
Recall that for all $x=(u,v)\in\HH^\alpha$ and all $t\ge 0$ one has
\[
e^{tA_\epsilon}x=\sum_{n\in\N}\bigl(f^\epsilon_n(t)e_n, g^\epsilon_n(t)e_n\bigr)
\]
where for all $n\in\N$ the mapping $t\ge 0\mapsto (f^\epsilon_n(t),g^\epsilon_n(t))$ solves~\eqref{eq:fngn}.

Let $n\in\N$. Observe that $f\mapsto f_n^\epsilon(t)$ solves the second-order linear differential equation
\begin{equation}\label{eq:f}
\epsilon^2(f^\epsilon_n)''(t)+(f^\epsilon_n)'(t)+\lambda_n f^\epsilon_n(t)=0.
\end{equation}
It then follows that one has
\[
\epsilon^2\frac{d}{dt}|(f^\epsilon_n(t))'|^2+2|(f^\epsilon_n(t))'|^2+\lambda_n\frac{d}{dt}|f^\epsilon_n(t)|^2=2\big\langle\epsilon^2(f^\epsilon_n(t))''+(f^\epsilon_n(t))'+\lambda_n f^\epsilon_n(t), (f^\epsilon_n(t))'\big\rangle=0.
\]
By integration one obtains for all $t\ge 0$
\[
\epsilon^2|(f^\epsilon_n(t))'|^2+2\int_0^t|(f^\epsilon_n(s))'|^2\,ds+\lambda_n|f^\epsilon_n(t)|^2=\epsilon^2|(f^\epsilon_n(0))'|^2+\lambda_n|f^\epsilon_n(0)|^2.
\]
For all $t\ge 0$ one has $g^\epsilon_n(t)=\epsilon(f^\epsilon_n(t))'$ owing to~\eqref{eq:fngn}. Therefore multiplying both sides by $\lambda_n^{\alpha-1}$ and omitting the integral term in the left-hand side, the equality above yields the following inequality: for all $n\in\N$ and all $t\ge 0$, one has
\[
\lambda_n^{\alpha-1}|g^\epsilon_n(t)|^2+\lambda_n^\alpha|f^\epsilon_n(t)|^2\le \lambda_n^{\alpha-1}|g^\epsilon_n(0)|^2+\lambda_n^\alpha|f^\epsilon_n(0)|^2.
\]
Finally, using the identity
\begin{equation}
\|e^{tA_\epsilon}x\|_{\HH^\alpha}^2=\sum_{n\in\N}\lambda_n^\alpha|f^\epsilon_n(t)|^2+\sum_{n\in\N}\lambda_n^{\alpha-1}|g^\epsilon_n(t)|^2
\end{equation}
and the inequality above gives
\[
\|e^{tA_\epsilon}x\|_{\HH^\alpha}^2\le \|x\|_{\HH^\alpha}^2.
\]
This concludes the proof of Lemma~\ref{lem:expAeps-bound}.
\end{proof}

Let us now state two further auxiliary results, which are variants of the smoothing property~\eqref{eq:lem_expL-smoothing} for the heat semigroup, but require more attention. Those results play an important role in the analysis below.

\begin{lemma}\label{lem:expAeps-smoothing1}
For all $\delta\in[0,\frac12]$ and $\rho\in[2\delta,1]$, there exists $C_{\delta,\rho}\in(0,\infty)$ such that, for all $\alpha\in\R$, all $\epsilon\in(0,1)$, all $t>0$ and all $u\in H^{\alpha+\rho-2\delta-1}$, 
\begin{equation}\label{eq:lem_expAeps-smoothing1}
\|e^{tA_\epsilon}(0,u)\|_{\HH^\alpha}\le C_{\delta,\rho}t^{-\delta}\epsilon^{\rho}\Bigl(1+\epsilon^{2\delta-\rho}e^{-\frac{t}{2\epsilon^2}}\Bigr)\|u\|_{H^{\alpha+\rho-2\delta-1}}.
\end{equation}
\end{lemma}

As stated in Proposition~\ref{propo:sdwe-momentbounds} below, the result~\eqref{eq:lem_expAeps-smoothing1} will be used to analyze the moment bounds of stochastic convolution $\mathcal{Z}^{\epsilon}$ in the moment estimates of $X^{\epsilon}$, with $\delta=\frac{1-\delta_0}{2}, \rho=1$, where $\delta_0\in(0,\beta-\alpha]$ is arbitrarily small. In addition to this, for each application of Lemma~\ref{lem:expAeps-smoothing1}, the appropriate values of the parameters $\delta, \rho$ will be specified.

\begin{lemma}\label{lem:expAeps-smoothing2}
For all $\alpha,\alpha_1, \alpha_2\in[0,1]$ and all $\delta\in[0,\frac{\alpha}{2}]$, there exist $C_{\alpha,\delta},C_{\alpha_1},C_{\alpha_2}\in(0,\infty)$ such that, for all $\epsilon\in(0,1)$ and all $t>0$, one has
\begin{align}
&\|\Pi_u e^{tA_\epsilon}(0,v)\|_{H}\le C_{\alpha,\delta}\epsilon^{\alpha}t^{-\delta}\|v\|_{H^{\alpha-1-2\delta}}, \quad \forall~ v\in H^{\alpha-1-2\delta},\label{eq:lem_expAeps-smoothing2-1}\\
&\|\Pi_u e^{tA_\epsilon}(u,0)\|_{H}\le C_{\alpha_1}t^{-\frac{\alpha_1}{2}}\epsilon^{\alpha_1}\|u\|_H+C_{\alpha_2}t^{-\frac{\alpha_2}{2}}\|u\|_{H^{-\alpha_2}}, \quad \forall~u\in H.\label{eq:lem_expAeps-smoothing2-2}
\end{align}
\end{lemma}

The proofs of Lemmas~\ref{lem:expAeps-smoothing1} and~\ref{lem:expAeps-smoothing2} are based on a refined analysis of properties of the solutions of the systems~\eqref{eq:fngn}, depending on the values of $\epsilon\in(0,1)$ and $n\in\N$. For arbitrary $n\in\N$ and $\epsilon\in(0,1)$, let $t\ge 0\mapsto (f_n^{\epsilon,0,1}(t),g_n^{\epsilon,0,1}(t))$ and $t\ge 0\mapsto (f_n^{\epsilon,1,0}(t),g_n^{\epsilon,1,0}(t))$ denote the solutions of the linear two-dimensional system of ordinary differential equations~\eqref{eq:fngn} with initial values $(f_n^{\epsilon,0,1}(0),g_n^{\epsilon,0,1}(0))=(0,1)$ and $(f_n^{\epsilon,1,0}(0),g_n^{\epsilon,1,0}(0))=(1,0)$ respectively.

\begin{lemma}\label{lem:fngn}
There exists $C\in(0,\infty)$ such that the following holds.
\begin{itemize}
\item If $1-4\lambda_n\epsilon^2<0$, then for all $t\ge 0$,
   \begin{align}
&     |f_n^{\epsilon,0,1}(t)|\le C\lambda_n^{-1/2}e^{-\frac{t}{4\epsilon^2}},~|g_n^{\epsilon,0,1}(t)|\le Ce^{-\frac{t}{4\epsilon^2}},\label{eq:lem_fngn01-1}\\
&     |f_n^{\epsilon,1,0}(t)|\le Ce^{-\frac{t}{4\epsilon^2}}.\label{eq:lem_fngn10-1}
     \end{align}
\item If $1-4\lambda_n\epsilon^2\ge 0$, then for all $t\ge 0$ one has
  \begin{align}
&	|f_n^{\epsilon,0,1}(t)|\le C\epsilon e^{-\lambda_n t},~|g_n^{\epsilon,0,1}(t)|\le Ce^{-\frac{t}{\epsilon^2}}+C\lambda_n\epsilon^2e^{-\lambda_n t},\label{eq:lem_fngn01-2}\\
&	|f_n^{\epsilon,1,0}(t)|\le Ce^{-t\lambda_n}.\label{eq:lem_fngn10-2}
  \end{align}
\end{itemize}
\end{lemma}

\subsubsection{Proof of Lemma~\ref{lem:fngn}}

A few supplementary ingredients are required for the proof of Lemma~\ref{lem:fngn}.

First, one can provide the expressions of $f_n^\epsilon(t)$, for all $n\in\N$, all $\epsilon\in(0,1)$, and for all $t\ge 0$, for any initial value $(f_n^\epsilon(0),g_n^\epsilon(0))$. The expressions depend on the sign of $1-4\lambda_n\epsilon^2$. In fact, using~\eqref{eq:f} it is straightforward to check that the mapping $t\ge 0\mapsto \tilde{f}_n^\epsilon(t)=e^{\frac{t}{2\epsilon^2}}f_n^\epsilon(t)$ is solution of the second order linear differential equation
\[
(\tilde{f}^\epsilon_n)''(t)+\frac{4\epsilon^2\lambda_n-1}{4\epsilon^4}\tilde{f}^\epsilon_n(t)=0,
\]
with initial values $\tilde{f}^\epsilon_n(0)=f_n^\epsilon(0)$ and $(\tilde{f}_n^\epsilon)'(0)=(f_n^\epsilon)'(0)+\frac{1}{2\epsilon^2}f_n^\epsilon(0)=\frac{2\epsilon g_n^\epsilon(0)+f_n^\epsilon(0)}{2\epsilon^2}$. Solving the second order linear differential equation above yields the following expressions for $f_n^\epsilon(t)=e^{-\frac{t}{2\epsilon^2}}\tilde{f}_n^\epsilon(t)$ and for $g_n^\epsilon(t)=\epsilon \bigl(f_n^\epsilon\bigr)'(t)$, see~\eqref{eq:fngn}.

\begin{itemize}
\item If $1-4\lambda_n\epsilon^2<0$, then for all $t\ge 0$ one has
\begin{equation}\label{eq:fn1}
\left\{
\begin{aligned}
f^\epsilon_n(t)&=e^{-\frac{t}{2\epsilon^2}}\Big[f^\epsilon_n(0)\cos\bigl(\frac{\sqrt{4\lambda_n\epsilon^2-1}}{2\epsilon^2}t\bigr)+ \frac{2\epsilon g^\epsilon_n(0)+f^\epsilon_n(0)}{\sqrt{4\lambda_n\epsilon^2-1}}\sin\bigl(\frac{\sqrt{4\lambda_n\epsilon^2-1}}{2\epsilon^2}t\big)\Bigr],\\
g^\epsilon_n(t)&=e^{-\frac{t}{2\epsilon^2}}\Big[g^\epsilon_n(0)\cos\bigl(\frac{\sqrt{4\lambda_n\epsilon^2-1}}{2\epsilon^2}t\bigr)-\frac{2\lambda_n\epsilon f^\epsilon_n(0)+g^\epsilon_n(0)}{\sqrt{4\lambda_n\epsilon^2-1}}\sin\bigl(\frac{\sqrt{4\lambda_n\epsilon^2-1}}{2\epsilon^2}t\big)\Big].
\end{aligned}
\right.
\end{equation}

\item If $1-4\lambda_n\epsilon^2=0$, then for all $t\ge 0$ one has
\begin{equation}\label{eq:fn2}
\left\{
\begin{aligned}
f^\epsilon_n(t)&=e^{-\frac{t}{2\epsilon^2}}\Big[f^\epsilon_n(0)\Big(1+\frac{t}{2\epsilon^2}\Big)+\frac{t}{\epsilon}g^\epsilon_n(0)\Big],\\
g^\epsilon_n(t)&=e^{-\frac{t}{2\epsilon^2}}\Big[g^\epsilon_n(0)\Big(1-\frac{t}{2\epsilon^2}\Big)-\frac{t}{4\epsilon^3}f^\epsilon_n(0)\Big].
\end{aligned}
\right.
\end{equation}

\item If $1-4\lambda_n\epsilon^2>0$, then for all $t\ge 0$ one has
\begin{equation}\label{eq:fn3}
\left\{
\begin{aligned}
f^\epsilon_n(t)&=e^{-\frac{t}{2\epsilon^2}}\Big[f^\epsilon_n(0)\cosh\bigl(\frac{\sqrt{1-4\lambda_n\epsilon^2}}{2\epsilon^2}t\bigr)+ \frac{2\epsilon g^\epsilon_n(0)+f^\epsilon_n(0)}{\sqrt{1-4\lambda_n\epsilon^2}}\sinh\bigl(\frac{\sqrt{1-4\lambda_n\epsilon^2}}{2\epsilon^2}t\big)\Bigr],\\
g^\epsilon_n(t)&=e^{-\frac{t}{2\epsilon^2}}\Big[g^\epsilon_n(0)\cosh\bigl(\frac{\sqrt{1-4\lambda_n\epsilon^2}}{2\epsilon^2}t\bigr)-\frac{2\lambda_n\epsilon f^\epsilon_n(0)+ g^\epsilon_n(0)}{\sqrt{1-4\lambda_n\epsilon^2}}\sinh\bigl(\frac{\sqrt{1-4\lambda_n\epsilon^2}}{2\epsilon^2}t\big)\Big].
\end{aligned}
\right.
\end{equation}
\end{itemize}

To prove some inequalities from Lemma~\ref{lem:fngn}, the following expression of $g_n^\epsilon(t)$ is employed: for all $n\in\N$, $\epsilon\in(0,1)$ and $t\ge 0$, one has
\begin{equation}\label{eq:exactsolution-g}
	g^\epsilon_n(t)=e^{-\frac{t}{\epsilon^2}}g^\epsilon_n(0)-\int_0^t\frac{\lambda_n}{\epsilon}e^{-\frac{t-s}{\epsilon^2}}f^\epsilon_n(s)\,ds.
\end{equation}
The proof of the identity~\eqref{eq:exactsolution-g} is straightforward: from~\eqref{eq:fngn} one obtains
\[
\frac{d}{dt}\Bigl(\epsilon e^{\frac{t}{\epsilon^ 2}}g_n^\epsilon(t)\Bigr)=\frac{d}{dt}\Bigl(\epsilon^2 e^{\frac{t}{\epsilon^ 2}}(f_n^\epsilon)'(t)\Bigr)=-\lambda_n e^{\frac{t}{\epsilon^2}}f_n^\epsilon(t),
\]
which yields by integration for all $t\ge 0$
\[
\epsilon e^{\frac{t}{\epsilon^2}}g^\epsilon_n(t)=\epsilon g^\epsilon_n(0)-\int_0^t \lambda_n e^{\frac{s}{\epsilon^2}}f^\epsilon_n(s)\,ds,
\]
and the identity~\eqref{eq:exactsolution-g} then follows.

Finally, for all $n\in\N$ and $\epsilon\in(0,1)$, let us define the auxiliary parameters
\begin{align*}
&\theta_{n,\epsilon}=\min\Bigl(\frac{1}{4\epsilon^2},\lambda_n\Bigr)=
\begin{cases}
\frac{1}{4\epsilon^2}, \text{if}~1-4\lambda_n\epsilon^2\le 0\\
\lambda_n, \text{if}~1-4\lambda_n\epsilon^2\ge 0
\end{cases},\\
&\Theta_{n,\epsilon}=\lambda_n-\theta_{n,\epsilon}+\epsilon^2\theta_{n,\epsilon}^2.
\end{align*}
In addition, for all $t\ge 0$ set
\[
\tilde{w}_{n,\epsilon}(t)=e^{\theta_{n,\epsilon}t}f^\epsilon_n(t).
\]
By elementary computations, using~\eqref{eq:fngn} it is straightforward to check that the auxiliary mapping $\tilde{w}_{n,\epsilon}$ is solution of the linear second-order ordinary differential equation
\begin{equation}\label{eq:tildew}
\tilde{w}_{n,\epsilon}''(t)+\big(\frac{1}{\epsilon^2}-2\theta_{n,\epsilon}\big)\tilde{w}_{n,\epsilon}'(t)+\frac{1}{\epsilon^2}\Theta_{n,\epsilon}\tilde{w}_{n,\epsilon}(t)=0.
\end{equation}
with initial values
\[
\tilde{w}_{n,\epsilon}(0)=f^\epsilon_n(0),~\tilde{w}_{n,\epsilon}'(0)=\theta_{n,\epsilon}f^\epsilon_n(0)+(f_n^\epsilon)'(0)=\theta_{n,\epsilon}f^\epsilon_n(0)+\frac{1}{\epsilon}g^\epsilon_n(0).
\]
Two useful identities are obtained using~\eqref{eq:tildew}. First, one has
\begin{align*}
\frac12\frac{d}{dt}|\tilde{w}_{n,\epsilon}'(t)|^2&+\big(\frac{1}{\epsilon^2}-2\theta_{n,\epsilon}\big)|\tilde{w}_{n,\epsilon}'(t)|^2+\frac{\Theta_{n,\epsilon}}{2\epsilon^2}\frac{d}{dt}|\tilde{w}_{n,\epsilon}(t)|^2\\
&=\Bigl\langle \tilde{w}_{n,\epsilon}''(t)+\big(\frac{1}{\epsilon^2}-2\theta_{n,\epsilon}\big)\tilde{w}_{n,\epsilon}'(t)+\frac{\Theta_{n,\epsilon}}{\epsilon^2}\tilde{w}_{n,\epsilon}(t),\tilde{w}_{n,\epsilon}'(t)\Bigr\rangle\\
&=0.
\end{align*}
One then obtains the following result by integration and using the initial values for $\tilde{w}_{n,\epsilon}(0)$ and $\tilde{w}_{n,\epsilon}'(0)$: for all $t\ge 0$ one has
\begin{align} \label{eq:tildew1}
\epsilon^2|\tilde{w}_{n,\epsilon}'(t)|^2+2(1-2\epsilon^2\theta_{n,\epsilon})\int_{0}^{t}|\tilde{w}_{n,\epsilon}'(s)|^2\,ds+\Theta_{n,\epsilon}|\tilde{w}_{n,\epsilon}(t)|^2\nonumber&=\epsilon^2|\tilde{w}_{n,\epsilon}'(0)|^2+\Theta_{n,\epsilon}|\tilde{w}_{n,\epsilon}(0)|^2 \nonumber\\
&=|\epsilon\theta_{n,\epsilon}f^\epsilon_n(0)+g^\epsilon_n(0)|^2+\Theta_{n,\epsilon}|f^\epsilon_n(0)|^2.
\end{align}
Second, one has
\begin{align*}
\frac{1}{2}\frac{d}{dt}\bigl|\epsilon^2\tilde{w}_{n,\epsilon}'(t)+(1-2\epsilon^2\theta_{n,\epsilon})\tilde{w}_{n,\epsilon}(t)\bigr|^2&=\Bigl\langle \epsilon^2\tilde{w}_{n,\epsilon}''(t)+(1-2\epsilon^2\theta_{n,\epsilon})\tilde{w}_{n,\epsilon}'(t), \epsilon^2\tilde{w}_{n,\epsilon}'(t)+(1-2\epsilon^2\theta_{n,\epsilon})\tilde{w}_{n,\epsilon}(t)  \Bigr\rangle\\
&=-\Bigl\langle\Theta_{n,\epsilon}\tilde{w}_{n,\epsilon}(t), \epsilon^2\tilde{w}_{n,\epsilon}'(t)+(1-2\epsilon^2\theta_{n,\epsilon})\tilde{w}_{n,\epsilon}(t)\Bigr\rangle\\
&=-\frac{\epsilon^2}{2}\Theta_{n,\epsilon}\frac{d}{dt}|\tilde{w}_{n,\epsilon}(t)|^2-(1-2\epsilon^2\theta_{n,\epsilon})\Theta_{n,\epsilon}|\tilde{w}_{n,\epsilon}(t)|^2.
\end{align*}
One then obtains the following result by integration and using the initial values for $\tilde{w}_{n,\epsilon}(0)$ and $\tilde{w}_{n,\epsilon}'(0)$: for all $t\ge 0$ one has
\begin{align}\label{eq:tildew2}
	&\bigl|\epsilon^2\tilde{w}_{n,\epsilon}'(t)+(1-2\epsilon^2\theta_{n,\epsilon})\tilde{w}_{n,\epsilon}(t)\bigr|^2+\epsilon^2\Theta_{n,\epsilon}|\tilde{w}_{n,\epsilon}(t)|^2+2(1-2\epsilon^2\theta_{n,\epsilon})\Theta_{n,\epsilon}\int_0^t|\tilde{w}_{n,\epsilon}(s)|^2\,ds \nonumber\\
   &=\bigl|\epsilon^2\tilde{w}_{n,\epsilon}'(0)+(1-2\epsilon^2\theta_{n,\epsilon})\tilde{w}_{n,\epsilon}(0)\bigr|^2+\epsilon^2\Theta_{n,\epsilon}|\tilde{w}_{n,\epsilon}(0)|^2\nonumber\\
   &=\bigl|\epsilon g^\epsilon_n(0)+(1-\epsilon^2\theta_{n,\epsilon})f^\epsilon_n(0)\bigr|^2+\epsilon^2(\lambda_n-\theta_{n,\epsilon}+\epsilon^2\theta_{n,\epsilon}^2)|f^\epsilon_n(0)|^2.
\end{align}

Set $\tilde{w}_{n,\epsilon}^{0,1}(t)=e^{\theta_{n,\epsilon}t}f_n^{\epsilon,0,1}(t)$ and $\tilde{w}_{n,\epsilon}^{1,0}(t)=e^{\theta_{n,\epsilon}t}f_n^{\epsilon,1,0}(t)$ for all $t\ge 0$.

The proof of Lemma~\ref{lem:fngn} can now be performed using the auxiliary tools introduced above.

\begin{proof}[Proof of Lemma~\ref{lem:fngn}]
Let $n\in\N$ and $\epsilon\in(0,1)$. First, let us assume that $1-4\lambda_n\epsilon^2<0$.
\begin{itemize}
\item Proof of the inequality~\eqref{eq:lem_fngn01-1}. 
It is straightforward to check that for all $z\in\R$ one has
\[
|\sin(z)|\le \frac{|z|}{\max(1,|z|)},
\]
by combining the upper bounds $|\sin(z)|\le 1$ and $|\sin(z)|\le |z|$. Using the expression~\eqref{eq:fn1} of $f_n^{\epsilon,0,1}(t)$ and the inequality above with $z=\frac{\sqrt{4\lambda_n\epsilon^2-1}}{2\epsilon^2}t$, one obtains
\begin{align*}
 	|f_n^{\epsilon,0,1}(t)|^2&=\Bigl|e^{-\frac{t}{2\epsilon^2}}\frac{2\epsilon}{\sqrt{4\lambda_n\epsilon^2-1}}\sin\big(\frac{\sqrt{4\lambda_n\epsilon^2-1}}{2\epsilon^2}t\big)\Bigr|^2\\
 	&\le e^{-\frac{t}{\epsilon^2}}\frac{t^2}{\epsilon^2}\frac{1}{\max(1,\frac{(4\lambda_n\epsilon^2-1)t^2}{4\epsilon^4})}.
\end{align*}
Then writing
\[
\frac{t^2}{\epsilon^2}=\frac{1}{4\lambda_n}\frac{4\lambda_n\epsilon^2 t^2}{\epsilon^4}=\frac{1}{4\lambda_n}\frac{(4\lambda_n\epsilon^2-1)t^2}{\epsilon^4}+\frac{1}{4\lambda_n}\frac{t^2}{\epsilon^4},
\]
and the upper bound $\underset{x\ge 0}\sup~(1+x^2)e^{-\frac{x}{2}}<\infty$, one obtains the inequality
\[
|f_n^{\epsilon,0,1}(t)|^2\le \frac{1}{4\lambda_n}e^{-\frac{t}{\epsilon^2}}\big(1+\frac{t^2}{\epsilon^4}\big)\le \frac{C}{\lambda_n}e^{-\frac{t}{2\epsilon^2}}.
\]
This concludes the proof of the first inequality in~\eqref{eq:lem_fngn01-1}. To prove the second inequality, note that due to the assumption $1-4\lambda_n\epsilon^2<0$ one has
\[
\theta_{n,\epsilon}=\frac{1}{4\epsilon^2}>0,\quad \Theta_{n,\epsilon}=\lambda_n-\theta_{n,\epsilon}+\epsilon^2\theta_{n,\epsilon}^2=\lambda_n-\frac{3}{16\epsilon^2}>\frac{\lambda_n}{4},\quad 1-2\epsilon^2\theta_{n,\epsilon}=\frac12.
\]
As a result, using the inequalities~\eqref{eq:tildew1} and~\eqref{eq:tildew2}, with the initial values $f_n^{\epsilon,0,1}(0)=0$ and $g_n^{\epsilon,0,1}(0)=1$ one obtains for all $t\ge 0$
\begin{align*}
&\epsilon|(\tilde{w}_{n,\epsilon}^{0,1})'(t)|\le 1\\
&|\epsilon^2(\tilde{w}_{n,\epsilon}^{0,1})'(t)+\frac{1}{2}\tilde{w}_{n,\epsilon}^{0,1}(t)\big|\le \epsilon,
\end{align*}
and using the triangle inequality then gives
\[
|\tilde{w}_{n,\epsilon}^{0,1}(t)|\le 2\Bigl(\big|\epsilon^2(\tilde{w}_{n,\epsilon}^{0,1})'(t)+\frac{1}{2}\tilde{w}_{n,\epsilon}^{0,1}(t)\big|+\epsilon^2|(\tilde{w}_{n,\epsilon}^{0,1})'(t)|\Bigr)\le 4\epsilon.
\]
Finally, using the identity
\[
g_n^{\epsilon,0,1}(t)=\epsilon (f_n^{\epsilon,0,1})'(t)=\epsilon e^{-\theta_{n,\epsilon}t}\bigl((\tilde{w}_{n,\epsilon}^{0,1})'(t)-\theta_{n,\epsilon}\tilde{w}_{n,\epsilon}^{0,1}(t)\bigr),
\]
and the value $\theta_{n,\epsilon}=\frac{1}{4\epsilon^2}$, one obtains for all $t\ge 0$
\[
|g^{\epsilon,0,1}_n(t)|\le \epsilon e^{-\frac{t}{4\epsilon^2}}|(\tilde{w}_{n,\epsilon}^{0,1})'(t)|+\frac{1}{4\epsilon}e^{-\frac{t}{4\epsilon^2}}|\tilde{w}_{n,\epsilon}^{0,1}(t)|\le 2e^{-\frac{t}{4\epsilon^2}}.
\]
This concludes the proof of the second inequality in~\eqref{eq:lem_fngn01-1}. The proof of~\eqref{eq:lem_fngn01-1} is thus completed.

\item Proof of the inequality~\eqref{eq:lem_fngn10-1}. Using the expression~\eqref{eq:fn1} of $f_n^{\epsilon,1,0}(t)$ and the inequalities $|\cos(z)|\le 1$ and $|\sin(z)|\le |z|$ for all $z\in\R$, one has
\begin{align*}
	|f_n^{\epsilon,1,0}(t)|&= e^{-\frac{t}{2\epsilon^2}}\Big|\cos(\frac{\sqrt{4\lambda_n\epsilon^2-1}}{2\epsilon^2}t)+\frac{1}{\sqrt{4\lambda_n\epsilon^2-1}}\sin\big(\frac{\sqrt{4\lambda_n\epsilon^2-1}}{2\epsilon^2}t\big)\Big|\\
	&\le e^{-\frac{t}{2\epsilon^2}}\big(1+\frac{t}{2\epsilon^2}\big).
\end{align*}
Using then the upper bound $\underset{x\ge 0}\sup~(1+x)e^{-\frac{x}{2}}<\infty$ then gives
\[
|f_n^{\epsilon,1,0}(t)|\le Ce^{-\frac{t}{4\epsilon^2}}.
\]
The proof of~\eqref{eq:lem_fngn10-1} is thus completed.
\end{itemize}
Let us now assume that $1-4\lambda_n\epsilon^2\ge 0$.
\begin{itemize}
\item Proof of the inequality~\eqref{eq:lem_fngn01-2}. Note that due to the assumption $1-4\lambda_n\epsilon^2\ge 0$ one has
\[
\theta_{n,\epsilon}=\lambda_n>0,\quad \Theta_{n,\epsilon}=\lambda_n-\theta_{n,\epsilon}+\epsilon^2\theta_{n,\epsilon}^2=\lambda_n^2\epsilon^2>0,~\quad 1-2\epsilon^2\theta_{n,\epsilon}\ge 2\lambda_n\epsilon^2.
\]
As a result, using the inequalities~\eqref{eq:tildew1} and~\eqref{eq:tildew2} with the initial values $f_n^{\epsilon,0,1}(0)=0$ and $g_n^{\epsilon,0,1}(0)=1$, one obtains for all $t\ge 0$
\begin{align*}
&\epsilon|(\tilde{w}_{n,\epsilon}^{0,1})'(t)|\le 1\\
&\bigl|\epsilon^2(\tilde{w}_{n,\epsilon}^{0,1})'(t)+(1-2\epsilon^2\theta_{n,\epsilon})\tilde{w}_{n,\epsilon}^{0,1}(t)\bigr|\le \epsilon.
\end{align*}
Moreover, using the assumption $1-4\lambda_n\epsilon^2\ge 0$, one has the lower bound
\[
1-2\lambda_n\epsilon^2\ge \frac12.
\]
Therefore, using the triangle inequality and the two inequalities above, for all $t\ge 0$ one obtains
\begin{align*}
|\tilde{w}_{n,\epsilon}^{0,1}(t)|&\le 2(1-2\lambda_n\epsilon^2)|\tilde{w}_{n,\epsilon}^{0,1}(t)|\\
&\le 2\Bigl(\big|\epsilon^2(\tilde{w}_{n,\epsilon}^{0,1})'(t)+(1-2\lambda_n\epsilon^2)\tilde{w}_{n,\epsilon}^{0,1}(t)\big|+\epsilon^2|(\tilde{w}_{n,\epsilon}^{0,1})'(t)|\Bigr)\\
&\le 4\epsilon.
\end{align*}
Thus for all $t\ge 0$ one has
\[
|f_n^{\epsilon,0,1}(t)|=e^{-\theta_{n,\epsilon}t}|\tilde{w}_{n,\epsilon}^{0,1}(t)|\le 4\epsilon e^{-\lambda_n t}.
\]
This provides the first inequality from~\eqref{eq:lem_fngn01-2}. To prove the second inequality, owing to the equality~\eqref{eq:exactsolution-g}, recalling that $g_n^{\epsilon,0,1}(0)=1$ and using the inequality proved above for $|f_n^{\epsilon,0,1}(t)|$, one has for all $t\ge 0$
\begin{align*}
	|g_n^{\epsilon,0,1}(t)|&\le e^{-\frac{t}{\epsilon^2}}+\int_0^t\frac{\lambda_n}{\epsilon}e^{-\frac{t-s}{\epsilon^2}}|f_n^{\epsilon,0,1}(s)|\,ds\\
	&\le e^{-\frac{t}{\epsilon^2}}+4\lambda_n\int_0^te^{-\frac{t-s}{\epsilon^2}}e^{-\lambda_ns}\,ds\\
	&\le e^{-\frac{t}{\epsilon^2}}+\frac{4\lambda_n\epsilon^2}{1-\lambda_n\epsilon^2}\big(e^{-\lambda_nt}-e^{-\frac{t}{\epsilon^2}}\big)\\
	&\le e^{-\frac{t}{\epsilon^2}}+C\lambda_n\epsilon^2e^{-\lambda_nt},
\end{align*}
using the fact that $1-\lambda_n\epsilon^2\ge\frac34$ owing to the assumption $1-4\lambda_n\epsilon^2\ge 0$. This concludes the proof of the second inequality in~\eqref{eq:lem_fngn01-2}. The proof of~\eqref{eq:lem_fngn01-2} is thus completed.
\item Proof of the inequality~\eqref{eq:lem_fngn10-2}. Using the inequality~\eqref{eq:tildew1}, with the initial values $f_n^{\epsilon,1,0}(0)=1$ and $g_n^{\epsilon,1,0}(0)=0$, one obtains for all $t\ge 0$
\[
\Theta_{n,\epsilon}|\tilde{w}_{n,\epsilon}^{1,0}(t)|^2\le (\epsilon \theta_{n,\epsilon})^2+\Theta_{n,\epsilon}.
\]
Recalling that $\theta_{n,\epsilon}=\lambda_n$ and $\Theta_{n,\epsilon}=\lambda^2\epsilon^2$ owing to the assumption $1-4\lambda_n\epsilon^2\ge 0$, one has for all $t\ge 0$
\[
|\tilde{w}_{n,\epsilon}^{1,0}(t)|^2\le 2.
\]
Therefore for all $t\ge 0$ one has
\[
|f_n^{\epsilon,1,0}(t)|=e^{-\theta_{n,\epsilon}t}|\tilde{w}_{n,\epsilon}^{1,0}(t)|\leq \sqrt{2}e^{-t\lambda_n}.
\]
The proof of~\eqref{eq:lem_fngn10-2} is thus completed.
\end{itemize}
The proof of Lemma~\ref{lem:fngn} is now completed.
\end{proof}

\subsubsection{Proofs of Lemma~\ref{lem:expAeps-smoothing1} and Lemma~\ref{lem:expAeps-smoothing2}}

\begin{proof}[Proof of Lemma~\ref{lem:expAeps-smoothing1}]
Let $u\in H^{\alpha+\rho-2\delta-1}$, then one has the following expression for $e^{tA_\epsilon}(0,u)$: for all $t\ge 0$,
\[
e^{tA_\epsilon}(0,u)=\sum_{n\in\N}\langle u,e_n\rangle f_n^{\epsilon,0,1}(t)(e_n,0)+\sum_{n\in\N}\langle u,e_n\rangle g_n^{\epsilon,0,1}(t)(0,e_n),
\]
using~\eqref{eq:expAeps} and the notation introduced above the statement of Lemma~\ref{lem:fngn}. Alternatively, this may be written as
\[
\Pi_u\Bigl(e^{tA_\epsilon}(0,u)\Bigr)=\sum_{n\in\N}\langle u,e_n\rangle f_n^{\epsilon,0,1}(t)e_n,\quad \Pi_v\Bigl(e^{tA_\epsilon}(0,u)\Bigr)=\sum_{n\in\N}\langle u,e_n\rangle g_n^{\epsilon,0,1}(t)e_n,
\]
and as a result one has
\begin{align*}
\big\|e^{tA_\epsilon}\big(0,u\big)\big\|_{\HH^\alpha}^2
&=\big\|\Pi_ue^{tA_\epsilon}\big(0,u\big)\big\|^2_{H^\alpha}+\big\|\Pi_ve^{tA_\epsilon}\big(0,u\big)\big\|_{H^{\alpha-1}}^2\\
&=\sum_{n\in\N}\lambda_n^\alpha \langle u,e_n\rangle^2 |f_n^{\epsilon,0,1}(t)|^2+\sum_{n\in\N}\lambda_n^{\alpha-1}\langle u,e_n\rangle^2 |g_n^{\epsilon,0,1}(t)|^2\\
&=\sum_{n\in\N}\lambda_n^{\alpha+\rho-2\delta-1}\langle u,e_n\rangle^2 \lambda_n^{1+2\delta-\rho}|f_n^{\epsilon,0,1}(t)|^2+\sum_{n\in\N}\lambda_n^{\alpha+\rho-2\delta-1}\langle u,e_n\rangle^2 \lambda_n^{2\delta-\rho}|g_n^{\epsilon,0,1}(t)|^2\\
&\le \|u\|_{H^{\alpha-2\delta+\rho-1}}^2 \underset{n\in\N}\sup~\Bigl(\lambda_n^{1+2\delta-\rho}|f_n^{\epsilon,0,1}(t)|^2+\lambda_n^{2\delta-\rho}|g_n^{\epsilon,0,1}(t)|^2\Bigr).
\end{align*}

To proceed, one needs to treat separately the cases $1-4\lambda_n\epsilon^2<0$ and $1-4\lambda_n\epsilon^2\ge 0$: for all $n\in\N$, $\epsilon\in(0,1)$ and $t\in(0,\infty)$, using the inequalities~\eqref{eq:lem_fngn01-1} and~\eqref{eq:lem_fngn01-2} from Lemma~\ref{lem:fngn}, one obtains
\begin{align*}
\lambda_n^{1+2\delta-\rho}|f_n^{\epsilon,0,1}(t)|^2+\lambda_n^{2\delta-\rho}|g_n^{\epsilon,0,1}(t)|^2&=\bigl(\lambda_n^{1+2\delta-\rho}|f_n^{\epsilon,0,1}(t)|^2+\lambda_n^{2\delta-\rho}|g_n^{\epsilon,0,1}(t)|^2\bigr)\mathds{1}_{1-4\lambda_n\epsilon^2<0}\\
&+\bigl(\lambda_n^{1+2\delta-\rho}|f_n^{\epsilon,0,1}(t)|^2+\lambda_n^{2\delta-\rho}|g_n^{\epsilon,0,1}(t)|^2\bigr)\mathds{1}_{1-4\lambda_n\epsilon^2\ge 0}\\
&\le C\lambda_n^{2\delta-\rho}e^{-\frac{t}{2\epsilon^2}}\mathds{1}_{1-4\lambda_n\epsilon^2<0}\\
&+C\Bigl(\epsilon^2\lambda_n^{1+2\delta-\rho}e^{-2\lambda_nt}+\lambda_n^{2\delta-\rho}e^{-\frac{2t}{\epsilon^2}}+\lambda_n^{2\delta-\rho+2}\epsilon^4e^{-2\lambda_nt}\Bigr)\mathds{1}_{1-4\lambda_n\epsilon^2\ge 0}.
\end{align*}
On the one hand, using the inequality $\underset{x\ge 0}\sup~x^{2\delta} e^{-x}<\infty$, one obtains
\begin{align*}
C\lambda_n^{2\delta-\rho}e^{-\frac{t}{2\epsilon^2}}\mathds{1}_{1-4\lambda_n\epsilon^2<0}&\le C\lambda_n^{2\delta-\rho}\frac{2\epsilon^{4\delta}}{t^{2\delta}}\mathds{1}_{1-4\lambda_n\epsilon^2<0}\\
&\le Ct^{-2\delta}\epsilon^{2\rho}\bigl(\lambda_n\epsilon^2\bigr)^{2\delta-\rho}\mathds{1}_{1-4\lambda_n\epsilon^2<0}\\
&\le Ct^{-2\delta}\epsilon^{2\rho},
\end{align*}
using the conditions $2\delta-\rho\le 0$ and $4\lambda_n\epsilon^2>1$ in the last step.

On the other hand, using the inequality $\underset{x\ge 0}\sup~x^{2\delta} e^{-x}<\infty$, one obtains
\begin{align*}
\epsilon^2\lambda_n^{1+2\delta-\rho}e^{-2\lambda_nt}\mathds{1}_{1-4\lambda_n\epsilon^2\ge 0}&\le Ct^{-2\delta}\epsilon^2\lambda_n^{1-\rho}\mathds{1}_{1-4\lambda_n\epsilon^2\ge 0}\\
&\le Ct^{-2\delta}\epsilon^{2\rho}\bigl(\lambda_n\epsilon^2\bigr)^{1-\rho}\mathds{1}_{1-4\lambda_n\epsilon^2\ge 0}\\
&\le Ct^{-2\delta}\epsilon^{2\rho}.
\end{align*}
Finally, observe that the assumption $1-4\lambda_n\epsilon^2\ge 0$ ensures that
\[
\lambda_n^{2\delta-\rho+2}\epsilon^4e^{-2\lambda_nt}\mathds{1}_{1-4\lambda_n\epsilon^2\ge 0}\le \epsilon^2\lambda_n^{1+2\delta-\rho}e^{-2\lambda_nt}\mathds{1}_{1-4\lambda_n\epsilon^2\ge 0}\le Ct^{-2\delta}\epsilon^{2\rho},
\]
and using again the inequality $\underset{x\ge 0}\sup~x^{2\delta} e^{-x}<\infty$, one obtains
\begin{align*}
\lambda_n^{2\delta-\rho}e^{-\frac{2t}{\epsilon^2}}\mathds{1}_{1-4\lambda_n\epsilon^2\ge 0}&\le C\lambda_n^{2\delta-\rho}\frac{\epsilon^{4\delta}}{t^{2\delta}}e^{-\frac{t}{\epsilon^2}}\mathds{1}_{1-4\lambda_n\epsilon^2\ge 0}\\
&\le Ct^{-2\delta}\epsilon^{2\rho}\lambda_n^{2\delta-\rho}\epsilon^{2(2\delta-\rho)}e^{-\frac{t}{\epsilon^2}}\mathds{1}_{1-4\lambda_n\epsilon^2\ge 0}\\
&\le Ct^{-2\delta}\epsilon^{2\rho}\epsilon^{2(2\delta-\rho)}e^{-\frac{t}{\epsilon^2}},
\end{align*}
using the condition $2\delta-\rho\le 0$.

Gathering the estimates yields for all $t\in(0,\infty)$ the upper bound
\[
\underset{n\in\N}\sup~\Bigl(\lambda_n^{1+2\delta-\rho}|f_n^{\epsilon,0,1}(t)|^2+\lambda_n^{2\delta-\rho}|g_n^{\epsilon,0,1}(t)|^2\Bigr)\le Ct^{-2\delta}\epsilon^{2\rho}\Bigl(1+\epsilon^{2(2\delta-\rho)}e^{-\frac{t}{\epsilon^2}}\Bigr).
\]
Due to the inequality above, one obtains~\eqref{eq:lem_expAeps-smoothing1} and the proof of Lemma~\ref{lem:expAeps-smoothing1} is thus completed.
\end{proof}

\begin{proof}[Proof of Lemma~\ref{lem:expAeps-smoothing2}]
Like in the proof of Lemma~\ref{lem:expAeps-smoothing1} above, for all $v\in H^{\alpha-1-2\delta}$ and for all $t\ge 0$ one has
\[
\Pi_u\Bigl(e^{tA_\epsilon}(0,v)\Bigr)=\sum_{n\in\N}\langle v,e_n\rangle f_n^{\epsilon,0,1}(t)e_n
\]
and as a result one has
\begin{align*}
\|\Pi_v\Bigl(e^{tA_\epsilon}(0,v)\Bigr)\|_H^2&=\sum_{n\in\N}\langle v,e_n\rangle^2 |f_n^{\epsilon,0,1}(t)|^2\\
&=\sum_{n\in\N}\lambda_n^{\alpha-1-2\delta} \langle v,e_n\rangle^2 \lambda_n^{1+2\delta-\alpha}|f_n^{\epsilon,0,1}(t)|^2\\
&\le \|v\|_{H^{\alpha-1-2\delta}}^2 \underset{n\in\N}\sup~\Bigl(\lambda_n^{1+2\delta-\alpha}|f_n^{\epsilon,0,1}(t)|^2\Bigr).
\end{align*}
To proceed, one needs to treat separately the cases $1-4\lambda_n\epsilon^2<0$ and $1-4\lambda_n\epsilon^2\ge 0$: for all $n\in\N$, $\epsilon\in(0,1)$ and $t\in(0,\infty)$, using the inequalities~\eqref{eq:lem_fngn01-1} and~\eqref{eq:lem_fngn01-2} from Lemma~\ref{lem:fngn}, one obtains
\begin{align*}
\lambda_n^{1+2\delta-\alpha}|f_n^{\epsilon,0,1}(t)|^2&=\lambda_n^{1+2\delta-\alpha}|f_n^{\epsilon,0,1}(t)|^2\mathds{1}_{1-4\lambda_n\epsilon^2<0}+\lambda_n^{1+2\delta-\alpha}|f_n^{\epsilon,0,1}(t)|^2\mathds{1}_{1-4\lambda_n\epsilon^2\ge 0}\\
&\le C\lambda_n^{2\delta-\alpha}e^{-\frac{t}{2\epsilon^2}}\mathds{1}_{1-4\lambda_n\epsilon^2<0}+C\lambda_n^{1+2\delta-\alpha}\epsilon^2e^{-2\lambda_nt}\mathds{1}_{1-4\lambda_n\epsilon^2\ge 0}.
\end{align*}
On the one hand, using the inequality $\underset{x\ge 0}\sup~x^{2\delta} e^{-x}<\infty$, one obtains
\begin{align*}
\lambda_n^{2\delta-\alpha}e^{-\frac{t}{2\epsilon^2}}\mathds{1}_{1-4\lambda_n\epsilon^2<0}&\le C\lambda_n^{2\delta-\alpha}\frac{\epsilon^{4\delta}}{t^{2\delta}}\mathds{1}_{1-4\lambda_n\epsilon^2<0}\\
&\le Ct^{-2\delta}\epsilon^{2\alpha}\bigl(\lambda_n\epsilon^2\bigr)^{2\delta-\alpha}\mathds{1}_{1-4\lambda_n\epsilon^2<0}\\
&\le Ct^{-2\delta}\epsilon^{2\alpha},
\end{align*}
using the conditions $2\delta-\alpha\le 0$ and $4\lambda_n\epsilon^2>1$ in the last step.

On the other hand, using the inequality $\underset{x\ge 0}\sup~x^{2\delta} e^{-x}<\infty$, one obtains
\begin{align*}
\lambda_n^{1+2\delta-\alpha}\epsilon^2e^{-2\lambda_nt}\mathds{1}_{1-4\lambda_n\epsilon^2\ge 0}&\le Ct^{-2\delta}\lambda_n^{1-\alpha}\epsilon^2\mathds{1}_{1-4\lambda_n\epsilon^2\ge 0}\\
&\le Ct^{-2\delta}\epsilon^{2\alpha}\bigl(\lambda_n\epsilon^2\bigr)^{1-\alpha}\mathds{1}_{1-4\lambda_n\epsilon^2\ge 0}\\
&\le Ct^{-2\delta}\epsilon^{2\alpha}.
\end{align*}
Gathering the estimates yields for all $t\in(0,\infty)$ the upper bound
\[
\underset{n\in\N}\sup~\Bigl(\lambda_n^{1+2\delta-\alpha}|f_n^{\epsilon,0,1}(t)|^2\Bigr)\le Ct^{-2\delta}\epsilon^{2\alpha}.
\]
Due to the inequality above, one obtains~\eqref{eq:lem_expAeps-smoothing2-1}.

It remains to prove the inequality~\eqref{eq:lem_expAeps-smoothing2-2}. Observe that for all $u\in H$ and all $t\ge 0$ one has
\[
\Pi_u\Bigl(e^{tA_\epsilon}(u,0)\Bigr)=\sum_{n\in\N}\langle u,e_n\rangle f_n^{\epsilon,1,0}(t)e_n
\]
and as a result one has
\begin{align*}
\|\Pi_u\Bigl(e^{tA_\epsilon}(u,0)\Bigr)\|_H^2&=\sum_{n\in\N}\langle u,e_n\rangle^2|f_n^{\epsilon,1,0}(t)|^2\\
&=\sum_{n\in\N}\langle u,e_n\rangle^2|f_n^{\epsilon,1,0}(t)|^2\mathds{1}_{1-4\lambda_n\epsilon^2<0}+\sum_{n\in\N}\langle u,e_n\rangle^2|f_n^{\epsilon,1,0}(t)|^2\mathds{1}_{1-4\lambda_n\epsilon^2\ge 0}.
\end{align*}
On the one hand, using the inequality~\eqref{eq:lem_fngn10-1} from Lemma~\ref{lem:fngn} and the inequality $\underset{x\ge 0}\sup~x^{\alpha_1}e^{-x}<\infty$, for all $n\in\N$, $\epsilon\in(0,1)$ and $t\in(0,\infty)$ one has
\[
|f_n^{\epsilon,1,0}(t)|^2\mathds{1}_{1-4\lambda_n\epsilon^2<0}\le Ce^{-\frac{t}{2\epsilon^2}}\le C\epsilon^{2\alpha_1}t^{-\alpha_1}.
\]
On the other hand, using the inequality~\eqref{eq:lem_fngn10-2} from Lemma~\ref{lem:fngn} and the inequality $\underset{x\ge 0}\sup~x^{\alpha_2}e^{-x}<\infty$, for all $n\in\N$, $\epsilon\in(0,1)$ and $t\in(0,\infty)$ one has
\begin{align*}
|f_n^{\epsilon,1,0}(t)|^2\mathds{1}_{1-4\lambda_n\epsilon^2\ge 0}\le Ce^{-2\lambda_nt}\mathds{1}_{1-4\lambda_n\epsilon^2\ge 0}\le Ct^{-\alpha_2}\lambda_n^{-\alpha_2}.
\end{align*}
Therefore one obtains for all $t\in(0,\infty)$ and all $u\in H$
\begin{align*}
\|\Pi_u\Bigl(e^{tA_\epsilon}(u,0)\Bigr)\|_H^2&\le C\epsilon^{2\alpha_1}t^{-\alpha_1}\sum_{n\in\N}\langle u,e_n\rangle^2+Ct^{-\alpha_2}\sum_{n\in\N}\lambda_n^{-\alpha_2}\langle u,e_n\rangle^2\\
&\le C\epsilon^{2\alpha_1}t^{-\alpha_1}\|u\|_{H}^2+Ct^{-\alpha_2}\|u\|_{H^{-\alpha_2}}^2.
\end{align*}
This concludes the proof of~\eqref{eq:lem_expAeps-smoothing2-2}. The proof of Lemma~\ref{lem:expAeps-smoothing2} is thus completed.
\end{proof}

\subsection{Convergence and error bounds on the semigroups}\label{sec:errorboundssemigroups}

The objective of this section is to provide precise results to justify the convergence properties
\[
\Pi_u e^{tA_\epsilon}(u,0)\underset{\epsilon\to 0}\to e^{-t\IL}u
\]
and
\[
\Pi_u e^{tA_\epsilon}(0,\frac{v}{\epsilon})\underset{\epsilon\to 0}\to e^{-t\IL}v,
\]
and to obtain error bounds, depending on the regularity properties of $u$ and $v$.
\begin{lemma}\label{lem:semigroups-convergence}
For all $\alpha\in[0,1]$ and all $\delta\in[0,\frac{\alpha}{2}]$, there exists $C_{\alpha,\delta}\in(0,\infty)$ such that, for all $\epsilon\in(0,1)$, all $t>0$ and all $u\in H^{\alpha-2\delta}$ one has
\begin{equation}\label{eq:lem_semigroups-convergence-u}
\|\Pi_u e^{tA_\epsilon}(u,0)-e^{-t\IL}u\|_{H}\le C_{\alpha,\delta}\epsilon^{\alpha}t^{-\delta}\|u\|_{H^{\alpha-2\delta}}.
\end{equation}
Moreover, for all $\alpha\in[0,1]$ and all $\delta\in[\frac{\alpha}{2},\frac12]$, there exists $C_{\alpha,\delta}\in(0,\infty)$ such that, for all $\epsilon\in(0,1)$, all $t>0$ and all $v\in H^{\alpha-2\delta}$ one has
\begin{equation}\label{eq:lem_semigroups-convergence-v}
\|\Pi_u e^{tA_\epsilon}(0,\frac{v}{\epsilon})-e^{-t\IL}v\|_{H}\le C_{\alpha,\delta}\epsilon^{\alpha}t^{-\delta}\|v\|_{H^{\alpha-2\delta}}.
\end{equation}
\end{lemma}
Notice that one has $\alpha-2\delta\ge 0$ for~\eqref{eq:lem_semigroups-convergence-u} but $\alpha-2\delta\le 0$ for~\eqref{eq:lem_semigroups-convergence-v}.

The proof of Lemma~\ref{lem:semigroups-convergence} is a straightforward consequence of the following auxiliary result. The statement and the proof of Lemma~\ref{lem:aux-conv} exploit the notation introduced in Section~\ref{sec:expAeps} for the proof of Lemma~\ref{lem:fngn}.

\begin{lemma}\label{lem:aux-conv}
For all $\alpha\in[0,1]$ and all $\delta\in[0,\frac{\alpha}{2}]$, there exists $C_{\alpha,\delta}\in(0,\infty)$ such that, for all $\epsilon\in(0,1)$ and all $t>0$ one has
\begin{equation}\label{eq:lem_aux-conv-1}
\underset{n\in\N}\sup~\Bigl(\lambda_n^{2\delta-\alpha}\big|f_n^{\epsilon,1,0}(t)-e^{-t\lambda_n}\big|^2\Bigr)\le Ct^{-2\delta}\epsilon^{2\alpha}.
\end{equation}

Moreover, for all $\alpha\in[0,1]$ and all $\delta\in[\frac{\alpha}{2},\frac12]$, there exists $C_{\alpha,\delta}\in(0,\infty)$ such that, for all $\epsilon\in(0,1)$ and all $t>0$ one has
\begin{equation}\label{eq:lem_aux-conv-2}
\sup~\Bigl(\lambda_n^{2\delta-\alpha}\big|\frac{f_n^{\epsilon,0,1}(t)}{\epsilon}-e^{-t\lambda_n}\big|^2\Bigr)\le Ct^{-2\delta}\epsilon^{2\alpha}.
\end{equation}
\end{lemma}

\begin{proof}[Proof of Lemma~\ref{lem:aux-conv}]

Let us first prove the inequality~\eqref{eq:lem_aux-conv-1}, and assume that $\delta\in[0,\frac{\alpha}{2}]$. Let $n\in\N$ and $\epsilon\in(0,1)$.

On the one hand, if $1-4\lambda_n\epsilon^2<0$, owing to the inequality~\eqref{eq:lem_fngn10-1} and to the lower bound $\lambda_n\ge \epsilon^2/4$, one has for all $t\ge 0$
\begin{align*}
|f_n^{\epsilon,1,0}(t)-e^{-t\lambda_n}|^2&\le 2|f_n^{\epsilon,1,0}(t)|^2+2e^{-2t\lambda_n}\\
&\le Ce^{-\frac{t}{2\epsilon^2}}.
\end{align*}
As a result, using the inequality $\underset{x\in\R}\sup~x^{2\delta}e^{-x}<\infty$, one obtains for all $t>0$
\begin{align*}
\lambda_n^{2\delta-\alpha}\big|f_n^{\epsilon,1,0}(t)-e^{-t\lambda_n}\big|^2\mathds{1}_{1-4\lambda_n\epsilon^2<0}&\le C\lambda_n^{2\delta-\alpha}\frac{\epsilon^{4\delta}}{t^{2\delta}}\mathds{1}_{1-4\lambda_n\epsilon^2<0}\\
&\le Ct^{-2\delta}\epsilon^{2\alpha}(\lambda_n\epsilon^2)^{2\delta-\alpha}\mathds{1}_{1-4\lambda_n\epsilon^2<0}\\
&\le Ct^{-2\delta}\epsilon^{2\alpha},
\end{align*}
using the conditions $2\delta-\alpha\le 0$ and $4\lambda_n\epsilon^2>1$ in the last step.

On the other hand, if $1-4\lambda_n\epsilon^2\ge 0$, recalling that $\tilde{w}_{n,\epsilon}^{1,0}(t)=e^{\theta_{n,\epsilon}t}f_n^{\epsilon,1,0}(t)$ with $\theta_{n,\epsilon}=\lambda_n$, one has for all $t\ge 0$
\[
f_n^{\epsilon,1,0}(t)-e^{-t\lambda_n}=e^{-t\lambda_n}\bigl(\tilde{w}_{n,\epsilon}^{1,0}(t)-1\bigr)=e^{-t\lambda_n}\bigl(\tilde{w}_{n,\epsilon}^{1,0}(t)-\tilde{w}_{n,\epsilon}^{1,0}(0)\bigr).
\]
Using the inequality~\eqref{eq:tildew1} and the Cauchy--Schwarz inequality, one has
\begin{align*}
\big|\tilde{w}_{n,\epsilon}^{1,0}(t)-\tilde{w}_{n,\epsilon}^{1,0}(0)\big|^2&=\Big|\int_0^t(\tilde{w}_{n,\epsilon}^{1,0})'(s)\,ds\Big|^2\\
&\le t\int_0^t\big|(\tilde{w}_{n,\epsilon}^{1,0})'(s)\big|^2\,ds\\
&\le t\frac{\lambda_n^2\epsilon^2}{1-2\lambda_n\epsilon^2}\\
&\le 2t\lambda_n^2\epsilon^2,
\end{align*}
using the lower bound $1-2\lambda_n\epsilon^2\ge \frac12$ in the last step.

As a result, using the inequality $\underset{x\in\R}\sup~x^{1+2\delta}e^{-x}<\infty$, one obtains for all $t>0$
\begin{align*}
\lambda_n^{2\delta-\alpha}\big|f_n^{\epsilon,1,0}(t)-e^{-t\lambda_n}\big|^2\mathds{1}_{1-4\lambda_n\epsilon^2\ge 0}&\le Ct\lambda_n^{2+2\delta-\alpha}\epsilon^2e^{-2t\lambda_n}\mathds{1}_{1-4\lambda_n\epsilon^2\ge 0}\\
&\le Ct^{-2\delta}\lambda_n^{1-\alpha}\epsilon^2\mathds{1}_{1-4\lambda_n\epsilon^2\ge 0}\\
&\le Ct^{-2\delta}\epsilon^{2\alpha}\bigl(\lambda_n\epsilon^2\bigr)^{1-\alpha}\mathds{1}_{1-4\lambda_n\epsilon^2\ge 0}\\
&\le Ct^{-2\delta}\epsilon^{2\alpha},
\end{align*}
using the condition $4\lambda_n\epsilon^2\le 1$ in the last step.

Gathering the estimates, one thus obtains for all $t>0$ and all $\epsilon\in(0,1)$
\[
\underset{n\in\N}\sup~\Bigl(\lambda_n^{2\delta-\alpha}\big|f_n^{\epsilon,1,0}(t)-e^{-t\lambda_n}\big|^2\Bigr)\le Ct^{-2\delta}\epsilon^{2\alpha}.
\]
This concludes the proof of the inequality~\eqref{eq:lem_aux-conv-1}.

It remains to prove the inequality~\eqref{eq:lem_aux-conv-2}. Assume that $\delta\in[\frac{\alpha}{2},\frac12]$ and let $n\in\N$ and $\epsilon\in(0,1)$.

On the one hand, if $1-4\lambda_n\epsilon^2<0$, owing to the inequality~\eqref{eq:lem_fngn01-1} from Lemma~\ref{lem:fngn}, one has for all $t\ge 0$
\begin{align*}
\big|\frac{f^{\epsilon,0,1}_n(t)}{\epsilon}-e^{-t\lambda_n}\big|^2&\le \frac{2}{\epsilon^2}\big|f^{\epsilon,0,1}_n(t)\big|^2+2e^{-2t\lambda_n}\le C\frac{1}{\lambda_n\epsilon^2}e^{-\frac{t}{2\epsilon^2}}+2e^{-2t\lambda_n}.
\end{align*}
As a result, using the inequality $\underset{x\in\R}\sup~x^{2\delta}e^{-x}<\infty$, one obtains for all $t>0$
\begin{align*}
\lambda_n^{2\delta-\alpha}\big|f_n^{\epsilon,0,1}(t)-e^{-t\lambda_n}\big|^2\mathds{1}_{1-4\lambda_n\epsilon^2<0}&\le C\Bigl(\lambda_n^{2\delta-\alpha-1}t^{-2\delta}\epsilon^{4\delta-2}+\lambda_n^{-\alpha}t^{-2\delta}\Bigr)\mathds{1}_{1-4\lambda_n\epsilon^2<0}\\
&\le Ct^{-2\delta}\epsilon^{2\alpha}\Bigl(\bigl(\lambda_n\epsilon^2\bigr)^{2\delta-\alpha-1}+\bigl(\lambda_n\epsilon^2)^{-\alpha}\Bigr)\mathds{1}_{1-4\lambda_n\epsilon^2<0}\\
&\le Ct^{-2\delta}\epsilon^{2\alpha},
\end{align*}
using the conditions $1+\alpha-2\delta\ge 0$, $\alpha\ge 0$ and $4\lambda_n\epsilon^2>1$ in the last step.

On the other hand, if $1-4\lambda_n\epsilon^2\ge 0$, set
\[
F_n^{\epsilon,0,1}(t)=f_n^{\epsilon,0,1}(t)+\epsilon g_n^{\epsilon,0,1}(t).
\]
Note that owing to~\eqref{eq:fngn} for all $t\ge 0$ one has
\begin{align*}
(F_n^{\epsilon,0,1})'(t)&=(f_n^{\epsilon,0,1})'(t)+\epsilon (g_n^{\epsilon,0,1})'(t)\\
&=-\lambda_n f_n^{\epsilon,0,1}(t)\\
&=-\lambda_n F_n^{\epsilon,0,1}(t)+\epsilon \lambda_n g_n^{\epsilon,0,1}(t).
\end{align*}
Since $F_n^{\epsilon,0,1}(0)=f_n^{\epsilon,0,1}(0)+\epsilon g_n^{\epsilon,0,1}(0)=\epsilon$, one has for all $t\ge 0$
\[
F_n^{\epsilon,0,1}(t)=\epsilon e^{-\lambda_nt}+\epsilon\lambda_n\int_0^te^{-\lambda_n(t-s)}g_n^{\epsilon,0,1}(s)\,ds.
\]
Moreover, one has $g_n^{\epsilon,0,1}(s)=\epsilon (f_n^{\epsilon,0,1})'(s)$ for all $s\ge 0$, therefore integrating by parts one obtains
\begin{align*}
F_n^{\epsilon,0,1}(t)-\epsilon e^{-\lambda_nt}&=\epsilon\lambda_n\int_0^te^{-\lambda_n(t-s)}g_n^{\epsilon,0,1}(s)\,ds\\
&=\epsilon^2\lambda_n\int_0^te^{-\lambda_n(t-s)}(f_n^{\epsilon,0,1})'(s)\,ds\\
&=\epsilon^2\lambda_n\Bigl(f_n^{\epsilon,0,1}(t)-\lambda_n\int_0^te^{-\lambda_n(t-s)}f_n^{\epsilon,0,1}(s)\,ds\Bigr).
\end{align*}
As a result, using the inequality~\eqref{eq:lem_fngn01-2} from Lemma~\ref{lem:fngn}, for all $t\ge 0$ one has
\begin{align*}
\bigl|\frac{1}{\epsilon}F_n^{\epsilon,0,1}(t)-e^{-\lambda_nt}\bigr|=\frac{1}{\epsilon}\bigl|F_n^{\epsilon,0,1}(t)-\epsilon e^{-\lambda_nt}\bigr|&\le\epsilon\lambda_n\Bigl(|f_n^{\epsilon,0,1}(t)|+\lambda_n\int_0^te^{-\lambda_n(t-s)}|f_n^{\epsilon,0,1}(s)|\,ds\Bigr)\\
&\le C\epsilon^2\lambda_n e^{-\lambda_nt}+C\epsilon^2\lambda_n^2\int_0^te^{-\lambda_n(t-s)}e^{-\lambda_ns}\,ds\\
&\le C\epsilon^2\lambda_n e^{-\lambda_nt}+C\epsilon^2\lambda_n^2te^{-\lambda_nt}.
\end{align*}
Finally, using the identity $f_n^{\epsilon,0,1}(t)=F_n^{\epsilon,0,1}(t)-\epsilon g_n^{\epsilon,0,1}(t)$, the inequality above and the inequality~\eqref{eq:lem_fngn01-2} from Lemma~\ref{lem:fngn}, one obtains for all $t\ge 0$
\begin{align*}
	 \big|\frac{1}{\epsilon}f_n^{\epsilon,0,1}(t)-e^{-t\lambda_n}\big|^2&\le 2\big|\frac{1}{\epsilon}F_n^{\epsilon,0,1}(t)-e^{-t\lambda_n}\big|^2+2|g_n^{\epsilon,0,1}(t)|^2\\
	 &\le C\Bigl(\epsilon^4\lambda_n^2 e^{-2\lambda_nt}+\epsilon^4\lambda_n^4t^2e^{-2\lambda_nt}+e^{-\frac{2t}{\epsilon^2}}\Bigr).
\end{align*}
As a result, using the inequalities $\underset{x\in\R}\sup~x^{2\delta}e^{-x}<\infty$ and $\underset{x\in\R}\sup~x^{2+2\delta}e^{-x}<\infty$, one obtains for all $t>0$
\begin{align*}
\lambda_n^{2\delta-\alpha}\big|\frac{f_n^{\epsilon,0,1}(t)}{\epsilon}-e^{-t\lambda_n}\big|^2\mathds{1}_{1-4\lambda_n\epsilon^2\ge 0}&\le C\lambda_n^{2\delta-\alpha}\Bigl(\epsilon^4\lambda_n^2 e^{-2\lambda_nt}+\epsilon^4\lambda_n^4t^2e^{-2\lambda_nt}+e^{-\frac{2t}{\epsilon^2}}\Bigr)\mathds{1}_{1-4\lambda_n\epsilon^2\ge 0}\\
&\le Ct^{-2\delta}\Bigl(\epsilon^4\lambda_n^{2-\alpha}+\lambda_n^{2\delta-\alpha}\epsilon^{4\delta}\Bigr)\mathds{1}_{1-4\lambda_n\epsilon^2\ge 0}\\
&\le Ct^{-2\delta}\epsilon^{2\alpha}\Bigl(\bigl(\lambda_n\epsilon^2\bigr)^{2-\alpha}+\bigl(\lambda_n\epsilon^2\bigr)^{2\delta-\alpha}\Bigr)\mathds{1}_{1-4\lambda_n\epsilon^2\ge 0}\\
&\le Ct^{-2\delta}\epsilon^{2\alpha},
\end{align*}
using the conditions $2-\alpha\ge 0$, $2\delta-\alpha\ge 0$ and $4\lambda_n\epsilon^2\le 1$ in the last step.

Gathering the estimates, one thus obtains for all $t>0$ and all $\epsilon\in(0,1)$
\[
\underset{n\in\N}\sup~\Bigl(\lambda_n^{2\delta-\alpha}\big|\frac{f_n^{\epsilon,0,1}(t)}{\epsilon}-e^{-t\lambda_n}\big|^2\Bigr)\le Ct^{-2\delta}\epsilon^{2\alpha}.
\]
This concludes the proof of the inequality~\eqref{eq:lem_aux-conv-2}.

The proof of Lemma~\ref{lem:aux-conv} is thus completed.
\end{proof}

\begin{proof}[Proof of Lemma~\ref{lem:semigroups-convergence}]
Let us first prove the inequality~\eqref{eq:lem_semigroups-convergence-u}. Assume that $\delta\in[0,\frac{\alpha}{2}]$. Let $u\in H^{\alpha-2\delta}$, then for all $t\ge 0$ one has
\[
\Pi_u\Bigl(e^{tA_\epsilon}(u,0)\Bigr)=\sum_{n\in\N}\langle u,e_n\rangle f_n^{\epsilon,1,0}(t)e_n
\]
where $t\ge 0\mapsto (f_n^{\epsilon,1,0}(t),g_n^{\epsilon,1,0}(t))$ is the solution of~\eqref{eq:fngn} with initial values $f_n^{\epsilon,1,0}(0)=1$ and $g_n^{\epsilon,1,0}(0)=0$. As a result, one has
\begin{align*}
\|\Pi_u\Bigl(e^{tA_\epsilon}(u,0)\Bigr)-e^{-t\IL}u\|_{H}^2&=\sum_{n\in\N}\big|f_n^{\epsilon,1,0}(t)-e^{-t\lambda_n}\big|^2 \langle u,e_n\rangle^2\\
&=\sum_{n\in\N}\lambda_n^{2\delta-\alpha}\big|f_n^{\epsilon,1,0}(t)-e^{-t\lambda_n}\big|^2 \lambda_n^{\alpha-2\delta}\langle u,e_n\rangle^2\\
&\le \underset{n\in\N}\sup~\Bigl(\lambda_n^{2\delta-\alpha}\big|f_n^{\epsilon,1,0}(t)-e^{-t\lambda_n}\big|^2\Bigr)\|u\|_{H^{\alpha-2\delta}}^2.
\end{align*}
Employing the inequality~\eqref{eq:lem_aux-conv-1} from Lemma~\ref{lem:aux-conv} then yields the inequality~\eqref{eq:lem_semigroups-convergence-u}.

It remains to prove the inequality~\eqref{eq:lem_semigroups-convergence-v}. Assume that $\delta\in[\frac{\alpha}{2},\frac{1}{2}]$. Let $v\in H^{\alpha-2\delta}$, then for all $t\ge 0$ one has
\[
\Pi_u\Bigl(e^{tA_\epsilon}(0,\frac{v}{\epsilon})\Bigr)=\sum_{n\in\N}\langle v,e_n\rangle \frac{f_n^{\epsilon,0,1}(t)}{\epsilon}e_n
\]
where $t\ge 0\mapsto (f_n^{\epsilon,0,1}(t),g_n^{\epsilon,0,1}(t))$ is the solution of~\eqref{eq:fngn} with initial values $f_n^{\epsilon,0,1}(0)=0$ and $g_n^{\epsilon,0,1}(0)=1$. As a result, one has
\begin{align*}
\|\Pi_u\Bigl(e^{tA_\epsilon}(0,\frac{v}{\epsilon})\Bigr)-e^{-t\IL}v\|_{H}^2&=\sum_{n\in\N}\big|\frac{f_n^{\epsilon,0,1}(t)}{\epsilon}-e^{-t\lambda_n}\big|^2 \langle v,e_n\rangle^2\\
&=\sum_{n\in\N}\lambda_n^{2\delta-\alpha}\big|\frac{f_n^{\epsilon,0,1}(t)}{\epsilon}-e^{-t\lambda_n}\big|^2 \lambda_n^{\alpha-2\delta}\langle v,e_n\rangle^2\\
&\le \underset{n\in\N}\sup~\Bigl(\lambda_n^{2\delta-\alpha}\big|\frac{f_n^{\epsilon,0,1}(t)}{\epsilon}-e^{-t\lambda_n}\big|^2\Bigr)\|v\|_{H^{\alpha-2\delta}}^2.
\end{align*}
Employing the inequality~\eqref{eq:lem_aux-conv-2} from Lemma~\ref{lem:aux-conv} then yields the inequality~\eqref{eq:lem_semigroups-convergence-v}.

The proof of Lemma~\ref{lem:semigroups-convergence} is thus completed.
\end{proof}

\section{Properties of the solutions to the stochastic partial differential equations}\label{sec:properties}

In Section~\ref{sec:SPDE}, we have introduced stochastic evolution equations and stated well-posedness results for mild solutions. Using the properties on the semigroups presented in Sections~\ref{sec:expL} and~\ref{sec:expAeps}, it is possible to prove rigorously these results by standard fixed point approaches, the details are omitted. In this section, the aim is to provide some moment bounds and regularity properties, for which the properties given by Lemma~\ref{lem:expL} for $\bigl(e^{-t\IL})_{t\ge 0}$ and Lemmas~\ref{lem:expAeps-bound} and~\ref{lem:expAeps-smoothing1} for $\bigl(e^{tA_\epsilon}\bigr)_{t\ge 0}$ play an important role.

\subsection{Properties of the solutions to the stochastic heat equation}

Let $u_0^0$ be $\HH$-valued $\mathcal{F}_0$-measurable random variable. Recall that $\bigl(u^{0}(t)\bigr)_{t\ge 0}$ denotes the unique mild solution of~\eqref{eq:she} -- the mild formulation is given by~\eqref{eq:mild-she} -- with initial value $u^0(0)=u_0^0$.

\begin{propo}\label{propo:she-momentbounds}
Let $\beta\in(0,1]$ such that the condition~\eqref{eq:condition_beta} from Assumption~\ref{ass:noise} holds. For all $p\in[1,\infty)$, $T\in(0,\infty)$ and all $\alpha\in[0,\beta)$, there exists $C_{p,\alpha}(T)\in(0,\infty)$ such that, if the initial value $u_0^0$ is a $\mathcal{F}_0$-measurable random variable, which satisfies $\E[\|u_0^0\|_{H^\alpha}^{2p}]<\infty$, then one has
\begin{equation}\label{eq:propo-she-momentbounds}
\underset{0\le t\le T}\sup~\E[\|u^{0}(t)\|_{H^\alpha}^{2p}]\le C_{p,\alpha}(T)(1+\E[\|u_0\|_{H^{\alpha}}^{2p}]).
\end{equation}
\end{propo}

The result stated in Proposition~\ref{propo:she-momentbounds} is standard and may be found in many articles and books. For completeness and for pedagogical reasons, a short proof is given below.

\begin{proof}
We recall that the stochastic convolution $Z(t)$ is given by~\eqref{eq:Z}.

Using the It\^o isometry property, one has, for all $t\ge 0$ and all $\alpha\in[0,\beta)$,
\begin{align*}
\E[\|Z(t)\|_{H^\alpha}^2]&=\sum_{n=1}^{\infty}q_n\int_0^t\bigl\|e^{-(t-s)\IL}e_n\bigr\|_{H^\alpha}^2\,ds\\
&=\sum_{n=1}^{\infty}q_n\lambda_n^{\alpha}\int_0^t e^{-2(t-s)\lambda_n}\,ds\\
&=\frac12\sum_{n\in\N}q_n\lambda_n^{\alpha-1}(1-e^{-2t\lambda_n}).
\end{align*}
Therefore one has for all $\alpha\in[0,\beta)$
\begin{equation}\label{eq:boundZ}
\underset{t\ge 0}\sup~\E[\|Z(t)\|_{H^\alpha}^2]<\infty.
\end{equation}
Since $Z(t)$ is a Gaussian random variable for all $t\ge 0$, it is sufficient to prove moment bounds for $p=1$, and from~\eqref{eq:boundZ} one obtains for all $p\in[1,\infty)$
\begin{equation}\label{eq:boundZp}
\underset{t\ge 0}\sup~\E[\|Z(t)\|_{H^\alpha}^{2p}]<\infty.
\end{equation}

Let us first prove the inequality~\eqref{eq:propo-she-momentbounds} for $u^{0}(t)$, when $\alpha=0$.

Using the mild formulation~\eqref{eq:mild-she}, the inequality~\eqref{eq:lem_expL-bound} from Lemma~\ref{lem:expL} and the Lipschitz continuity of $f$ (Assumption~\ref{ass:fLip}, one obtains for all $t\ge 0$
\begin{align*}
   \|u^{0}(t)\|_{H}&\le \|e^{-t\IL}u_0^0\|_{H}+\int_0^t\bigl\|e^{-(t-s)\IL}f(u^{0}(s))\bigr\|_{H}\,ds+\|Z(t)\|_{H}\\
   &\le \|u_0^0\|_{H}+\int_0^t\|f(u^{0}(s))\|_{H}\,ds+\|Z(t)\|_{H}\\
   &\le \|u_0^0\|_{H}+C\int_0^t\big(1+\|u^{0}(s)\|_{H}\big)\,ds+\|Z(t)\|_{H}.
\end{align*}
Applying the Minkowski inequality and using the moment bounds~\eqref{eq:boundZp} obtained above for the stochastic convolution $Z(t)$, one obtains for all $t\ge 0$
\begin{align*}
   \big(\E[\|u^{0}(t)\|_{H}^{2p}]\big)^{\frac{1}{2p}}&\le  \big(\E[\|u_0\|_{H}^{2p}]\big)^{\frac{1}{2p}}+C\int_0^t\big(1+\big(\E[\|u^{0}(s)\|_{H}^{2p}]\big)^{\frac{1}{2p}}\big)\,ds+\big(\E[\|Z(t)\|_{H}^{2p}]\big)^{\frac{1}{2p}}\\
   &\le C_p\big(1+\big(\E[\|u_0\|_{H}^{2p}]\big)^{\frac{1}{2p}}\big)+C\int_0^t\big(\E[\|u^{0}(s)\|_{H}^{2p}]\big)^{\frac{1}{2p}}\,ds.
\end{align*}
Finally, applying the Gr\"onwall inequality, there exists $C_{p,0}(T)\in(0,\infty)$ such that for all $t\in[0,T]$ one has 
\[
\big(\E[\|u^{0}(t)\|_{H}^{2p}]\big)^{\frac{1}{2p}}\le C_{p,0}(T)\big(1+\big(\E[\|u_0^0\|_{H}^{2p}]\big)^{\frac{1}{2p}}\big).
\]
This concludes the proof of the inequality~\eqref{eq:propo-she-momentbounds} when $\alpha=0$. Let us then deal with the case $\alpha\in(0,\beta)$. Using the mild formulation~\eqref{eq:mild-she}, the inequality~\eqref{eq:lem_expL-bound} and the smoothing inequality~\eqref{eq:lem_expL-smoothing} from Lemma~\ref{lem:expL} and the Lipschitz continuity property of $f$ (Assumption~\ref{ass:fLip}), one has, for all $t\ge 0$
\begin{align*}
    \|u^{0}(t)\|_{H^{\alpha}}&\le \|e^{-t\IL} u_0^0\|_{H^{\alpha}}+\int_0^t\bigl\|e^{-(t-s)\IL}f(u^{0}(s))\bigr\|_{H^{\alpha}}\,ds+\|Z(t)\|_{H^{\alpha}}\\
    &\le \|u_0\|_{H^{\alpha}}+C\int_0^t(t-s)^{-\frac{\alpha}{2}}\|f(u^{0}(s))\|_{H}\,ds+\|Z(t)\|_{H^{\alpha}}\\
    &\le \|u_0\|_{H^{\alpha}}+C\int_0^t(t-s)^{-\frac{\alpha}{2}}\big(1+\|u^{0}(s)\|_H\big)\,ds+\|Z(t)\|_{H^{\alpha}}.
\end{align*}
Applying the Minkowski inequality and using the moment bounds~\eqref{eq:boundZp} obtained above for the stochastic convolution $Z(t)$ and the moment bounds for $\|u^{0}(t)\|_H$ proved above in the case $\alpha=0$, one obtains for all $t\in[0,T]$
\begin{align*}
    \big(\E[\|u^{0}(t)\|_{H^{\alpha}}^{2p}]\big)^{\frac{1}{2p}}&\le  \big(\E[\|u_0\|_{H^{\alpha}}^{2p}]\big)^{\frac{1}{2p}}+C\int_0^t(t-s)^{-\frac{\alpha}{2}}\big(1+ \big(\E[\|u^{0}(s)\|_{H}^{2p}]\big)^{\frac{1}{2p}}\big)\,ds+C_{p,\alpha}\\
    &\le \big(\E[\|u_0\|_{H^{\alpha}}^{2p}]\big)^{\frac{1}{2p}}+C_{p,0}(T)\bigl(1+(\E[\|u_0\|_H^{2p}])^{\frac{1}{2p}}\bigr)\int_0^t(t-s)^{-\frac{\alpha}{2}}\,ds+C_{p,\alpha}\\
    &\le C_{p,\alpha}(T)(1+(\E[\|u_0\|_{H^{\alpha}}^{2p}])^{\frac{1}{2p}}),
\end{align*}
for some $C_{p,\alpha}(T)\in(0,\infty)$. This concludes the proof of the inequality~\eqref{eq:propo-she-momentbounds} when $\alpha\in(0,\beta)$.

The proof of Proposition~\ref{propo:she-momentbounds} is thus completed.
\end{proof}

\subsection{Properties of the solutions to the stochastic damped wave equation}

For all $\epsilon\in(0,1)$ let $x_0^\epsilon=(u_0^\epsilon,v_0^\epsilon)$ be an $\HH$-valued $\mathcal{F}_0$-measurable random variable. Recall that $\bigl(X^{\epsilon}(t)\bigr)_{t\ge 0}$ denotes the unique mild solution of~\eqref{eq:sdwe1}, with the mild formulation given by~\eqref{eq:mild-sdwe}, with initial value $X^\epsilon(0)=x_0^\epsilon$.

\begin{propo}\label{propo:sdwe-momentbounds}
Let $\beta\in(0,1]$ such that the condition~\eqref{eq:condition_beta} from Assumption~\ref{ass:noise} holds. For all $p\in[1,\infty)$, $T\in(0,\infty)$ and all $\alpha\in[0,\beta)$, there exists $C_{p,\alpha}(T)\in(0,\infty)$ such that, if for all $\epsilon\in(0,1)$ the initial value $x_0^\epsilon$ is a $\mathcal{F}_0$-measurable random variable, which satisfies $\underset{\epsilon\in(0,1)}\sup~\E[\|x_0^\epsilon\|_{\HH^\alpha}^{2p}]<\infty$, then one has
\begin{equation}\label{eq:propo-sdwe-momentbounds}
\underset{\epsilon\in(0,1)}\sup~\underset{0\le t\le T}\sup~\E[\|X^{\epsilon}(t)\|_{\HH^\alpha}^{2p}]\le C_{p,\alpha}(T)\bigl(1+\underset{\epsilon\in(0,1)}\sup~\E[\|x_0^\epsilon\|_{\HH^\alpha}^{2p}]\bigr).
\end{equation}
\end{propo}

\begin{proof}
We recall that the stochastic convolution $\IZ^\epsilon(t)$ is given by~\eqref{eq:IZ}.

Using the It\^o isometry property, one has
\[
\E[\|\IZ^{\epsilon}(t)\|_{\HH^\alpha}^2]=\frac{1}{\epsilon^2}\sum_{n=1}^{\infty}q_n\int_0^t\bigl\|e^{(t-s)A_\epsilon}\bigl(0,e_n\bigr)\bigr\|_{\HH^\alpha}^2\,ds.
\]
Let $\alpha\in[0,\beta)$, and let $\delta_0\in(0,\beta-\alpha)$. Applying the inequality~\eqref{eq:lem_expAeps-smoothing1} from Lemma~\ref{lem:expAeps-smoothing1} with $\delta=\frac{1-\delta_0}{2}$ and $\rho=1$, one obtains
\begin{align*}
\bigl\|e^{(t-s)A_\epsilon}\bigl(0,e_n\bigr)\bigr\|_{\HH^\alpha}^2&\le C_{\delta_0}(t-s)^{\delta_0-1}\epsilon^2\bigl(\epsilon^{-2\delta_0}e^{-\frac{t-s}{\epsilon^2}}+1\bigr)\|e_n\|_{H^{\alpha+\delta_0-1}}^2\\
&\le C_{\delta_0}(t-s)^{\delta_0-1}\epsilon^2\bigl(\epsilon^{-2\delta_0}e^{-\frac{t-s}{\epsilon^2}}+1\bigr)\lambda_n^{\alpha+\delta_0-1}.
\end{align*}
Owing to the condition $\delta_0\in(0,\beta-\alpha)$ and to Assumption~\ref{ass:noise}, one has
\[
\sum_{n\in\N}q_n\lambda_n^{\alpha+\delta_0-1}<\infty.
\]
As a result, for all $t\ge 0$ one has
\begin{align*}
\E[\|\IZ^{\epsilon}(t)\|_{\HH^\alpha}^2]&\le C_{\alpha,\delta_0}\Bigl(\int_0^t(t-s)^{\delta_0-1}\epsilon^{-2\delta_0}e^{-\frac{t-s}{\epsilon^2}}\,ds+\int_0^t(t-s)^{\delta_0-1}\,ds\Bigr)\\
&\le C_{\alpha,\delta_0}\Bigl(\int_0^\infty r^{\delta_0-1}e^{-r}\,dr+\int_0^tr^{\delta_0-1}\,dr\Bigr).
\end{align*}
Therefore one has for all $T\in(0,\infty)$, $\alpha\in[0,\beta)$
\begin{equation}\label{eq:boundIZ}
\underset{\epsilon\in(0,1)}\sup~\underset{t\in[0,T]}\sup~\E[\|\IZ^{\epsilon}(t)\|_{H^\alpha}^2]<\infty.
\end{equation}
Since $\IZ^{\epsilon}(t)$ is a Gaussian random variable for all $t\ge 0$, it is sufficient to prove moment bounds for $p=1$, and from~\eqref{eq:boundIZ} one obtains for all $p\in[1,\infty)$
\begin{equation}\label{eq:boundIZp}
\underset{\epsilon\in(0,1)}\sup~\underset{t\in[0,T]}\sup~\E[\|\IZ^{\epsilon}(t)\|_{H^\alpha}^{2p}]<\infty.
\end{equation}

Let us first prove the inequality~\eqref{eq:propo-sdwe-momentbounds} for $X^{\epsilon}(t)$, when $\alpha=0$.

Using the mild formulation~\eqref{eq:mild-sdwe}, the inequality~\eqref{eq:lem_expAeps-bound} from Lemma~\ref{lem:expAeps-bound} and the inequality~\eqref{eq:lem_expAeps-smoothing1} from Lemma~\ref{lem:expAeps-smoothing1} with $\delta=1/2$ and $\rho=1$, one obtains for all $t\ge 0$
   \begin{align*}
       \|X^{\epsilon}(t)\|_{\HH}&\le \|e^{tA_\epsilon}P_Nx_0^\epsilon\|_{\HH}+\frac{1}{\epsilon}\int_0^t\bigl\|e^{(t-s)A_{\epsilon}}\bigl(0,f(u^{\epsilon}(s))\bigr)\bigr\|_{\HH}\,ds+\|\IZ^{\epsilon}(t)\|_{\HH}\\
       &\le \|x_0^\epsilon\|_{\HH}+C\int_0^t(t-s)^{-\frac12}\|f(u^{\epsilon}(s))\|_{H^{-1}}\,ds+\|\IZ^{\epsilon}(t)\|_{\HH}\\
       &\le \|x^\epsilon_0\|_{\HH}+C\int_0^t(t-s)^{-\frac12}(1+\|X^{\epsilon}(s)\|_{\HH})\,ds+\|\IZ^{\epsilon}(t)\|_{\HH},
   \end{align*}
using the Lipschitz continuity of $f$ in the last step, more precisely using the inequalities
\[
\|f(u)\|_{H^{-1}}\le C\|f(u)\|_{H}\le C(1+\|u\|_H)\le C(1+\|x\|_{\HH}),\quad \forall~x=(u,v)\in\HH.
\]
Applying the Minkowski inequality and using the moment bounds~\eqref{eq:boundIZp} obtained above for the stochastic convolution $\IZ^{\epsilon,N}(t)$, one obtains for all $t\in[0,T]$
   \begin{align*}
(\E[\|X^{\epsilon}(t)\|_{\HH}^{2p}])^{\frac{1}{2p}}&\le (\E[\|x_0^\epsilon\|_{\HH}^{2p}])^{\frac{1}{2p}}+C\int_0^t(t-s)^{-\frac12}\bigl(1+(\E[\|X^{\epsilon}(s)\|_{\HH}^{2p}])^{\frac{1}{2p}}\bigr)\,ds+(\E[\|\IZ^{\epsilon}(t)\|_{\HH}^{2p}])^{\frac{1}{2p}}\\
    &\le C_p(T)\bigl(1+(\E[\|x_0^\epsilon\|_{\HH}^{2p}])^{\frac{1}{2p}}\bigr)+C\int_0^t(t-s)^{-\frac12}(\E[\|X^{\epsilon}(s)\|_{\HH}^{2p}])^{\frac{1}{2p}}\,ds.
   \end{align*}
Using that inequality, the Fubini theorem, and a scaling property of the beta function, one has for all $t\ge 0$
\begin{align*}
(\E[\|X^{\epsilon}(t)\|_{\HH}^{2p}])^{\frac{1}{2p}}&\le C_p(T)\bigl(1+(\E[\|x_0^\epsilon\|_{\HH}^{2p}])^{\frac{1}{2p}}\bigr)\Bigl(1+\int_0^t(t-s)^{-\frac12}\,ds\Bigr)\\
&\quad+C\int_0^t\int_0^s(t-s)^{-\frac12}(s-r)^{-\frac12}(\E[\|X^{\epsilon}(r)\|_{\HH}^{2p}])^{\frac{1}{2p}}\,dr\,ds\\
&\le C_p(T)\bigl(1+(\E[\|x_0^\epsilon\|_{\HH}^{2p}])^{\frac{1}{2p}}\bigr)+C\int_0^t\int_r^t(t-s)^{-\frac12}(s-r)^{-\frac12}\,ds~(\E[\|X^{\epsilon}(r)\|_{\HH}^{2p}])^{\frac{1}{2p}}\,dr\\
&\le C_p(T)\bigl(1+(\E[\|x_0^\epsilon\|_{\HH}^{2p}])^{\frac{1}{2p}}\bigr)+C\int_0^t(\E[\|X^{\epsilon}(r)\|_{\HH}^{2p}])^{\frac{1}{2p}}\,dr.
\end{align*}
Finally, applying the Gr\"onwall inequality, there exists $C_{p,0}(T)\in(0,\infty)$ such that for all $\epsilon\in(0,1)$ and for all $t\in[0,T]$ one has 
\[
\big(\E[\|X^{\epsilon}(t)\|_{\HH}^{2p}]\big)^{\frac{1}{2p}}\le C_{p,0}(T)\big(1+\big(\E[\|X_0^\epsilon\|_{H}^{2p}]\big)^{\frac{1}{2p}}\big).
\]
This concludes the proof of the inequality~\eqref{eq:propo-sdwe-momentbounds} when $\alpha=0$, and let us then deal with the case $\alpha\in(0,\beta)$. Using the mild formulation~\eqref{eq:mild-sdwe}, the inequality~\eqref{eq:lem_expAeps-bound} from Lemma~\ref{lem:expAeps-bound} and the inequality~\eqref{eq:lem_expAeps-smoothing1} from Lemma~\ref{lem:expAeps-smoothing1} with $\delta=1/2$ and $\rho=1$, one obtains for all $t\ge 0$
    \begin{align*}
       \|X^{\epsilon}(t)\|_{\HH^{\alpha}}&\le \|e^{tA_\epsilon}x_0^\epsilon\|_{\HH^{\alpha}}+\frac{1}{\epsilon}\int_0^t\bigl\|e^{(t-s)A_{\epsilon}}\bigl(0,f(u^{\epsilon}(s))\bigr)\bigr\|_{\HH^{\alpha}}\,ds+\|\IZ^{\epsilon}(t)\|_{\HH^{\alpha}}\\
       &\le \|x_0^\epsilon\|_{\HH^{\alpha}}+C\int_0^t(t-s)^{-\frac12}\|f(u^{\epsilon}(s))\|_{H^{\alpha-1}}\,ds+\|\IZ^{\epsilon}(t)\|_{\HH^{\alpha}}\\
       &\le \|x_0^\epsilon\|_{\HH^{\alpha}}+C\int_0^t(t-s)^{-\frac12}(1+\|X^{\epsilon}(s)\|_{\HH})\,ds+\|\IZ^{\epsilon,N}(t)\|_{\HH^{\alpha}},
   \end{align*}
using the condition $\alpha\le 1$ to have the inequality
\[
\|f(u)\|_{H^{\alpha-1}}\le C\|f(u)\|_{H}\le C(1+\|u\|_H)\le C(1+\|x\|_{\HH}),\quad \forall~x=(u,v)\in\HH
\]
in the last step. Applying the Minkowski inequality and using the moment bounds~\eqref{eq:boundIZp} obtained above for the stochastic convolution $\IZ^{\epsilon}(t)$ and the moment bounds for $\|X^{\epsilon}(t)\|_{\mathcal H}$ proved above in the case $\alpha=0$, one obtains for all $t\in[0,T]$
\begin{align*}
    (\E[ \|X^{\epsilon}(t)\|_{\HH^{\alpha}}^{2p}])^{\frac{1}{2p}}&\le C_{p,\alpha}(T)\bigl(1+(\E[ \|x_0^\epsilon\|_{\HH^{\alpha}}^{2p}])^{\frac{1}{2p}}\bigr)+C\int_0^t(t-s)^{-\frac12}(\E[ \|X^{\epsilon}(s)\|_{\HH}^{2p}])^{\frac{1}{2p}}\,ds\\
    &\le C_{p,\alpha}(T)\bigl(1+(\E[ \|x_0^\epsilon\|_{\HH^\alpha}^{2p}])^{\frac{1}{2p}}\bigr),
\end{align*}
for some $C_{p,\alpha}(T)\in(0,\infty)$. This concludes the proof of the inequality~\eqref{eq:propo-sdwe-momentbounds} when $\alpha\in(0,\beta)$.

The proof of Proposition~\ref{propo:sdwe-momentbounds} is thus completed.
\end{proof}

\subsection{Properties of solutions to spectral Galerkin approximations}

Propositions~\ref{propo:she-momentbounds} and~\ref{propo:sdwe-momentbounds} can be extended when $u^0$ and $X^\epsilon$ are replaced by the solutions $u^{0,N}$ and $X^{\epsilon,N}$ to~\eqref{eq:she-galerkin} and~\eqref{eq:sdwe-galerkin} obtained by the application of the spectral Galerkin approximation. The bounds are uniform with respect to $N\in\N$.

\begin{propo}\label{propo:Galerkin}
Let $\beta\in(0,1]$ such that the condition~\eqref{eq:condition_beta} from Assumption~\ref{ass:noise} holds. For all $p\in[1,\infty)$, $T\in(0,\infty)$ and all $\alpha\in[0,\beta)$, there exists $C_{p,\alpha}(T)\in(0,\infty)$ such that, if the initial value $u_0^0$ is a $\mathcal{F}_0$-measurable random variable, which satisfies $\E[\|u_0^0\|_{H^\alpha}^{2p}]<\infty$, then one has
\begin{equation}\label{eq:propo-she-momentbounds-Galerkin}
\underset{N\in\N}\sup~\underset{0\le t\le T}\sup~\E[\|u^{0,N}(t)\|_{H^\alpha}^{2p}]\le C_{p,\alpha}(T)(1+\E[\|u_0\|_{H^{\alpha}}^{2p}]),
\end{equation}
and if the initial value $x_0^\epsilon$ is a $\mathcal{F}_0$-measurable random variable, which satisfies $\underset{\epsilon\in(0,1)}\sup~\E[\|x_0^\epsilon\|_{\HH^\alpha}^{2p}]<\infty$, then one has
\begin{equation}\label{eq:propo-sdwe-momentbounds-Galerkin}
\underset{N\in\N}\sup~\underset{\epsilon\in(0,1)}\sup~\underset{0\le t\le T}\sup~\E[\|X^{\epsilon,N}(t)\|_{\HH^\alpha}^{2p}]\le C_{p,\alpha}(T)\bigl(1+\underset{\epsilon\in(0,1)}\sup~\E[\|x_0^\epsilon\|_{\HH^\alpha}^{2p}]\bigr).
\end{equation}
\end{propo}

The proof of Proposition~\ref{propo:Galerkin} is omitted as the arguments are similar to those used in the proofs of Propositions~\ref{propo:she-momentbounds} and~\ref{propo:sdwe-momentbounds}.

\section{Proof of Theorem~\ref{theo:strong}}\label{sec:proof-theo-strong}

\begin{proof}[Proof of Theorem~\ref{theo:strong}]
Let $p\in[1,\infty)$ and $T\in(0,\infty)$. For all $\epsilon\in(0,1)$, owing to the mild formulation~\eqref{eq:mild-sdwe} for the solution $\bigl(X^\epsilon(t))_{t\ge 0}$ of~\eqref{eq:sdwe1}, one has for all $t\in[0,T]$
\begin{align*}
u^\epsilon(t)&=\Pi_uX^\epsilon(t)\\
&=\Pi_ue^{tA_{\epsilon}}x_0^\epsilon+\frac{1}{\epsilon}\int_0^t\Pi_ue^{(t-s)A_{\epsilon}}F(X^\epsilon(s))\,ds+\frac{1}{\epsilon}\int_0^t\Pi_ue^{(t-s)A_{\epsilon}}\,d\IW^Q(s).
\end{align*}
Recall that $F(X^\epsilon(s))=\big(0,f(u^\epsilon(s))\big)$. As a result, owing to the mild formulation~\eqref{eq:mild-she} for the solution $\bigl(u^0(t)\bigr)_{t\ge 0}$ of~\eqref{eq:she}, one obtains the following decomposition of the error $r^\epsilon(t)=u^{\epsilon}(t)-u^{0}(t)$: for all $t\ge 0$ one has
\begin{equation}\label{eq:decomp-strong}
r^\epsilon(t)=u^{\epsilon}(t)-u^{0}(t)=r_1^{\epsilon}(t)+r_2^{\epsilon}(t)+r_3^{\epsilon}(t),
\end{equation}
with the error terms defined by
\begin{align}
r_1^{\epsilon}(t)&=\Pi_ue^{tA_{\epsilon}}x_0^\epsilon-e^{-t\IL}u_0, \label{eq:decomp-strong-r1}\\
r_2^{\epsilon}(t)&=\frac{1}{\epsilon}\int_0^t\Pi_ue^{(t-s)A_{\epsilon}}(0,f(u^\epsilon(s)))\,ds-\int_0^t e^{-(t-s)\IL}f(u^0(s))\,ds, \label{eq:decomp-strong-r2}\\
r_3^{\epsilon}(t)&=\frac{1}{\epsilon}\int_0^t\Pi_ue^{(t-s)A_{\epsilon}}\,d\IW^Q(s)-\int_0^t e^{-(t-s)\IL}\,dW^Q(s).\label{eq:decomp-strong-r3}
\end{align}
We claim that the following error bounds hold: for all $\alpha\in[0,\beta)$, there exists $C_{p,\alpha}(T)\in(0,\infty)$ such that for all $\epsilon\in(0,1)$ one has
\begin{align}
\underset{0\le t\le T}\sup~\bigl(\E[\|r_{1}^{\epsilon}(t)\|_{H}^{p}]\bigr)^{\frac{1}{p}}&\le C_{p,\alpha}(T) \epsilon^\alpha,\label{eq:error-strong-r1}\\
\underset{0\le t\le T}\sup~\bigl(\E[\|r_{2}^{\epsilon}(t)\|_{H}^{p}]\bigr)^{\frac{1}{p}}&\le C_{p,\alpha}(T) \epsilon^\alpha+C\int_0^t\bigl(\E[\|r^{\epsilon}(s)\|_H^{p}]\bigr)^{\frac{1}{p}}\,ds,\label{eq:error-strong-r2}\\
\underset{0\le t\le T}\sup~\bigl(\E[\|r_{3}^{\epsilon}(t)\|_{H}^{p}]\bigr)^{\frac{1}{p}}&\le C_{p,\alpha}(T) \epsilon^\alpha.\label{eq:error-strong-r3}
\end{align}

$\bullet$ Proof of the inequality~\eqref{eq:error-strong-r1}.

The error term $r_1^{\epsilon}(t)$ defined by~\eqref{eq:decomp-strong-r1} is decomposed as
\[
r_1^{\epsilon}(t)=r_{1,1}^{\epsilon}(t)+r_{1,2}^{\epsilon}(t)+r_{1,3}^{\epsilon}(t)
\]
where for all $\epsilon\in(0,1)$ and all $t\ge 0$ one has
\begin{align*}
r_{1,1}^{\epsilon}(t)&=\Pi_ue^{tA_{\epsilon}}(u_0^\epsilon,0)-\Pi_ue^{tA_{\epsilon}}(u_0^0,0),\\
r_{1,2}^{\epsilon}(t)&=\Pi_ue^{tA_{\epsilon}}(u_0^0,0)-e^{-t\IL}u_0^0,\\
r_{1,3}^{\epsilon}(t)&=\Pi_ue^{tA_{\epsilon}}(0,v_0^\epsilon).
\end{align*}
First, owing to the inequality~\eqref{eq:lem_expAeps-bound} from Lemma~\ref{lem:expAeps-bound}, one obtains for all $t\ge 0$
\[
\|r_{1,1}^{\epsilon}(t)\|_{H}=\|\Pi_ue^{tA_{\epsilon}}(u_0^\epsilon-u_0^0,0)\|_H\le \|e^{tA_{\epsilon}}(u_0^\epsilon-u_0^0,0)\|_{\HH}\le \|(u_0^\epsilon-u_0^0,0)\|_{\HH}=\|u_0^\epsilon-u_0\|_H.
\]
Therefore using the condition~\eqref{eq:init-strong-error} from Assumption~\ref{ass:init-strong}, for all $\epsilon\in(0,1)$ one obtains
\[
\underset{t\ge 0}\sup~\bigl(\E[\|r_{1,1}^{\epsilon}(t)\|_{H}^{p}]\bigr)^{\frac{1}{p}}\le C_{p,\alpha} \epsilon^\alpha.
\]
Second, using the inequality~\eqref{eq:lem_semigroups-convergence-u} from Lemma~\ref{lem:semigroups-convergence} with $\delta=0$, one obtains for all $t\ge 0$
\[
\|r_{1,2}^{\epsilon}(t)\|_H=\|\Pi_ue^{tA_{\epsilon}}(u_0^0,0)-e^{-t\IL}u_0^0\|_H\le C_\alpha \epsilon^\alpha\|u_0^0\|_{H^\alpha}.
\]
Therefore using the condition~\eqref{eq:init-strong-bound} from Assumption~\ref{ass:init-strong}, for all $\epsilon\in(0,1)$ one obtains
\[
\underset{t\ge 0}\sup~\bigl(\E[\|r_{1,2}^{\epsilon}(t)\|_{H}^{p}]\bigr)^{\frac{1}{p}}\le C_{p,\alpha} \epsilon^\alpha.
\]
Finally, using the inequality~\eqref{eq:lem_expAeps-smoothing2-1} from Lemma~\ref{lem:expAeps-smoothing2} with $\delta=0$, one obtains for all $t\ge 0$
\[
\|r_{1,3}^{\epsilon}(t)\|_H=\|\Pi_ue^{tA_{\epsilon}}(0,v_0^\epsilon)\|_H\le C_\alpha \epsilon^\alpha\|v_0^\epsilon\|_{H^{\alpha-1}}.
\]
Therefore using the condition~\eqref{eq:init-strong-bound} from Assumption~\ref{ass:init-strong}, for all $\epsilon\in(0,1)$ one obtains
\[
\underset{t\ge 0}\sup~\bigl(\E[\|r_{1,3}^{\epsilon}(t)\|_{H}^{p}]\bigr)^{\frac{1}{p}}\le C_{p,\alpha} \epsilon^\alpha.
\]

Gathering the three estimates obtained above, for all $\epsilon\in(0,1)$ one has
\[
\underset{t\ge 0}\sup~\bigl(\E[\|r_{1}^{\epsilon}(t)\|_{H}^{p}]\bigr)^{\frac{1}{p}}\le \underset{t\ge 0}\sup~\bigl(\E[\|r_{1,1}^{\epsilon}(t)\|_{H}^{p}]\bigr)^{\frac{1}{p}}+\underset{t\ge 0}\sup~\bigl(\E[\|r_{1,2}^{\epsilon}(t)\|_{H}^{p}]\bigr)^{\frac{1}{p}}+\underset{t\ge 0}\sup~\bigl(\E[\|r_{1,3}^{\epsilon}(t)\|_{H}^{p}]\bigr)^{\frac{1}{p}}\le C_{p,\alpha} \epsilon^\alpha.
\]
This concludes the proof of the inequality~\eqref{eq:error-strong-r1}.

$\bullet$ Proof of the inequality~\eqref{eq:error-strong-r2}.

The error term $r_2^\epsilon(t)$ defined by~\eqref{eq:decomp-strong-r2} is decomposed as
\[
r_2^\epsilon(t)=r_{2,1}^\epsilon(t)+r_{2,2}^\epsilon(t),
\]
where for all $\epsilon\in(0,1)$ and all $t\ge 0$ one has
\begin{align*}
r_{2,1}^\epsilon(t)&=\int_0^t\Big(\frac{1}{\epsilon}\Pi_u e^{(t-s)A_\epsilon}(0,f(u^{\epsilon}(s)))-e^{-(t-s)\IL}f(u^{\epsilon}(s))\Big)\,ds,\\
r_{2,2}^\epsilon(t)&=\int_0^te^{-(t-s)\IL}\big(f(u^{\epsilon}(s))-f(u^{0}(s))\big)\,ds.
\end{align*}
On the one hand, using the inequality~\eqref{eq:lem_semigroups-convergence-v} from Lemma~\ref{lem:semigroups-convergence} with $\delta=\frac{\alpha}{2}$, the Lipschitz continuity of $f$ from Assumption~\ref{ass:fLip} and the inequality~\eqref{eq:Pi-bound}, one obtains
\begin{align*}
\|r_{2,1}^\epsilon(t)\|_H&\le \int_{0}^{t}\|\frac{1}{\epsilon}\Pi_u e^{(t-s)A_\epsilon}(0,f(u^{\epsilon}(s)))-e^{-(t-s)\IL}f(u^{\epsilon}(s))\|_H\,ds\\
&\le C_\alpha \epsilon^{\alpha}\int_{0}^{t}(t-s)^{-\frac{\alpha}{2}}\|f(u^\epsilon(s))\|_H\,ds\\
&\le C_\alpha \epsilon^{\alpha}\int_{0}^{t}(t-s)^{-\frac{\alpha}{2}}\bigl(1+\|u^\epsilon(s)\|_H\bigr)\,ds\\
&\le C_\alpha \epsilon^{\alpha}\int_{0}^{t}(t-s)^{-\frac{\alpha}{2}}\bigl(1+\|X^\epsilon(s)\|_{\HH}\bigr)\,ds.
\end{align*}
Using the Minkowski inequality, the moment bounds~\eqref{eq:propo-sdwe-momentbounds} from Proposition~\ref{propo:sdwe-momentbounds} and the condition~\eqref{eq:init-strong-bound} from Assumption~\ref{ass:init-strong}, there exists $C_\alpha(T)\in(0,\infty)$ such that for all $t\in[0,T]$ and all $\epsilon\in(0,1)$ one has
\begin{align*}
\bigl(\E[\|r_{2,1}^\epsilon(t)\|_H^{p}]\bigr)^{\frac{1}{p}}&\le C_{p,\alpha}(T) \epsilon^{\alpha}\Bigl(1+\bigl(\E[\|x_0^\epsilon\|_{\HH}^{p}]\bigr)^{\frac{1}{p}}\Bigr)\int_{0}^{t}(t-s)^{-\frac{\alpha}{2}}\,ds\\
&\le C_{p,\alpha}(T)\epsilon^{\alpha}.
\end{align*}

On the other hand, using the inequality~\eqref{eq:lem_expL-bound} from Lemma~\ref{lem:expL} and the Lipschitz continuity of $f$ from Assumption~\ref{ass:fLip}, one obtains for all $t\ge 0$ and all $\epsilon\in(0,1)$
\begin{align*}
\|r_{2,2}^\epsilon(t)\|_H&\le \int_0^t\big\|e^{-(t-s)\IL}\big(f(u^{\epsilon}(s))-f(u^{0}(s))\big)\big\|_H\,ds\\
    &\le C\int_0^t\|u^{\epsilon}(s)-u^{0}(s)\|_H\,ds\\
    &\le C\int_0^t\|r^{\epsilon}(s)\|_H\,ds.
\end{align*}
Using the Minkowski inequality, one then has for all $t\ge 0$ and all $\epsilon\in(0,1)$
\[
\bigl(\E[\|r_{2,2}^\epsilon(t)\|_H^{p}]\bigr)^{\frac{1}{p}}\le C\int_0^t\bigl(\E[\|r^{\epsilon}(s)\|_H^p]\bigr)^{\frac{1}{p}}\,ds.
\]
Gathering the two estimates obtained above, for all $\epsilon\in(0,1)$, one has
\begin{align*}
\underset{t\in[0,T]}\sup~\bigl(\E[\|r_{2}^{\epsilon}(t)\|_{H}^{p}]\bigr)^{\frac{1}{p}}&\le \underset{t\in [0,T]}\sup~\bigl(\E[\|r_{2,1}^{\epsilon}(t)\|_{H}^{p}]\bigr)^{\frac{1}{p}}+\underset{t\in[0,T]}\sup~\bigl(\E[\|r_{2,2}^{\epsilon}(t)\|_{H}^{p}]\bigr)^{\frac{1}{p}}\\
&\le C_{p,\alpha}(T) \epsilon^\alpha+C\int_0^t\bigl(\E[\|r^{\epsilon}(s)\|_H^{p}]\bigr)^{\frac{1}{p}}\,ds.
\end{align*}
This concludes the proof of the inequality~\eqref{eq:error-strong-r2}.

$\bullet$ Proof of the inequality~\eqref{eq:error-strong-r3}.

Observe that $r_{3}^{\epsilon}(t)$ is a Gaussian random variable for all $t\ge 0$. Therefore it is sufficient to prove the inequality~\eqref{eq:error-strong-r3} when $p=1$.

Using the expressions~\eqref{eq:WQ} and~\eqref{eq:IWQ} for $W^Q(s)$ and $\IW^Q(s)$ respectively, the error term $r_3^\epsilon(t)$ defined by~\eqref{eq:decomp-strong-r3} is written as
\begin{align*}
r_3^{\epsilon}(t)&=\frac{1}{\epsilon}\int_0^t\Pi_ue^{(t-s)A_{\epsilon}}\,d\IW^Q(s)-\int_0^t e^{-(t-s)\IL}\,dW^Q(s)\\
&=\sum_{n\in\N}\sqrt{q_n}\int_0^t\Bigl(\Pi_ue^{(t-s)A_{\epsilon}}(0,\frac{e_n}{\epsilon})-e^{-(t-s)\IL}e_n\Bigr)\,d\beta_n(s).
\end{align*}
Therefore using the It\^o isometry property, for all $t\ge 0$ and all $\epsilon\in(0,1)$ one has
\begin{align*}
\E[\|r_3^\epsilon(t)\|_H^2]&=\sum_{n\in\N}q_n\int_{0}^{t}\|\Pi_ue^{(t-s)A_{\epsilon}}(0,\frac{e_n}{\epsilon})-e^{-(t-s)\IL}e_n\|_H^2\,ds,
\end{align*}
and using then the inequality~\eqref{eq:lem_semigroups-convergence-v} from Lemma~\ref{lem:semigroups-convergence} with $\delta=\frac12-\frac{\beta-\alpha}{4}$, one obtains for all $t\in[0,T]$
\begin{align*}
\E[\|r_3^\epsilon(t)\|_H^2]&\le C_{\alpha,\delta}\epsilon^{2\alpha}\sum_{n\in\N}q_n\int_{0}^{t}(t-s)^{-2\delta}\|e_n\|_{H^{\alpha-2\delta}}^2\,ds\\
&\le C_{\alpha,\beta}\epsilon^{2\alpha}\sum_{n\in\N}q_n\lambda_n^{\frac{\alpha+\beta}{2}-1}\int_{0}^{t}(t-s)^{-1+\frac{\beta-\alpha}{2}}\,ds\\
&\le C_{\alpha,\beta}(T)\epsilon^{2\alpha},
\end{align*}
using the condition~\eqref{eq:condition_beta} from Assumption~\ref{ass:noise} in the last step.

Therefore for all $\epsilon\in(0,1)$ one has
\[
\underset{0\le t\le T}\sup~\bigl(\E[\|r_{3}^{\epsilon}(t)\|_{H}^2]\bigr)^{\frac12}\le C_{\alpha,\beta}(T) \epsilon^\alpha.
\]
This concludes the proof of the inequality~\eqref{eq:error-strong-r3}.

$\bullet$ Conclusion.

Recalling the decomposition~\eqref{eq:decomp-strong} of the error, applying the Minkowski inequality and the auxiliary error estimates~\eqref{eq:error-strong-r1},~\eqref{eq:error-strong-r2} and~\eqref{eq:error-strong-r3}, one obtains for all $\epsilon\in(0,1)$ and all $t\in[0,T]$
\begin{align*}
\bigl(\E[\|u^{\epsilon}(t)-u^{0}(t)\|_H^{p}]\bigr)^{\frac{1}{p}}=\bigl(\E[\|r^{\epsilon}(t)\|_H^{p}]\bigr)^{\frac{1}{p}}&\le \bigl(\E[\|r_1^{\epsilon}(t)\|_H^{p}]\bigr)^{\frac{1}{p}}+\bigl(\E[\|r_2^{\epsilon}(t)\|_H^{p}]\bigr)^{\frac{1}{p}}+\bigl(\E[\|r_3^{\epsilon}(t)\|_H^{p}]\bigr)^{\frac{1}{p}}\\
&\le C_{p,\alpha}(T)\epsilon^\alpha+C\int_0^t\bigl(\E[\|u^{\epsilon}(s)-u^{0}(s)\|_H^{p}]\bigr)^{\frac{1}{p}}\,ds.
\end{align*}
Finally, applying the Gr\"onwall inequality,  there exists $C_{p,\alpha}(T)\in(0,\infty)$ such that
\[
\underset{0\le t\le T}\sup~\bigl(\E[\|u^{\epsilon}(t)-u^{0}(t)\|_H^{p}]\bigr)^{\frac{1}{p}}\le C_{p,\alpha}(T) \epsilon^\alpha.
\]
The proof of Theorem~\ref{theo:strong} is thus completed.
\end{proof}

\begin{rem}
If one considers the spectral Galerkin approximation $u^{\epsilon,N}$ and $u^{0,N}$ instead of $u^{\epsilon}$ and $u^{0}$, one obtains the following variant of Theorem~\ref{theo:strong}, uniformly with respect to $N\in\N$. Let Assumptions~\ref{ass:fLip},~\ref{ass:noise} and~\ref{ass:init-strong} be satisfied. For all $p\in[1,\infty)$, $T\in(0,\infty)$ and $\alpha\in[0,\beta)$, there exists $C_{p,\alpha}(T)\in(0,\infty)$ such that for all $\epsilon\in(0,1)$ one has  
\begin{equation}\label{eq:theo-strong-N}
\underset{N\in\N}\sup~\underset{0\le t\le T}\sup~\bigl(\E[\|u^{\epsilon,N}(t)-u^{0,N}(t)\|_H^{p}]\bigr)^{\frac{1}{p}}\le C_{p,\alpha}(T) \epsilon^\alpha.
\end{equation}
The proof of the strong error estimates~\eqref{eq:theo-strong-N} follows the same arguments as the proof of~\eqref{eq:theo-strong}, using the moment bounds from Proposition~\ref{propo:Galerkin} which are uniform with respect to $N\in\N$. The details are omitted.
\end{rem}

\section{Auxiliary results for the proof of Theorem~\ref{theo:weak}}\label{sec:auxweak}

The objective of this section is to provide a series of auxiliary results which are required for the proof of Theorem~\ref{theo:weak}, see Section~\ref{sec:proof-theo-weak}. The results deal with Kolmogorov and Poisson equations associated with solutions to Galerkin approximations~\eqref{eq:she-galerkin} and~\eqref{eq:sdwe-galerkin}. Note that the regularity properties are uniform with respect to the parameter $N\in\N$. Moment bounds for Malliavin derivatives $\mathcal{D}u^{\epsilon,N}$ are also obtained.

Let us first introduce some auxiliary notation.

If $\phi:H\to\R$ is a function of class $\mathcal{C}^k$, with a bounded derivative of order $k$, set
\[
\vvvert \phi \vvvert_k=\sum_{\ell=1}^{k}\underset{u\in H}\sup~\underset{h_1,...,h_\ell\in H\setminus\{0\}}\sup~\frac{|D_u^\ell\phi(u).(h_1,...,h_\ell)|}{\|h_1\|_H ...\|h_\ell\|_H}\in[0,\infty),
\]
where $D_u^k\phi$ denotes the Fr\'echet derivative of order $k$ of $\phi$, for any integer $k\in\N$. The same notation is employed for functions $\phi:H_N\to\R$.

If $\phi:\HH_N\to\R$ is a function of class $\mathcal{C}^2$ for some integer $N\in\N$, for all $x=(u,v)\in \HH_N$ and all $h,k\in H_N$, the following expressions are used to denote partial first and second order derivatives with respect to $u$ and $v$:
\begin{align*}
&D_u\phi(u,v).h=D\phi(u,v).(h,0),~D_v\phi(u,v).h=D\phi(u,v).(0,h),\\
&D_{uu}^2\phi(u,v).(h,k)=D^2\phi(u,v).\bigl((h,0),(k,0)\bigr),~D_{vv}^2\phi(u,v).(h,k)=D^2\phi(u,v).\bigl((0,h),(0,k)\bigr),\\
&D_{uv}^2\phi(u,v).(h,k)=D^2\phi(u,v).\bigl((h,0),(0,k)\bigr),~D_{vu}^2\phi(u,v).(h,k)=D^2\phi(u,v).\bigl((0,h),(k,0)\bigr).
\end{align*}

\subsection{Regularity properties for solutions of the Kolmogorov equation associated with the stochastic heat equation}

Let $N\in\N$ be an arbitrary integer and let $\varphi:H\to\R$ be a function of class $\mathcal{C}^4$, which has bounded derivatives of order $1$ to $4$. Introduce the auxiliary function $\Phi^{0,N}:\R^+\times H_N\to\R$ defined by
\begin{equation}\label{eq:PhiN}
\Phi^{(0,N)}(t,u)=\E_{u}[\varphi(u^{0,N}(t))],
\end{equation}
for all $t\ge 0$ and all $u\in H_N$, where $\bigl(u^{0,N}(t)\bigr)_{t\ge 0}$ denotes the solution of the spectral Galerkin approximation~\eqref{eq:she-galerkin} of the semilinear stochastic heat equation~\eqref{eq:she} with initial value $u^{0,N}(0)=u$.

Let the infinitesimal generator $\mathbb{L}^{(N)}$ associated with the stochastic evolution equation~\eqref{eq:she-galerkin} be defined by
\begin{equation}\label{eq:LN}
	\mathbb{L}^{(N)}\phi(u)=D_u\phi(u).\Bigl(-\Lambda u+f_N(u)\Bigr)+\frac{1}{2}\sum_{n=1}^{N}q_n D^2_{uu}\phi(u).(e_n,e_n),
\end{equation}
for all $u\in H_N$ and for any function $\phi:H_N\to\R$ of class $\mathcal{C}^2$.

The auxiliary function $\Phi^{0,N}$ is the solution of the backward Kolmogorov equation
\begin{equation}\label{eq:kolmogorovPhiN}
\left\lbrace
\begin{aligned}
&\partial_t\Phi^{(0,N)}(t,u)=\mathbb{L}^{(N)}\Phi^{(0,N)}(t,u), \quad \forall~(t,u)\in\R^+\times H_N,\\
&\Phi^{(0,N)}(0,u)=\varphi(u),\quad \forall~u\in H_N.
\end{aligned}
\right.
\end{equation}

Let us state some regularity result for the derivatives of $\Phi^{(0,N)}(t,u)$ with respect to $u$. The bounds below are uniform with respect to $N\in\N$.
\begin{propo}\label{propo:regularity-PhiN}
Let Assumptions~\ref{ass:mapping-f} and~\ref{ass:noise} be satisfied. Let $\varphi:H\to\R$ be of class $\mathcal{C}^4$ with bounded derivatives of order $k\in\{1,2,3,4\}$.

For all $T\in(0,\infty)$, all $k\in\{1,2,3,4\}$ and all $\bigl(\alpha_\ell\bigr)_{1\le \ell\le k}$ such that $\sum_{\ell=1}^{k}\alpha_\ell<2$, there exists a real number $C_{\alpha_1,\ldots,\alpha_k}(T)\in(0,\infty)$ such that for all $N\in\N$, all $t\in(0,T]$ and all $u\in H_N$ and $\bigl(h_\ell\bigr)_{1\le \ell\le k}\in H_N^k$, one has
\begin{equation}\label{eq:propo-regularity-PhiN}
|D_u^k\Phi^{(0,N)}(t,u).(h_1,\ldots,h_k)|\le C_{\alpha_1,\ldots,\alpha_k}(T)\vvvert\varphi\vvvert_k t^{-\frac{\alpha_1+\ldots+\alpha_k}{2}}\|h_1\|_{H^{-\alpha_1}}\ldots\|h_k\|_{H^{-\alpha_k}}.
\end{equation}
\end{propo}

The proof of Proposition~\ref{propo:regularity-PhiN} is omitted as this is a standard result in the literature.

\subsection{Regularity properties for solutions of the Kolmogorov equation associated with the stochastic damped wave equation}

Let $N\in\N$ be an arbitrary integer, let $\epsilon\in(0,1)$, and let $\varphi:H\to\R$ be a function of class $\mathcal{C}^4$, which has bounded derivatives of order $1$ to $4$. Introduce the auxiliary function $\Phi^{\epsilon,N}:\R^+\times \HH_N\to\R$ defined by 
\begin{equation}\label{eq:PhiepsilonN}
\Phi^{(\epsilon,N)}(t,x)=\E_{x}[\varphi(u^{\epsilon,N}(t))],
\end{equation}
for all $t\ge 0$ and all $x\in \HH_N$, where $\bigl(X^{\epsilon,N}(t)\bigr)_{t\ge 0}$ denotes the solution of the spectral Galerkin approximation~\eqref{eq:sdwe1-galerkin} of the semilinear stochastic heat equation~\eqref{eq:sdwe1} with initial value $X^{\epsilon,N}(0)=x$, and with $u^{\epsilon,N}(t)=\Pi_uX^{\epsilon,N}(t)$ for all $t\ge 0$.

Let the infinitesimal generator $\mathcal{L}^{\epsilon,N}$ associated with the stochastic evolution equation~\eqref{eq:sdwe1-galerkin} be defined by
\begin{equation}\label{eq:LepsilonN}
	\mathcal{L}^{\epsilon,N}\phi(u,v)=\frac{1}{\epsilon}D_u\phi(u,v).v\!+\!D_v\phi(u,v).\big(\!-\!\frac{1}{\epsilon}\Lambda u\!-\!\frac{v}{\epsilon^2}\!+\!\frac{1}{\epsilon}f_N(u)\big)\!+\!\frac{1}{2\epsilon^2}\sum_{n=1}^{N}q_nD^2_{vv}\phi(u,v).(e_n,e_n),
\end{equation}
for all $(u,v)\in\HH_N$ and any function $\phi:\HH_N\to\R$ of class $\mathcal{C}^2$. One has the decomposition
\begin{equation}\label{def:operator-dampedwave}
\mathcal{L}^{\epsilon,N}\phi(u,v)=\frac{1}{\epsilon}\mathcal{A}^{(N)}\phi(u,v)+\frac{1}{\epsilon^2}\mathcal{L}^{(N)}\phi(u,v),
\end{equation}
with the auxiliary operators $\mathcal{A}^{(N)}$ and $\mathcal{L}^{(N)}$ defined by
\begin{align}
	\mathcal{A}^{(N)}\phi(u,v)&=D_u\phi(u,v). v+D_v\phi(u,v). (-\Lambda u+f_N(u)),\label{eq:genAN}\\
	\mathcal{L}^{(N)}\phi(u,v)&=-D_v\phi(u,v). v+\frac{1}{2}\sum_{n=1}^{N}q_nD^2_{vv}\phi(u,v).(e_n,e_n).\label{eq:genLN}
\end{align}
The auxiliary function $\Phi^{(\epsilon,N)}$ is the solution of the backward Kolmogorov equation
\begin{equation}\label{eq:kolmogorovPhiepsilonN}
\left\lbrace
\begin{aligned}
&\partial_t\Phi^{(\epsilon,N)}(t,u,v)=\mathcal{L}^{\epsilon,N}\Phi^{(\epsilon,N)}(t,u,v), \quad \forall~(t,u,v)\in\R^+\times H_N\times H_N,\\
&\Phi^{(\epsilon,N)}(0,u,v)=\varphi(u),\quad \forall~(u,v)\in H_N\times H_N.
\end{aligned}
\right.
\end{equation}

The main result of this section provides some regularity results for the first and second order derivatives of $\Phi^{(\epsilon,N)}(t,u,v)$ with respect to $u$ and $v$.

\begin{propo}\label{propo:regularity-PhiepsilonN}
Let Assumptions~\ref{ass:mapping-f} and~\ref{ass:noise} be satisfied. Let $\varphi:H\to\R$ be of class $\mathcal{C}^4$ with bounded first and second order derivatives.

For all $T\in(0,\infty)$ and all $\alpha\in[0,1]$, there exists a real number $C_\alpha(T)\in(0,\infty)$ such that, for all $\epsilon\in(0,1)$, all $t\in(0,T]$, all $N\in\N$ and all $(u,v)\in \HH_N$, $h\in H_N$, one has
 	\begin{align}
		\big|D_{u}\Phi^{(\epsilon,N)}(t,u,v).h\big|&\le C_\alpha(T)\vvvert \varphi\vvvert_1 t^{-\frac{\alpha}{2}}\bigl(\epsilon^{\alpha}\|h\|_H+\|h\|_{H^{-\alpha}}\bigr),\label{eq:derivative1u-PhiepsilonN}\\ 
	 \frac{1}{\epsilon}|D_v\Phi^{(\epsilon,N)}(t,u,v).h|&\le C_\alpha(T)\vvvert \varphi\vvvert_1 t^{-\frac{\alpha}{2}}\|h\|_{H^{-\alpha}}. \label{eq:derivative1v-PhiepsilonN}
	\end{align}
Moreover, for all $\alpha_1,\alpha_2\in[0,1)$, there exists $C_{\alpha_1,\alpha_2}(T)\in(0,\infty)$ such that, for all $\epsilon\in(0,1)$, all $t\in(0,T]$, all $N\in\N$ and all $(u,v)\in \HH_N$, $h,k\in H_N$, one has
\begin{align}
	&|D^2_{uu}\Phi^{(\epsilon,N)}(t,u,v).(h,k)|\le C_{\alpha_1,\alpha_2}(T)\vvvert\varphi\vvvert_2 t^{-\frac{\alpha_1+\alpha_2}{2}}\bigl(\epsilon^{\alpha_1}\|h\|_H\!+\!\|h\|_{H^{-\alpha_1}}\bigr)\bigl(\epsilon^{\alpha_2}\|k\|_H\!+\!\|k\|_{H^{-\alpha_2}}\bigr),\label{eq:derivative2uu-PhiepsilonN}\\
	&\frac{1}{\epsilon}|D^2_{uv}\Phi^{(\epsilon,N)}(t,u,v).(h,k)|\le C_{\alpha_1,\alpha_2}(T)\vvvert\varphi\vvvert_2 t^{-\frac{\alpha_1+\alpha_2}{2}}\bigl(\epsilon^{\alpha_1}\|h\|_H\!+\!\|h\|_{H^{-\alpha_1}}\bigr)\|k\|_{H^{-\alpha_2}}, \label{eq:derivative2uv-PhiepsilonN}\\
      &\frac{1}{\epsilon}|D^2_{vu}\Phi^{(\epsilon,N)}(t,u,v).(h,k)|\le C_{\alpha_1,\alpha_2}(T)\vvvert\varphi\vvvert_2 t^{-\frac{\alpha_1+\alpha_2}{2}}\|h\|_{H^{-\alpha_1}}\bigl(\epsilon^{\alpha_2}\|k\|_H\!+\!\|k\|_{H^{-\alpha_2}}\bigr),\label{eq:derivative2vu-PhiepsilonN}\\
      &\frac{1}{\epsilon^2}|D^2_{vv}\Phi^{(\epsilon,N)}(t,u,v).(h,k)|\le C_{\alpha_1,\alpha_2}(T)\vvvert\varphi\vvvert_2 t^{-\frac{\alpha_1+\alpha_2}{2}}\|h\|_{H^{-\alpha_1}}\|k\|_{H^{-\alpha_2}}. \label{eq:derivative2vv-PhiepsilonN}
	\end{align}
\end{propo}

\begin{proof}
Let $D$ define the derivative operator on $\HH_N$: then for all $\mathbf{h}=(h_u,h_v)\in\HH_N$ and $\mathbf{k}=(k_u,k_v)\in\HH_N$, one has
\begin{align*}
D\Phi^{(\epsilon,N)}(t,u,v).\mathbf{h}&=D_{u}\Phi^{(\epsilon,N)}(t,u,v).h_u+D_{v}\Phi^{(\epsilon,N)}(t,u,v).h_v,\\
D^2\Phi^{(\epsilon,N)}(t,u,v).(\mathbf{h},\mathbf{k})&=D^2_{uu}\Phi^{(\epsilon,N)}(t,u,v).(h_u,k_u)+D^2_{uv}\Phi^{(\epsilon,N)}(t,u,v).(h_u,k_v)\\
     &+D^2_{vu}\Phi^{(\epsilon,N)}(t,u,v).(h_v,k_u)+D^2_{vv}\Phi^{(\epsilon,N)}(t,u,v).(h_v,k_v).
\end{align*}
In order to obtain the regularity estimates stated in Proposition~\ref{propo:regularity-PhiepsilonN}, it suffices to establish the following claims. For any $\alpha\in[0,1]$ and any $\epsilon\in(0,1)$, introduce the auxiliary norm $\|\cdot\|_{\epsilon,\alpha}$ defined by
\[
\|\mathbf{h}\|_{\epsilon,\alpha}=\epsilon^{\alpha}\|h_u\|_H+\|h_u\|_{H^{-\alpha}}+\epsilon \|h_v\|_{H^{-\alpha}}.
\]
Then using the notation introduced in the statement of Proposition~\ref{propo:regularity-PhiepsilonN} one has
\begin{align}
\big|D\Phi^{(\epsilon,N)}(t,u,v).\mathbf{h}\big|&\le C_\alpha(T)\vvvert \varphi\vvvert_1 t^{-\frac{\alpha}{2}}\|\mathbf{h}\|_{\epsilon,\alpha},\label{eq:claim-regulPhiepsilonN1}\\
|D^2\Phi^{(\epsilon,N)}(t,u,v).(\mathbf{h},\mathbf{k})|&\le C_{\alpha_1,\alpha_2}(T)\vvvert\varphi\vvvert_2 t^{-\frac{\alpha_1+\alpha_2}{2}}\|\mathbf{h}\|_{\epsilon,\alpha_1}\|\mathbf{k}\|_{\epsilon,\alpha_2}.\label{eq:claim-regulPhiepsilonN2}
\end{align}
Indeed, the inequalities~\eqref{eq:derivative1u-PhiepsilonN} and~\eqref{eq:derivative1v-PhiepsilonN} follow directly from~\eqref{eq:claim-regulPhiepsilonN1} with $\mathbf{h}=(h,0)$ and $\mathbf{h}=(0,h)$ respectively, whereas the inequalities~\eqref{eq:derivative2uu-PhiepsilonN},~\eqref{eq:derivative2uv-PhiepsilonN},~\eqref{eq:derivative2vu-PhiepsilonN} and~\eqref{eq:derivative2vv-PhiepsilonN} follow directly from~\eqref{eq:claim-regulPhiepsilonN2} with $\bigl(\mathbf{h},\mathbf{k}\bigr)=\bigl((h,0),(k,0)\bigr)$, $\bigl(\mathbf{h},\mathbf{k}\bigr)=\bigl((h,0),(0,k)\bigr)$, $\bigl(\mathbf{h},\mathbf{k}\bigr)=\bigl((0,h),(k,0)\bigr)$ and $\bigl(\mathbf{h},\mathbf{k}\bigr)=\bigl((0,h),(0,k)\bigr)$ respectively.

Note that using the norm $\|\cdot\|_{\epsilon,\alpha}$ introduced above, and applying the inequality~\eqref{eq:lem_expAeps-smoothing2-2} with $\alpha_1=\alpha_2=\alpha\in[0,1]$ and the inequality~\eqref{eq:lem_expAeps-smoothing2-1} with $\alpha=1$ and $\delta=\frac{\alpha}{2}$ (see Lemma~\ref{lem:expAeps-smoothing2}), one obtains the inequality
\begin{equation}\label{eq:newinequality}
\|\Pi_u e^{tA_\epsilon}\mathbf{h}\|_H\le C_\alpha t^{-\frac{\alpha}{2}}\|\mathbf{h}\|_{\epsilon,\alpha},\qquad \forall~h\in\HH_N,~\forall~t\in(0,\infty).
\end{equation}

It thus remains to prove the inequalities~\eqref{eq:claim-regulPhiepsilonN1} and~\eqref{eq:claim-regulPhiepsilonN2}.

$\bullet$ Proof of the inequality~\eqref{eq:claim-regulPhiepsilonN1}.

Using the definition~\eqref{eq:PhiepsilonN} of $\Phi^{(\epsilon,N)}(t,u,v)$, the derivative $D\Phi^{(\epsilon,N)}(t,u,v).\mathbf{h}$ can be expressed as
\[
D\Phi^{(\epsilon,N)}(t,u,v).\mathbf{h}=\E\big[D\varphi(u^{\epsilon,N}(t,u,v)).\eta_u^{\epsilon,N,\mathbf{h}}(t)\big],
\]
where $\eta_u^{\epsilon,N,\mathbf{h}}(t)=\Pi_u\eta^{\epsilon,N,\mathbf{h}}(t)$, and $\bigl(\eta^{\epsilon,N,\mathbf{h}}(t)\bigr)_{t\ge 0}$ is the solution of the first order variational equation
\begin{equation}\label{eq:deta}
\left\lbrace
\begin{aligned}
&\frac{d}{dt}\eta^{\epsilon,N,\mathbf{h}}(t)=A_\epsilon \eta^{\epsilon,N,\mathbf{h}}(t)+\frac{1}{\epsilon}DF_N(X^{N,\epsilon}(t)).\eta^{\epsilon,N,\mathbf{h}}(t),~t\ge 0,\\
&\eta^{\epsilon,N,\mathbf{h}}(0)=\mathbf{h}.
\end{aligned}
\right.
\end{equation}
The solution of~\eqref{eq:deta} satisfies the mild formulation
\[
\eta^{\epsilon,N,\mathbf{h}}(t)=e^{tA_\epsilon}\mathbf{h}+\frac{1}{\epsilon}\int_0^te^{(t-s)A_\epsilon}DF_N(X^{\epsilon,N}(s)).\eta^{\epsilon,N,\mathbf{h}}(s)\,ds.
\]
Note that owing to the definition~\eqref{eq:F} of $F(u,v)=(0,f(u))$, one has
\[
DF_N(X^{N,\epsilon}(t)).\eta^{\epsilon,N,\mathbf{h}}(t)=\bigl(0,Df_N(u^{N,\epsilon}(t)).\eta_u^{\epsilon,N,\mathbf{h}}(t)\bigr).
\]
Therefore one obtains for all $t\ge 0$
\[
\|\eta_u^{\epsilon,N,\mathbf{h}}(t)\|_{H}\le \|\Pi_u e^{tA_\epsilon}\mathbf{h}\|_H+\frac{1}{\epsilon}\int_0^t\big\|\Pi_u e^{(t-s)A_\epsilon}\big(0,Df_N(u^{\epsilon,N}(s)).\eta_u^{\epsilon,N,\mathbf{h}}(s)\big)\big\|_H\,ds.
\]

Applying the inequality~\eqref{eq:newinequality} and using the Lipschitz continuity of $f$ (Assumption~\ref{ass:fLip}), one obtains for all $t>0$
\[
\|\eta_u^{\epsilon,N,\mathbf{h}}(t)\|_{H}\le C_\alpha t^{-\frac{\alpha}{2}}\|\mathbf{h}\|_{\epsilon,\alpha}+C\int_0^t\|\eta_u^{\epsilon,N,\mathbf{h}}(s)\|_H\,ds.
\]
Applying the Gr\"onwall inequality, for all $\alpha\in[0,1]$ there exists $C_\alpha(T)\in(0,\infty)$ such that for all $t\in(0,T]$ one has
\begin{equation}\label{eq:bound-eta}
\|\eta_u^{\epsilon,N,\mathbf{h}}(t)\|_{H}\le C_\alpha(T)t^{-\frac{\alpha}{2}}\|\mathbf{h}\|_{\epsilon,\alpha}.
\end{equation}
Finally, since the first order derivative of $\varphi$ is assumed to be bounded, one obtains
\[
\big|D\Phi^{(\epsilon,N)}(t,u,v).\mathbf{h}\big|=\big|\mathbb{E}\big[D\varphi(u^{\epsilon,N}(t,u,v)).\eta_u^{\epsilon,N,\mathbf{h}}(t)\big]\big|\le C_\alpha(T)\vvvert\varphi\vvvert_1t^{-\frac{\alpha}{2}}\|\mathbf{h}\|_{\epsilon,\alpha}.
\]
This concludes the proof of the inequality~\eqref{eq:claim-regulPhiepsilonN1}.

$\bullet$ Proof of the inequality~\eqref{eq:claim-regulPhiepsilonN2}.

Using the definition~\eqref{eq:PhiepsilonN} of $\Phi^{(\epsilon,N)}(t,u,v)$, the second order derivative $D^2\Phi^{(\epsilon,N)}(t,u,v).(\mathbf{h},\mathbf{k})$ can be expressed as
\[
D^2\Phi^{(\epsilon,N)}(t,u,v).(\mathbf{h},\mathbf{k})=\E\big[D\varphi(u^{\epsilon,N}(t)).\zeta_u^{\epsilon,N,\mathbf{h},\mathbf{k}}(t)\big]+\E\big[D^2\varphi(u^{\epsilon,N}(t)).(\eta^{\epsilon,N,\mathbf{h}}_u(t),\eta^{\epsilon,N,\mathbf{k}}_u(t))\big],
\]
where $\zeta_u^{\epsilon,N,\mathbf{h},\mathbf{k}}(t)=\Pi_u\zeta^{\epsilon,N,\mathbf{h},\mathbf{k}}(t)$, and $\bigl(\zeta^{\epsilon,N,\mathbf{h},\mathbf{k}}(t)\bigr)_{t\ge 0}$ is the solution of the second order variational equation
\begin{equation}\label{eq:dzeta}
\left\lbrace
\begin{aligned}
&\frac{d}{dt}\zeta^{\epsilon,N,\mathbf{h},\mathbf{k}}(t)=A_\epsilon \zeta^{\epsilon,N,\mathbf{h},\mathbf{k}}(t)+DF_N(X^{\epsilon,N}(t)).\zeta^{\epsilon,N,\mathbf{h},\mathbf{k}}(t)+D^2F_N(X^{\epsilon,N}(t)).\big(\eta^{\epsilon,N,\mathbf{h}}(t),\eta^{\epsilon,N,\mathbf{k}}(t)\big),~t\ge 0,\\
&\zeta^{\epsilon,N,\mathbf{h},\mathbf{k}}(0)=0.
\end{aligned}
\right.
\end{equation}
The solution of~\eqref{eq:dzeta} satisfies the mild formulation
\begin{align*}
\zeta^{\epsilon,N,\mathbf{h},\mathbf{k}}(t)&=\frac{1}{\epsilon}\int_0^te^{(t-s)A_\epsilon}D^2F_N(X^{\epsilon,N}(s)).\big(\eta^{\epsilon,N,\mathbf{h}}(s),\eta^{\epsilon,N,\mathbf{k}}(s)\big)\,ds\\
&~~~+\frac{1}{\epsilon}\int_0^te^{(t-s)A_\epsilon}DF_N(X^{\epsilon,N}(s)).\zeta^{\epsilon,N,\mathbf{h},\mathbf{k}}(s)\,ds.
\end{align*}
Note that owing to the definition~\eqref{eq:F} of $F(u,v)=(0,f(u))$, one has
\[
D^2F_N(X^{\epsilon,N}(s)).\big(\eta^{\epsilon,N,\mathbf{h}}(s),\eta^{\epsilon,N,\mathbf{k}}(s)\big)=\Bigl(0,D^2f_N(u^{\epsilon,N}(s)).\big(\eta_u^{\epsilon,N,\mathbf{h}}(s),\eta_u^{\epsilon,N,\mathbf{k}}(s)\big)\Bigr).
\]
Therefore, applying the inequality~\eqref{eq:lem_expAeps-smoothing2-1} from Lemma~\ref{lem:expAeps-smoothing2} with $\alpha=1$ and $\delta=0$ and the Lipschitz continuity of $f$ (Assumption~\ref{ass:fLip}), one obtains for all $t\ge 0$
\begin{align*}
  \|\zeta_u^{\epsilon,N,\mathbf{h},\mathbf{k}}(t)\|_{H}&\le \frac{1}{\epsilon}\int_0^t\big\|\Pi_u e^{(t-s)A_\epsilon}\big(0,D^2f_N(u^{\epsilon,N}(s)).\big(\eta_u^{\epsilon,N,\mathbf{h}}(s),\eta_u^{\epsilon,N,\mathbf{k}}(s)\big)\big)\big\|_{H}\,ds\\
 	&~~~+\frac{1}{\epsilon}\int_0^t\big\|\Pi_u e^{(t-s)A_\epsilon}\big(0,Df_N(u^{\epsilon,N}(s)).\zeta_u^{\epsilon,N,\mathbf{h},\mathbf{k}}(s)\big)\big\|_H\,ds\\
 &\le C\int_0^t\big\|D^2f_N(u^{\epsilon,N}(s)).\big(\eta_u^{\epsilon,N,\mathbf{h}}(s),\eta_u^{\epsilon,N,\mathbf{k}}(s)\big)\big\|_{H}\,ds+C\int_0^t\big\|Df_N(u^{\epsilon,N}(s)).\zeta_u^{\epsilon,N,\mathbf{h},\mathbf{k}}(s)\big)\big\|_H\,ds\\
 &\le C\int_0^t\|\eta_u^{\epsilon,N,\mathbf{h}}(s)\|_H\|\eta_u^{\epsilon,N,\mathbf{k}}(s)\|_H\,ds+C\int_0^t\|\zeta_u^{\epsilon,N,\mathbf{h},\mathbf{k}}(s)\|_H\,ds.
\end{align*}
Under the condition $\alpha_1+\alpha_2<2$, one has the integrability property
\[
\int_0^t s^{-\frac{\alpha_1+\alpha_2}{2}}\,ds\le C_{\alpha_1,\alpha_2}(T),\quad \forall~t\in[0,T],
\]
thus applying the inequality~\eqref{eq:bound-eta} one  obtains for all $t\ge 0$
\[
\|\zeta_u^{\epsilon,N,\mathbf{h},\mathbf{k}}(t)\|_{H}\le C_{\alpha_1,\alpha_2}(T)\|\mathbf{h}\|_{\epsilon,\alpha_1}\|\mathbf{k}\|_{\epsilon,\alpha_2}+C\int_0^t\|\zeta_u^{\epsilon,N,\mathbf{h},\mathbf{k}}(s)\|_H\,ds.
\]
Applying the Gr\"onwall inequality, there exists $C_{\alpha_1,\alpha_2}(T)\in(0,\infty)$ such that for all $t\in[0,T]$ one has
\begin{equation}\label{eq:bound-zeta}
\|\zeta_u^{\epsilon,N,\mathbf{h},\mathbf{k}}(t)\|_{H}\le C_{\alpha_1,\alpha_2}(T)t^{-\frac{\alpha_1+\alpha_2}{2}}\|\mathbf{h}\|_{\epsilon,\alpha_1}\|\mathbf{k}\|_{\epsilon,\alpha_2}.
\end{equation}
Finally, since the first and second order derivatives of $\varphi$ are assumed to be bounded, using~\eqref{eq:bound-eta} and~\eqref{eq:bound-zeta} one obtains
\begin{align*}
\big|D^2\Phi^{(\epsilon,N)}(t,u,v).(\mathbf{h},\mathbf{k})\big|&\le \big|\E\big[D\varphi(u^{\epsilon,N}(t)).\zeta_u^{\epsilon,N,\mathbf{h},\mathbf{k}}(t)\big]\big|+\big|\E\big[D^2\varphi(u^{\epsilon,N}(t)).(\eta^{\epsilon,N,\mathbf{h}}_u(t),\eta^{\epsilon,N,\mathbf{k}}_u(t))\big]\big|\\
&\le \vvvert\varphi\vvvert_1 \E[\|\zeta_u^{\epsilon,N,\mathbf{h},\mathbf{k}}(t)\|]+\vvvert\varphi\vvvert_2 \E[\|\eta_u^{\epsilon,N,\mathbf{h}}(t)\|\|\eta_u^{\epsilon,N,\mathbf{k}}(t)\|]\\
&\le C_{\alpha_1,\alpha_2}(T)\vvvert\varphi\vvvert_2 t^{-\frac{\alpha_1+\alpha_2}{2}}\|\mathbf{h}\|_{\epsilon,\alpha_1}\|\mathbf{k}\|_{\epsilon,\alpha_2}.
\end{align*}
This concludes the proof of the inequality~\eqref{eq:claim-regulPhiepsilonN2}.

The proof of Proposition~\ref{propo:regularity-PhiepsilonN} is thus completed.
\end{proof}

\subsection{Estimates on Malliavin derivatives} 
This subsection is devoted to deriving the estimate for the Malliavin derivatives of the solution $(u^{\epsilon,N}(t))_{t\in[0,T]}$, as stated in the following lemma.

\begin{lemma}\label{lem:Malliavinderivative-moment}
Let Assumptions~\ref{ass:fLip} and \ref{ass:mapping-f} be satisfied. Then for all $h\in H$ and $\kappa\in[0,1)$, there exists a constant $C_{\kappa}(T)\in(0,\infty)$ such that for all $0\le s\le t\le T$ one has 
\begin{equation}\label{eq:lemMalliavinderivative-moment}
\E[\|\mathcal{D}^{h}_su^{\epsilon,N}(t)\|^2_H]\le C_{\kappa}(T)(t-s)^{-\kappa} \|Q^{\frac12}h\|_{H^{-\kappa}}^2. 
\end{equation}
\end{lemma}

\begin{proof}
Owing to the property~\eqref{eq:stochasticintegral-W} of the Malliavin derivative, for all $s\ge 0$ and $h\in H$ one has 
\[
\mathcal{D}^{h}_s\int_0^se^{(s-r)A_\epsilon}\mathcal{P}_Nd\IW^Q(r)=\bigl(0,P_NQ^{\frac12}h)=\IP_N(0,Q^{\frac12}h).
\]
Differentiating equation~\eqref{eq:sdwe1-galerkin} of $X^{\epsilon,N}$ in the direction $h\in H$ and using the chain rule~\eqref{eq:chainrule}, one obtains
\[
\left\lbrace
\begin{aligned}
	d\mathcal{D}^{h}_sX^{\epsilon,N}(t)&=\bigl(A_{\epsilon}	\mathcal{D}^{h}_sX^{\epsilon,N}(t)+\frac{1}{\epsilon}DF_N(X^{\epsilon,N}(t)).	\mathcal{D}^{h}_sX^{\epsilon,N}(t)\bigr)\,dt,\quad t\ge s\ge 0,\\
	\mathcal{D}^{h}_sX^{\epsilon,N}(s)&=\frac{1}{\epsilon}\IP_N(0,Q^{\frac12}h).
\end{aligned}
\right.
\]
In addition, by the chain rule one has $\mathcal{D}^{h}_su^{\epsilon,N}(t)=\Pi_u\mathcal{D}^{h}_sX^{\epsilon,N}(t)$. As a result, for all $t\ge s \ge 0$ one has
\begin{align*}
	\mathcal{D}^{h}_su^{\epsilon,N}(t)= \frac{1}{\epsilon}\Pi_ue^{(t-s)A_\epsilon}\IP_N(0,Q^{\frac12}h)+\frac{1}{\epsilon}\int_s^t\Pi_ue^{(t-r)A_\epsilon}\big(0,Df_N(u^{\epsilon,N}(r)).\mathcal{D}_s^h u^{\epsilon,N}(r)\big)\,dr.
\end{align*}
Let $\kappa\in[0,1)$, applying Lemma~\ref{lem:expAeps-smoothing2} with $(\alpha,\delta)=(1,0)$ and $(\alpha,\delta)=(1,\frac{\kappa}{2})$, respectively, and recalling that $Df$ is bounded (since $f$ is globally Lipschitz continuous, see Assumption~\ref{ass:fLip}), there exists $C_{\kappa}\in(0,\infty)$ such that for all $t\ge s\ge 0$ one has
\begin{align*}
	\E[\|\mathcal{D}^{h}_su^{\epsilon,N}(t)\|^2_H]&\le \frac{C}{\epsilon^2} \E\Bigl[\bigl\|\Pi_ue^{(t-s)A_\epsilon}(0,P_NQ^{\frac12}h)\bigr\|^2_H+\Bigl\|\int_s^t\Pi_ue^{(t-r)A_\epsilon}\big(0,Df_N(u^{\epsilon,N}(r)).\mathcal{D}_s^h u^{\epsilon,N}(r)\big)dr\Bigr\|^2_H\Bigr]\\
	&\le C_{\kappa}(t-s)^{-\kappa}\E[\|Q^{\frac12}h\|_{H^{-\kappa}}^2]+C_{\kappa}\int_s^t\E\bigl[\bigl\|Df_N(u^{\epsilon,N}(r)).\mathcal{D}_s^h u^{\epsilon,N}(r)\bigr\|_H^2\bigr]\,dr\\
	&\le C_{\kappa}(t-s)^{-\kappa} \E[\|Q^{\frac12}h\|_{H^{-\kappa}}^2]+C_{\kappa}\int_s^t\E\bigl[\bigl\|\mathcal{D}_s^h u^{\epsilon,N}(r)\bigr\|_H^2\bigr]\,dr.
\end{align*}
Applying the Gr\"onwall inequality, there exists $C_{\kappa}(T)\in(0,\infty)$  such that for all $t\in[0,T]$ one has
\[
\E[\|\mathcal{D}^{h}_su^{\epsilon,N}(t)\|^2_H]\le C_{\kappa}(T)(t-s)^{-\kappa}\E[\|Q^{\frac12}h\|_{H^{-\kappa}}^2].
\]
The proof of Lemma~\ref{lem:Malliavinderivative-moment} is thus completed.
\end{proof}

\section{Proof of Theorem~\ref{theo:weak}}\label{sec:proof-theo-weak}

The objective of this section is to provide the proof of Theorem~\ref{theo:weak}. More precisely, the main task is to establish the following result on the weak error when $\epsilon\to 0$ for the spectral Galerkin approximations $u^{\epsilon,N}$ and $u^{0,N}$ of $u^{\epsilon}$ and $u^{0}$ respectively.

\begin{theo}\label{theo:weak-N}
Let Assumptions~\ref{ass:fLip}, \ref{ass:mapping-f},~\ref{ass:noise} and \ref{ass:init-weak} be satisfied.

Then for all $T\in(0,\infty)$, $\alpha\in[0,\beta)$ (where $\beta\in(0,1]$ is given by the condition~\eqref{eq:condition_beta} from Assumption~\ref{ass:noise}), and all mapping $\varphi\colon H\to\R$ of class $\mathcal{C}^4$, with bounded derivatives up to order $4$, there exists $C_\alpha(T,\varphi)\in(0,\infty)$ such that for all $\epsilon\in(0,1)$ and all $N\in\N$ one has
\begin{equation}\label{eq:theo-weak-N}
\underset{0\le t\le T}\sup~\bigl|\E[\varphi(u^{\epsilon,N}(t))]-\E[\varphi(u^{0,N}(t))]\bigr|\le C_{\alpha}(T,\varphi)\bigl(\epsilon^{\min(2\alpha,1)}+\epsilon\lambda_N^{\max(\frac12-\alpha,0)}+\epsilon^2\lambda_N^{1-\alpha}\bigr).
\end{equation}
\end{theo}

Decomposing the weak error appearing in the left-hand side of~\eqref{eq:theo-weak} from Theorem~\ref{theo:weak} as
\begin{align*}
\E[\varphi(u^{\epsilon}(t))]-\E[\varphi(u^{0}(t))]&=\E[\varphi(u^{\epsilon,N}(t))]-\E[\varphi(u^{0,N}(t))]\\
&+\E[\varphi(u^{0,N}(t))]-\E[\varphi(u^{0}(t))]\\
&+\E[\varphi(u^{\epsilon}(t))]-\E[\varphi(u^{\epsilon,N}(t))].
\end{align*}
Theorem~\ref{theo:weak} is a consequence of Theorem~\ref{theo:weak-N} above and of the weak error estimates on the spectral Galerkin approximation stated below.

\begin{lemma}\label{lem:weakerror-0-galerkin}
For all $T\in(0,\infty)$,  $\alpha\in[0,\beta)$ (where $\beta\in(0,1]$ is given by the condition~\eqref{eq:condition_beta} from Assumption~\ref{ass:noise}), and any mapping $\varphi\colon H\to\R$ of class $\mathcal{C}^2$, with bounded derivatives up to order $2$, there exists $C_{\alpha}(T,\varphi)\in(0,\infty)$ such that for all $N\in\N$ one has
\[
\underset{t\in[0,T]}\sup~\big|\E[\varphi(u^{0,N}(t))]-\E[\varphi(u^0(t))]\big|\le C_{\alpha}(T,\varphi)\lambda_N^{-\alpha}.
\]
\end{lemma}

\begin{lemma}\label{lem:weakerror-epsilon-galerkin}
For all $T\in(0,\infty)$,  $\alpha\in[0,\beta)$ (where $\beta\in(0,1]$ is given by the condition~\eqref{eq:condition_beta} from Assumption~\ref{ass:noise}), and any mapping $\varphi\colon H\to\R$ of class $\mathcal{C}^2$, with bounded derivatives up to order $2$, there exists $C_{\alpha}(T,\varphi)\in(0,\infty)$ such that for all $N\in\N$ and $\epsilon\in(0,1)$ one has
\[
\underset{t\in[0,T]}\sup~\big|\E[\varphi(u^{\epsilon,N}(t))]-\E[\varphi(u^\epsilon(t))]\big|\le C_{\alpha}(T,\varphi)(\epsilon^{\min(2\alpha,1)}+\lambda_N^{-\alpha}).
\]
\end{lemma}

Section~\ref{sec:proof-theo-weak-Galerkin} is devoted to the proofs of Lemmas~\ref{lem:weakerror-0-galerkin} and~\ref{lem:weakerror-epsilon-galerkin}. The strategy of the proof of Theorem~\ref{theo:weak-N} is presented in Section~\ref{sec:proof-theo-weak-strategy}, with a decomposition of the weak error appearing in the left-hand side of~\eqref{eq:theo-weak-N} and the statement of upper bounds on auxiliary error terms. Section~\ref{sec:proof-theo-weak-errorterms} is the technical part which contains the proofs of the upper bounds on the auxiliary error terms. Finally, Section~\ref{sec:proof-theo-weak-final} concludes the proof of Theorem~\ref{theo:weak} by the combination of Theorem~\ref{theo:weak-N}, Lemma~\ref{lem:weakerror-0-galerkin} and Lemma~\ref{lem:weakerror-epsilon-galerkin}, for an appropriate choice of the auxiliary parameter $N$ as a function of $\epsilon$.

\subsection{Proofs of Lemmas~\ref{lem:weakerror-0-galerkin} and~\ref{lem:weakerror-epsilon-galerkin}}\label{sec:proof-theo-weak-Galerkin}

\begin{proof}[Proof of Lemma~\ref{lem:weakerror-0-galerkin}]

For all $N\in\N$, the error can be written as
\begin{align*}
 \E[\varphi(u^{0,N}(t))]-\E[\varphi(u^0(t))]=\underset{K\to\infty}\lim~\bigl(\E[\varphi(u^{0,N}(t))]-\E[\varphi(u^{0,K}(t))]\bigr).
\end{align*}
Therefore, it suffices to prove the upper bounds
\[
\underset{K\ge N}\sup~\underset{t\in[0,T]}\sup~\big|\E[\varphi(u^{0,N}(t))]-\E[\varphi(u^{0,K}(t))]\big|\le C_{\alpha}(T,\varphi)\lambda_N^{-\alpha}.
\]
Let $K\ge N$, and consider the function $\Phi^{0,K}$ defined by~\eqref{eq:PhiN}. Then for all $t\in[0,T]$ one has
\begin{equation}\label{eq:decompose-R}
\mathbb{E}[\varphi(u^{0,N}(t))]-\mathbb{E}[\varphi(u^{0,K}(t))]=\E[\Phi^{(0,K)}(0,u^{0,N}(t))]-\E[\Phi^{(0,K)}(t,u^{0,K}(0))]=\mathbf{R}^{0,N,K}_{1}(t)+\mathbf{R}^{0,N,K}_{2}(t),
\end{equation}
where the error terms $\mathbf{R}^{0,N,K}_1(t)$ and $\mathbf{R}^{0,N,K}_2(t)$ are defined by
\begin{align*}
	\mathbf{R}^{0,N,K}_{1}(t)&=\E[\Phi^{(0,K)}(t,u^{0,N}(0))]-\E[\Phi^{(0,K)}(t,u^{0,K}(0))],\\
	\mathbf{R}^{0,N,K}_{2}(t)&=\E[\Phi^{(0,K)}(0,u^{0,N}(t))]-\E[\Phi^{(0,K)}(t,u^{0,N}(0))].
\end{align*}

\textbf{Analysis of the error term $\mathbf{R}^{0,N,K}_{1}(t)$}. Using the mean value theorem and the inequality \eqref{eq:propo-regularity-PhiN} from Proposition~\ref{propo:regularity-PhiN} with $k=1$ and $\alpha_1=0$, there exists $C(T,\varphi)\in(0,\infty)$ such that for all $t\in[0,T]$ one has
\begin{align*}
 |\mathbf{R}^{0,N,K}_{1}(t)|&\le \Bigl|\E\Bigl[\int_0^1D_u\Phi^{(0,K)}\bigl(t,u^{0,K}(0)+\theta(u^{0,N}(0)-u^{0,K}(0))\bigr).(u^{0,N}(0)-u^{0,K}(0))\,d\theta\Bigr]\Bigr|\\
 &\le C(T)\vvvert\varphi\vvvert_1 \E[\|u^{0,K}(0)-u^{0,N}(0)\|_H]\\
 &\le C(T,\varphi)\E[\|(P_K-P_N)u^0_0\|_H].
\end{align*}
Note that $(P_K-P_N)u^0_0=(I-P_N)P_Ku_0^0$. Therefore applying the inequality~\eqref{eq:errorPNK}  and the the condition \eqref{eq:init-weak-bound} on $u^0_0$ from Assumption~\ref{ass:init-weak}, one obtains the following result: there exists $C_\alpha(T,\varphi)\in(0,\infty)$ such that for all $K\ge N$ one has
\begin{equation}\label{eq:decompose-R1}
\underset{t\in[0,T]}\sup~\big|\mathbf{R}^{0,N,K}_{1}(t)\big|\le C_{\alpha}(T,\varphi)\lambda_N^{-\alpha}\E[\|u_0^{0}\|_{H^{2\alpha}}]\le  C_{\alpha}(T,\varphi)\lambda_N^{-\alpha}. 
\end{equation}

\textbf{Analysis of the error term $\mathbf{R}^{0,N,K}_{2}(t)$.} Using It\^{o}'s formula, the fact that mapping $\Phi^{(0,K)}$ is solution of the Kolmogorov equation~\eqref{eq:kolmogorovPhiN}, and the definition \eqref{eq:LN} of the operators $\mathbb{L}^{(N)}$ and $\mathbb{L}^{(K)}$, one has  
\begin{align*}
 \mathbf{R}^{0,N,K}_{2}(t)&=-\int_0^t\E\big[\partial_s \Phi^{(0,K)}(t-s,u^{0,N}(s))\big]ds+\int_0^t\E\big[\mathbb{L}^{(N)}\Phi^{(0,K)}(t-s,u^{0,N}(s))\big]ds\\
 &=\int_0^t\E\big[\bigl(\mathbb{L}^{(N)}-\mathbb{L}^{(K)}\bigr)\Phi^{(0,K)}(t-s,u^{0,N}(s))\big]ds\\
 &=\mathbf{R}^{0,N,K}_{2,1}(t)+\mathbf{R}^{0,N,K}_{2,2}(t)
\end{align*}
where for all $t\in[0,T]$ the auxiliary error terms $\mathbf{R}^{0,N,K}_{2,1}(t)$ and $\mathbf{R}^{0,N,K}_{2,2}(t)$ are given by
\begin{align*}
\mathbf{R}^{0,N,K}_{2,1}(t)&=\int_0^t\E\big[\big\langle D_u\Phi^{(0,K)}(t-s,u^{0,N}(s)), f_N(u^{0,N}(s))-f_K(u^{0,N}(s))\big\rangle\big] \,ds\\
\mathbf{R}^{0,N,K}_{2,2}(t)&=-\frac12\int_0^t \E\big[\sum_{n=N+1}^Kq_n D^2_{uu}\Phi^{(0,K)}(t-s,u^{0,N}(s)).(e_n,e_n)\big]\,ds.
\end{align*}

Let us first treat the error term $\mathbf{R}^{0,N,K}_{2,1}(t)$. For all $\alpha\in(0,\beta)$, applying the inequality \eqref{eq:propo-regularity-PhiN} from Proposition~\ref{propo:regularity-PhiN} with $k=1$ and $\alpha_1=2\alpha$, the inequality~\eqref{eq:errorPNK}, the linear growth property of $f$, and the moment bounds~\eqref{eq:propo-she-momentbounds-Galerkin} from Proposition~\ref{propo:Galerkin} with $\alpha=0$, there exists $C_\alpha(T)\in(0,\infty)$ such that for all $t\in[0,T]$ one has
\begin{align*}
|\mathbf{R}^{0,N,K}_{2,1}(t)|&\le \int_0^t\bigl|\E\big[\big\langle D_u\Phi^{(0,K)}(t-s,u^{0,N}(s)), (P_N-P_K)f(u^{0,N}(s))\big\rangle\big]\bigr| \,ds\\
&\le C_{\alpha}(T)\vvvert\varphi\vvvert_1\int_0^t(t-s)^{-\alpha} \E\bigl[\|(P_K-P_N)f(u^{0,N}(s))\|_{H^{-2\alpha}}\bigr]\,ds\\
&\le C_{\alpha}(T)\vvvert\varphi\vvvert_1\lambda_N^{-\alpha}\int_0^t(t-s)^{-\alpha} \E\bigl[\|f(u^{0,N}(s))\|_{H}\bigr]\,ds\\
&\le C_{\alpha}(T)\vvvert\varphi\vvvert_1\lambda_N^{-\alpha}\sup_{s\in[0,T]}(1+\E[\|u^{0,N}(s)\|_H])\\
&\le C_{\alpha}(T,\varphi)\lambda_N^{-\alpha}.
\end{align*}
Let us then treat the error term $\mathbf{R}^{0,N,K}_{2,2}(t)$. Let $\alpha\in(0,\beta)$ and $\kappa\in(0,2(\beta-\alpha))$, Applying the inequality~\eqref{eq:propo-regularity-PhiN} from Proposition~\ref{propo:regularity-PhiN} with $k=2$, $\alpha_1=1+\alpha-\kappa$, $\alpha_2=1-\alpha$, and the  inequality \eqref{eq:condition_beta} from Assumption~\ref{ass:noise}, there exist $C_{\alpha,\kappa}(T)\in(0,\infty)$ such that for all $t\in[0,T]$ one has
\begin{align*}
|\mathbf{R}^{0,N,K}_{2,2}(t)|&\le\frac12\int_0^t \E\big[\sum_{n=N+1}^Kq_n \big|D^2_{uu}\Phi^{(0,K)}(t-s,u^{0,N}(s)).(e_n,e_n)\big|\big] \,ds\\
&\le 	C_{\alpha,\kappa}(T)\vvvert\varphi\vvvert_2\sum_{n=N+1}^Kq_n\int_0^t(t-s)^{-1+\frac{\kappa}{2}} \,ds  \|e_n\|_{H^{-1-\alpha+\kappa}}\|e_n\|_{H^{-1+\alpha}}\\
&\le C_{\alpha,\kappa}(T)\vvvert\varphi\vvvert_2 \sum_{n=N+1}^Kq_n\lambda_n^{\alpha+\frac{\kappa}{2}-1} \lambda_n^{-\alpha}\\
&\le C_{\alpha,\kappa}(T,\varphi)\lambda_N^{-\alpha}.
\end{align*}

Combining the estimates above, for all $\alpha\in(0,\beta)$ there exists $C_{\alpha}(T,\varphi)\in(0,\infty)$ such that for all $K\ge N$ one has
\begin{align}\label{eq:decompose-R2}
\underset{t\in[0,T]}\sup~\big|\mathbf{R}^{0,N,K}_{2}(t)\big|\le C_{\alpha}(T,\varphi)\lambda_N^{-\alpha}.
\end{align}

It is now straightforward to conclude. Recalling the decomposition~\eqref{eq:decompose-R} of the error and combining the inequalities~\eqref{eq:decompose-R1} and~\eqref{eq:decompose-R2} for the error terms $\mathbf{R}^{0,N,K}_{1}(t)$ and $\mathbf{R}^{0,N,K}_{2}(t)$, there exists  $C_{\alpha}(T,\varphi)\in(0,\infty)$ such that for all $K\ge N$ one has
\begin{align*}
\underset{t\in[0,T]}\sup~\big|\mathbb{E}[\varphi(u^{0,N}(t))]-\mathbb{E}[\varphi(u^{0,K}(t))]\big|\le C_{\alpha}(T,\varphi)\lambda_N^{-\alpha}.
\end{align*}
Letting $K\to\infty$, the proof of Lemma~\ref{lem:weakerror-0-galerkin} is completed.
\end{proof}

\begin{proof}[Proof of Lemma~\ref{lem:weakerror-epsilon-galerkin}]
Let $N\in\N$ and $K\ge N$. For all $T\in(0,\infty)$ and all $\epsilon\in(0,1)$, by the definition \eqref{eq:PhiepsilonN} of function $\Phi^{(\epsilon,K)}$, the error is decomposed as
\begin{equation}\label{eq:decompose-vartheta}
\E[\varphi(u^{\epsilon,N}(t))]-\E[\varphi(u^\epsilon(t))]=\mathbf{R}^{\epsilon,N,K}_1(t)+\mathbf{R}^{\epsilon,N,K}_2(t)+\mathbf{R}^{\epsilon,N,K}_3(t),
\end{equation}
where for all $t\in[0,T]$ the error terms $\mathbf{R}^{\epsilon,N,K}_1$, $\mathbf{R}^{\epsilon,N,K}_2$ and $\mathbf{R}^{\epsilon,N,K}_3$ are defined as
\begin{align*}
\mathbf{R}^{\epsilon,N,K}_1(t)&=\E[\Phi^{(\epsilon,K)}(0,u^{\epsilon,N}(t),v^{\epsilon,N}(t))]-\E[\Phi^{(\epsilon,K)}(t,u^{\epsilon,N}(0),v^{\epsilon,N}(0))],\\
\mathbf{R}^{\epsilon,N,K}_2(t)&=\E[\Phi^{(\epsilon,K)}(t,u^{\epsilon,N}(0),v^{\epsilon,N}(0))]-\E[\Phi^{(\epsilon,K)}(t,u^{\epsilon,K}(0),v^{\epsilon,K}(0))],\\
\mathbf{R}^{\epsilon,N,K}_3(t)&=\E[\varphi(u^{\epsilon,K}(t))]-\E[\varphi(u^{\epsilon}(t))].
\end{align*}
We claim that the following error estimates hold: for all $\alpha\in(0,\beta)$, there exists a constant $C_{\alpha}(T,\varphi)\in(0,\infty)$ such that for all $t\in[0,T]$, all $K\ge N$ and all $\epsilon\in(0,1)$ one has
\begin{align}
 &|\mathbf{R}^{\epsilon,N,K}_1(t)|\le C_{\alpha}(T,\varphi)\lambda_N^{-\alpha},\label{lem:vartheta-1}\\
 &|\mathbf{R}^{\epsilon,N,K}_2(t)|\le C_{\alpha}(T,\varphi)\bigl(\epsilon\lambda_K^{\max(\frac12-\alpha,0)}+\lambda_N^{-\alpha}\bigr),\label{lem:vartheta-2} \\
&|\mathbf{R}^{\epsilon,N,K}_3(t)|\le C_\alpha(T,\varphi)\lambda_K^{-\min(\frac12,\alpha)}.\label{lem:vartheta-3}
\end{align}

$\bullet$ \textit{Proof of the claim~\eqref{lem:vartheta-1}.} Using It\^{o}'s formula, the Kolmogorov equation~\eqref{eq:kolmogorovPhiepsilonN}, and the definition \eqref{eq:LepsilonN} of operator $\mathcal{L}^{\epsilon,N}$, one has
\begin{align*}
\mathbf{R}^{\epsilon,N,K}_1(t)&=-\int_0^t\E\bigl[\partial_s\Phi^{(\epsilon,K)}(t-s,u^{\epsilon,N}(s),v^{\epsilon,N}(s))\bigr]\,ds+\int_0^t\E\bigl[\mathcal{L}^{\epsilon,N}\Phi^{(\epsilon,K)}(t-s,u^{\epsilon,N}(s),v^{\epsilon,N}(s))\bigr]\,ds\\
&=\int_0^t\E\bigl[\bigl(\mathcal{L}^{\epsilon,N}-\mathcal{L}^{\epsilon,K}\bigr)\Phi^{(\epsilon,K)}(t-s,u^{\epsilon,N}(s),v^{\epsilon,N}(s))\bigr]\,ds\\
&=\mathbf{R}^{\epsilon,N,K}_{1,1}(t)+\mathbf{R}^{\epsilon,N,K}_{1,2}(t),
\end{align*}
where for all $t\in[0,T]$ the error terms $\mathbf{R}^{\epsilon,N,K}_{1,1}(t)$ and $\mathbf{R}^{\epsilon,N,K}_{1,2}(t)$ are defined as
\begin{align*}
\mathbf{R}^{\epsilon,N,K}_{1,1}(t)&=\frac{1}{\epsilon}\int_0^t\E\big[\langle D_v\Phi^{(\epsilon,K)}(t-s,u^{\epsilon,N}(s),v^{\epsilon,N}(s)), f_N(u^{\epsilon,N}(s))-f_K(u^{\epsilon,N}(s))\rangle\big]\,ds\\
\mathbf{R}^{\epsilon,N,K}_{1,2}(t)&=-\frac{1}{2\epsilon^2}\int_0^t\E\Big[\sum_{n=N+1}^Kq_nD^2_{vv}\Phi^{(\epsilon,K)}(t-s,u^{\epsilon,N}(s), v^{\epsilon,N}(s)).(e_n,e_n)\Big]\,ds.
\end{align*}

Let us first treat the error term $\mathbf{R}^{\epsilon,N,K}_{1,1}(t)$. For all $\alpha\in(0,\beta)$, using the inequality~\eqref{eq:derivative1v-PhiepsilonN} from Proposition~\ref{propo:regularity-PhiepsilonN} with $\alpha=1$, there exist $C(T), C_{\alpha}(T)\in(0,\infty)$ such that for all $t\in[0,T]$ one has
\begin{align*}
|\mathbf{R}^{\epsilon,N,K}_{1,1}(t)|&\le \frac{1}{\epsilon}\int_0^t\E\Bigl[\Bigl|\bigl\langle D_v\Phi^{(\epsilon,K)}(t-s,u^{\epsilon,N}(s),v^{\epsilon,N}(s)), (P_N-P_K)f(u^{\epsilon,N}(s))\bigr\rangle\Bigr|\Bigr]\,ds\\
&\le C(T)\vvvert\varphi\vvvert_1\int_0^t(t-s)^{-\frac12} \E\bigl[\|(P_N-P_K)f(u^{\epsilon,N}(s))\|_{H^{-1}}\bigr]\,ds\\
&\le C_{\alpha}(T)\lambda_N^{-\alpha}\vvvert\varphi\vvvert_1 \int_0^t(t-s)^{-\frac12}\E\bigl[\|f(u^{\epsilon,N}(s))\|_{H^{2\alpha-1}}\bigr]\,ds
\end{align*}

On the one hand, if $\alpha\in(0,1/2]$, one has $2\alpha-1\le 0$, and since $f$ has at most linear growth, one obtains
\[
\|f(u^{\epsilon,N}(s))\|_{H^{2\alpha-1}}\le \|f(u^{\epsilon,N}(s))\|_{H}\le C\bigl(1+\|u^{\epsilon,N}(s)\|_H\bigr).
\]
On the other hand, if $\alpha\in(1/2,1)$, one has $2\alpha-1\in(0,1]$, and applying the inequality~\eqref{eq:falpha1} from Assumption~\ref{ass:mapping-f} with $\kappa=1-\alpha$ (for which one has $2\alpha-1+\kappa=\alpha$, one obtains
\[
\|f(u^{\epsilon,N}(s))\|_{H^{2\alpha-1}}\le C_\alpha\bigl(1+\|u^{\epsilon,N}(s)\|_{H^{\alpha}}\bigr).
\]
As a result, there exists $C_\alpha(T,\varphi)\in(0,\infty)$ such that for all $t\in[0,T]$ one has
\[
|\mathbf{R}^{\epsilon,N,K}_{1,1}(t)|\le C_{\alpha}(T,\varphi)\lambda_N^{-\alpha}\bigl(1+\underset{s\in[0,T]}\sup~\E[\|X^{\epsilon,N}(s)\|_{\HH}]+\underset{s\in[0,T]}\sup~\E[\|X^{\epsilon,N}(s)\|_{\HH^{\alpha}}]\bigr).
\]
Applying the moment bounds~\eqref{eq:propo-sdwe-momentbounds} from Proposition~\ref{propo:sdwe-momentbounds}, for all $t\in[0,T]$ one has
\[
\underset{K\ge N}\sup~|\mathbf{R}^{\epsilon,N,K}_{1,1}(t)|\le C_{\alpha}(T,\varphi)\lambda_N^{-\alpha}.
\]

Let us next treat the error term $\mathbf{R}^{\epsilon,N,K}_{1,2}(t)$. For all $\alpha\in(0,\beta)$. Applying the inequality \eqref{eq:derivative2vv-PhiepsilonN} from Proposition~\ref{propo:regularity-PhiepsilonN} with $k=2$ and $\alpha_1=\alpha_2=1-\frac{\beta-\alpha}{2}$, there exists $C_{\alpha}(T)\in(0,\infty)$ such that for all $t\in[0,T]$ one has
\begin{align*}
 |\mathbf{R}^{\epsilon,N,K}_{1,2}(t)|&\le  \frac{1}{2\epsilon^2}\int_0^t\E\Big[\sum_{n=N+1}^Kq_n\big|D^2_{vv}\Phi^{(\epsilon,K)}(t-s,u^{\epsilon,N}(s), v^{\epsilon,N}(s)).(e_n,e_n)\big|\Big]ds \\
 &\le C_{\alpha}(T)\vvvert\varphi\vvvert_2\int_0^t(t-s)^{-1+\frac{\beta-\alpha}{2}}\,ds  \sum_{n=N+1}^Kq_n \|e_n\|_{H^{-1+\frac{\beta-\alpha}{2}}}^2 \\
 &\le C_{\alpha}(T)\vvvert\varphi\vvvert_2\sum_{n=N+1}^Kq_n\lambda_n^{-1+\frac{\beta-\alpha}{2}}.
\end{align*} 
Moreover, since the sequence $\bigl(\lambda_n\bigr)_{n\in\N}$ is non-decreasing, one obtains
\[
\sum_{n=N+1}^Kq_n\lambda_n^{-1+\frac{\beta-\alpha}{2}}\le \lambda_N^{-\alpha}\sum_{n=N+1}^Kq_n\lambda_n^{-1+\frac{\beta+\alpha}{2}}\le C_\alpha \lambda_N^{-\alpha}.
\]
As a result, there exists $C_\alpha(T,\varphi)\in(0,\infty)$ such that for all $t\in[0,T]$ one has
\[
\underset{K\ge N}\sup~|\mathbf{R}^{\epsilon,N,K}_{1,2}(t)|\le C_{\alpha}(T,\varphi)\lambda_N^{-\alpha}.
\]
Combining the upper bounds for the error terms $\mathbf{R}^{\epsilon,N,K}_{1,1}(t)$ and $\mathbf{R}^{\epsilon,N,K}_{1,2}(t)$ concludes the proof of the inequality~\eqref{lem:vartheta-1} for the error term $\mathbf{R}^{\epsilon,N,K}_{1}(t)=\mathbf{R}^{\epsilon,N,K}_{1,1}(t)+\mathbf{R}^{\epsilon,N,K}_{1,2}(t)$.

$\bullet$ \textit{Proof of the claim~\eqref{lem:vartheta-2}.} Using Taylor's formula and the inequalities~\eqref{eq:derivative1u-PhiepsilonN}  and~\eqref{eq:derivative1v-PhiepsilonN} from Proposition~\ref{propo:regularity-PhiepsilonN} with $\alpha=0$, one has
\begin{align*}
|\mathbf{R}^{\epsilon,N,K}_2(t)|&\le \big|\E[\Phi^{(\epsilon,K)}(t,u^{\epsilon,N}(0),v^{\epsilon,N}(0))]-\E[\Phi^{(\epsilon,K)}(t,u^{\epsilon,K}(0),v^{\epsilon,N}(0))]\big|\\
 &\quad+\big|\E[\Phi^{(\epsilon,K)}(t,u^{\epsilon,K}(0),v^{\epsilon,N}(0))]-\E[\Phi^{(\epsilon,K)}(t,u^{\epsilon,K}(0),v^{\epsilon,K}(0))]\big|\\
 &\le C(T,\varphi)\bigl(\E\big[\|(P_K\!-\!P_N)u^{\epsilon}_0\|_H\big]+\epsilon\E\big[\|(P_K\!-\!P_N)v^{\epsilon}_0\|_{H}\big]\bigr).
\end{align*}
For the first term in the right-hand side above, applying the inequality~\eqref{eq:errorPNK} and owing to condition~\eqref{eq:init-weak-bound} from Assumption~\ref{ass:init-weak}, one has
\[
\E\big[\|(P_K\!-\!P_N)u^{\epsilon}_0\|_H\big]\le \lambda_N^{-\alpha}\E[\|u_0^\epsilon\|_{H^{2\alpha}}]\le C_\alpha\lambda_{N}^{-\alpha}.
\]
For the second term in the right-hand side above,
one needs to treat separately the cases $\alpha\in(0,1/2)$ and $\alpha\in[1/2,1)$. On the one hand, if $\alpha\in(0,1/2)$, using the inverse inequality~\eqref{eq:inverseinequality}, one obtains the upper bound
\[
\|(P_K\!-\!P_N)v^{\epsilon}_0\|_{H}\le \|P_K v^{\epsilon}_0\|_{H} \le \lambda_K^{\frac12-\alpha}\|v^\epsilon_0\|_{H^{2\alpha-1}}\le \lambda_K^{\frac12-\alpha}\|x^\epsilon_0\|_{\HH^{2\alpha}}.
\]
Therefore using the condition~\eqref{eq:init-weak-bound} from Assumption~\ref{ass:init-weak}, there exists $C_\alpha\in(0,\infty)$ such that for all $t\in[0,T]$ one has
\[
\E\big[\|(P_K\!-\!P_N)v^{\epsilon}_0\|_{H}\big]\le C_\alpha\lambda_K^{\frac12-\alpha}.
\]
On the other hand, if $\alpha\in[1/2,1)$, one has $2\alpha\ge 1$, therefore for all $t\in[0,T]$ one has
\[
\E\big[\|(P_K\!-\!P_N)v^{\epsilon}_0\|_{H}\big]\le \E[\|v^{\epsilon}_0\|_{H}]\le \E[\|x^\epsilon_0\|_{\HH^1}]\le C_\alpha\E[\|x^\epsilon_0\|_{\HH^{2\alpha}}]\le C_\alpha.
\]

This concludes the proof of the inequality~\eqref{lem:vartheta-2}.

$\bullet$ \textit{Proof of the claim~\eqref{lem:vartheta-3}.} 
From the expressions \eqref{eq:mild-sdwe-galerkin} of solution $X^{\epsilon,K}$ and \eqref{eq:mild-sdwe} of solution $X^{\epsilon}$, and identities $u^{\epsilon,K}=\Pi_uX^{\epsilon,K}$ and $u^{\epsilon}=\Pi_uX^{\epsilon}$, it follows that
\begin{align}\label{eq:decompose-u-epsilon,N}
u^{\epsilon,K}(t)-u^{\epsilon}(t)=Y^{\epsilon,K}(t)+Z^{\epsilon,K}(t)
\end{align}
with the error terms defined by
\begin{align*}
Y^{\epsilon,K}(t)&=\Pi_ue^{tA_{\epsilon}}(\IP_K-I)x_0^\epsilon+\frac{1}{\epsilon}\int_0^t\Pi_ue^{(t-s)A_{\epsilon}}\big(F_K(X^{\epsilon,K}(s))-F(X^{\epsilon,K}(s))\big)ds\\
&~~\quad+\frac{1}{\epsilon}\int_0^t\Pi_ue^{(t-s)A_{\epsilon}}\big(F(X^{\epsilon,K}(s))-F(X^{\epsilon}(s))\big)\,ds,\\
Z^{\epsilon,K}(t)&=\frac{1}{\epsilon}\int_0^t\Pi_ue^{(t-s)A_{\epsilon}}(\IP_K-I)d\IW^Q(s).
\end{align*}

Let us first deal with $Z^{\epsilon,K}(t)$. Let $\kappa=\beta-\alpha$. Applying the It\^{o} isometry formula, the inequality~\eqref{eq:lem_expAeps-smoothing2-1} from Lemma~\ref{lem:expAeps-smoothing2} with $\alpha=1,\delta=\frac{1- \frac{\kappa}{2}}{2}$, and the condition~\eqref{eq:condition_beta}  in Assumption~\ref{ass:noise}, one has
\begin{align*}
   \E[\|Z^{\epsilon,K}(t)\|_{H}^2]&\le \frac{1}{\epsilon^2}\int_0^t\sum_{n=K+1}^{\infty}q_n\big\|\Pi_ue^{(t-s)A_\epsilon}(0,e_n)\big\|_H^2\,ds \\
   &\le C_{\kappa}(T)\int_0^t(t-s)^{-1+ \frac{\kappa}{2}}ds \sum_{n=K+1}^{\infty}q_n\|e_n\|_{H^{-1+\frac{\kappa}{2}}}^2\\
   &\le C_{\kappa}(T)\sum_{n=K+1}^{\infty}q_n\lambda_n^{-1+\frac{\kappa}{2}}\\
   &\le C_{\kappa}(T)\lambda_K^{-\beta+\kappa} \sum_{n\ge K+1}q_n\lambda_n^{\beta-\frac{\kappa}{2}-1}\\
   &\le C_{\kappa}(T)\lambda_K^{-\alpha}.
\end{align*}

Let us next deal with the error term $Y^{\epsilon,K}(t)$. Applying the inequality~\eqref{eq:lem_expAeps-smoothing2-1} from Lemma~\ref{lem:expAeps-smoothing2} with the parameters $(\alpha,\delta)=(1,\frac12)$ and $(\alpha,\delta)=(1,0)$, respectively, one obtains
\begin{align*}
\|Y^{\epsilon,K}(t)\|_{H}&\le \|e^{tA_\epsilon}(\IP_K-I)x_0^\epsilon\|_{\HH}+\frac{1}{\epsilon}\int_0^t\big\|\Pi_ue^{(t-s)A_{\epsilon}}\big(0,(P_K-I)f(u^{\epsilon,K}(s))\big)\big\|_H\,ds\\
  &+\frac{1}{\epsilon}\int_0^t\!\big\|\Pi_ue^{(t-s)A_{\epsilon}}\big(0,f(u^{\epsilon,K}(s))-f(u^{\epsilon}(s))\big)\big\|_H\,ds \\
  &\le \|(\IP_K-I)x_0^\epsilon\|_{\HH}+C\int_0^t(t-s)^{-\frac12}\big\|(P_K-I)f(u^{\epsilon,K}(s))\big\|_{H^{-1}}\,ds\\
  &+C\int_0^t\|f(u^{\epsilon,K}(s))-f(u^{\epsilon}(s))\|_H\,ds.
\end{align*}
For the first term in the right-hand side, applying the inequality~\eqref{eq:errorPN}, for all $\alpha\in(0,\beta)$ one has
\[
\|(\IP_K-I)x_0^\epsilon\|_{\HH}\le \lambda_K^{-\alpha}\|x_0^\epsilon\|_{\HH^{2\alpha}}.
\]
For the second term in the right-hand side, applying the inequality~\eqref{eq:errorPN} and recalling that $f$ has at most linear growth, one obtains
\begin{align*}
\int_0^t(t-s)^{-\frac12}\big\|(P_K-I)f(u^{\epsilon,K}(s))\big\|_{H^{-1}}\,ds&\le \lambda_{K}^{-\frac12}\int_0^t(t-s)^{-\frac12}\big\|f(u^{\epsilon,K}(s))\big\|_{H}\,ds\\
&\le C\lambda_{K}^{-\frac12}\int_0^t(t-s)^{-\frac12}\bigl(1+\|u^{\epsilon,K}(s)\|_{H}\bigr)\,ds.
\end{align*}
For the third term in the right-hand side, since $f$ is globally Lipschitz continuous, one has
\[
\int_0^t\|f(u^{\epsilon,K}(s))-f(u^{\epsilon}(s))\|_H\,ds\le C\int_0^t\|u^{\epsilon,K}(s)-u^{\epsilon}(s)\|_H\,ds.
\]
Therefore, applying the Minkowski inequality, and using the condition~\eqref{eq:init-weak-bound} from Assumption~\ref{ass:init-weak} on the initial value $x_0^\epsilon$ and the moment bounds~\eqref{eq:propo-sdwe-momentbounds} from Proposition~\ref{propo:sdwe-momentbounds}, there exists $C(T)\in(0,\infty)$ such that for all $t\in[0,T]$ one has
\begin{align*}
  \big(\E[ \|Y^{\epsilon,K}(t)\|_{H}^2]\big)^{\frac12}
 &\le \lambda_K^{-\alpha}(\E[\|x_0^\epsilon\|_{\HH^{2\alpha}}^2])^{\frac12}+C(T)\lambda_K^{-\frac12}\bigl(1+(\E[\|x_0^\epsilon\|_{\HH}^2])^{\frac12}\bigr)\\
 &\quad +C \int_0^t\big(\E[\|u^{\epsilon,K}(s)-u^{\epsilon}(s)\|_H^2]\big)^{\frac12}\,ds.
\end{align*}
Recalling the decomposition~\eqref{eq:decompose-u-epsilon,N} of $u^{\epsilon,K}(t)-u^\epsilon(t)$ and combining the upper bounds on $Y^{\epsilon,K}(t)$ and $Z^{\epsilon,K}(t)$, there exists $C_{\alpha,\kappa}(T)\in(0,\infty)$ such that for all $t\in[0,T]$ one has
\begin{align*}
 \big(\E[ \|u^{\epsilon,K}(t)-u^{\epsilon}(t)\|_{H}^2]\big)^{\frac12}
 &\le	\big(\E[ \|Y^{\epsilon,K}(t)\|_{H}^2]\big)^{\frac12}+ 	\big(\E[\|\IZ^{\epsilon,K}(t)\|_{H}^2]\big)^{\frac12}\\
 &\le C \int_0^t\big(\E[\|u^{\epsilon,K}(s)-u^{\epsilon}(s)\|_H^2]\big)^{\frac12}\,ds+C_{\alpha}(T)\lambda_K^{-\min(\alpha,\frac12)}.
\end{align*}
Applying the  Gr\"onwall inequality then yields the inequality
\begin{align*}
   \underset{t\in[0,T]}\sup~\big(\E[\|u^{\epsilon,K}(t)-u^{\epsilon}(t)\|_{H}^2]\big)^{\frac12}
 \leq C_{\alpha}(T)\lambda_K^{-\min(\alpha,\frac12)}.
\end{align*}
Finally, since $\varphi$ is assumed to be globally Lipschitz continuous, one obtains
\[
|\E[\varphi(u^{\epsilon,K}(t))]-\E[\varphi(u^{\epsilon}(t))]|\le \vvvert\varphi\vvvert_1\big(\E[\|u^{\epsilon,K}(t)-u^{\epsilon}(t)\|_{H}^2]\big)^{\frac12}\le C_{\alpha,\kappa}(T,\varphi)\lambda_K^{-\min(\alpha,\frac12)}.
\]
This concludes the proof of the claim~\eqref{lem:vartheta-3}.

$\bullet$ \textit{Conclusion.} 
Recalling the decomposition~\eqref{eq:decompose-vartheta} of the error, applying the error estimates \eqref{lem:vartheta-1}, \eqref{lem:vartheta-2}, and \eqref{lem:vartheta-3}, one obtains for all $\epsilon\in(0,1)$, $K\ge N$ and $t\in[0,T]$
\begin{align*}
\big|\E[\varphi(u^{\epsilon,N}(t))]-\E[\varphi(u^\epsilon(t))]\big|
	&\le |\mathbf{R}^{\epsilon,N,K}_1(t)|+|\mathbf{R}^{\epsilon,N,K}_2(t)|+|\mathbf{R}^{\epsilon,N,K}_3(t)|\\
	&\le C_{\alpha}(T,\varphi)\bigl(\lambda_N^{-\alpha}+\epsilon\lambda_K^{\max(\frac12-\alpha,0)}+\lambda_K^{-\min(\alpha,\frac12)}\bigr).
\end{align*}
The cases $\alpha\in(0,1/2)$ and $\alpha\in[1/2,1)$ are treated separetely. On the one hand, if $\alpha\in(0,1/2)$, choosing the parameter $K$ such that $\lambda_K\sim\epsilon^{-2}$, one obtains
\[
\epsilon\lambda_K^{\max(\frac12-\alpha,0)}+\lambda_K^{-\min(\alpha,\frac12)}=\epsilon\lambda_K^{\frac12-\alpha}+\lambda_K^{-\alpha}\sim 2\epsilon^{2\alpha}.
\]
On the other hand, if $\alpha\in[1/2,1)$, one obtains
\[
\epsilon\lambda_K^{\max(\frac12-\alpha,0)}+\lambda_K^{-\min(\alpha,\frac12)}=\epsilon+\lambda_K^{-\frac12}= \epsilon^{\min(2\alpha,1)}+\lambda_K^{-\frac12},
\]
and it suffices to let $K\to\infty$.

The proof of Lemma~\ref{lem:weakerror-epsilon-galerkin} is completed.
\end{proof}

\subsection{Strategy of the proof of Theorem~\ref{theo:weak-N}}\label{sec:proof-theo-weak-strategy}

The objective of this section is to present the strategy of the proof of Theorem~\ref{theo:weak-N}. We explain how to decompose the error using the tools introduced in Section~\ref{sec:auxweak}, and we state technical results on error terms, which are then established in Section~\ref{sec:proof-theo-weak-errorterms}.

Let $T\in(0,\infty)$. For all $t\in[0,T]$, $N\in\N$ and $\epsilon\in(0,1)$, set
\[
E^{\epsilon,N}(t)=\E[\varphi(u^{\epsilon,N}(t))]-\E[\varphi(u^{0,N}(t))].
\]
Recalling the definition~\eqref{eq:PhiN} of the mapping $\Phi^{0,N}$, the weak error $E^{\epsilon,N}(t)$ is written as
\[
E^{\epsilon,N}(t)=\E[\Phi^{(0,N)}(0,u^{\epsilon,N}(t))]-\E[\Phi^{(0,N)}(t,u^{0,N}(0))].
\]
Recalling that the initial values are given by $u^{\epsilon,N}(0)=P_Nu^{\epsilon}_0$ and $u^{0,N}(0)=P_Nu^0_0$, one obtains the decomposition
\begin{equation}\label{eq:decompose-E1}
E^{\epsilon,N}(t)=E_1^{\epsilon,N}(t)+E_2^{\epsilon,N}(t),
\end{equation}
where the error terms $E_1^{\epsilon,N}(t)$ and $E_2^{\epsilon,N}(t)$ are defined as
\begin{align}
    E^{\epsilon,N}_1(t)&=\E[\Phi^{(0,N)}(0,u^{\epsilon,N}(t))]-\E[\Phi^{(0,N)}(t,u^{\epsilon,N}(0))],\label{eq:defEepsN1}\\
    E^{\epsilon,N}_2(t)&=\E[\Phi^{(0,N)}(t,P_Nu^{\epsilon}_0)]-\E[\Phi^{(0,N)}(t,P_Nu_0^0)].\label{eq:defEepsN2}
\end{align}
The treatment of the second error term $E^{\epsilon,N}_2(t)$ is straightforward (see Lemma~\ref{lem:error-T1}), however the treatment of the error term $E^{\epsilon,N}_1(t)$ requires a substantial effort.

Applying the It\^o formula and the stochastic evolution system \eqref{eq:sdwe-galerkin} for $(u^{\epsilon,N},v^{\epsilon,N})$, for all $t\in[0,T]$ one has
\[
 	E^{\epsilon,N}_1(t)=-\int_0^t\E[\partial_s\Phi^{(0,N)}(t-s,u^{\epsilon,N}(s))]\,ds+\frac{1}{\epsilon}\int_0^t\E\bigl[D_u\Phi^{(0,N)}(t-s,u^{\epsilon,N}(s)). v^{\epsilon,N}(s)\bigr]\,ds.
\]
For all $(s,u,v)\in[0,t]\times H_N\times H_N$, set
\begin{align}\label{def-Theta1}
\Theta^{(N)}_1(t,s,u,v)=D_u\Phi^{(0,N)}(t-s,u).v.	
\end{align}
Applying the It\^o formula  and the definition \eqref{def:operator-dampedwave} of the operator $\mathcal{L}^{\epsilon,N}$,  one obtains
\begin{align*}
&\quad\E[\Theta^{(N)}_1(t,t,u^{\epsilon,N}(t),v^{\epsilon,N}(t))]-\E[\Theta^{(N)}_1(t,0,u^{\epsilon,N}(0),v^{\epsilon,N}(0))]\nonumber\\
&=\int_0^t\E\big[\partial_s\Theta^{(N)}_1(t,s,u^{\epsilon,N}(s),v^{\epsilon,N}(s))\big]\,ds+\int_0^t\E\big[\mathcal{L}^{\epsilon,N}\Theta^{(N)}_1(t,s,u^{\epsilon,N}(s),v^{\epsilon,N}(s))\big]\,ds\nonumber\\
	&=\int_0^t\E\big[\partial_s\Theta^{(N)}_1(t,s,u^{\epsilon,N}(s),v^{\epsilon,N}(s))\big]\,ds+\frac{1}{\epsilon}\int_0^t\E\big[\mathcal{A}^{(N)}\Theta^{(N)}_1(t,s,u^{\epsilon,N}(s),v^{\epsilon,N}(s))\big]\,ds\nonumber\\
	&~~~\quad +\frac{1}{\epsilon^2}\int_0^t\E\big[\mathcal{L}^{(N)}\Theta^{(N)}_1(t,s,u^{\epsilon,N}(s),v^{\epsilon,N}(s))\big]\,ds.
\end{align*}
Moreover,  it follows from the definition~\eqref{eq:genLN} of the operator $\mathcal L^{(N)}$ that for all $s\le t$ and all $u,v\in H_N$ one has
\begin{align}
\mathcal L^{(N)}\Theta^{(N)}_1(t,s,u,v)&= \mathcal L^{(N)}(D_u\Phi^{(0,N)}(t-s,u).v)\nonumber\\
&=-D_v(D_u\Phi^{(0,N)}(t-s,u).v).v+\frac12\sum_{n=1}^Nq_nD^2_{vv}(D_u\Phi^{(0,N)}(t-s,u).v).(e_n,e_n)\nonumber\\
&=-D_u\Phi^{(0,N)}(t-s,u).v=-\Theta^{(N)}_1(t,s,u,v).
\end{align}
Therefore, one obtains
\begin{align*}
\frac{1}{\epsilon}\int_0^t\E\bigl[D_u\Phi^{(0,N)}(t-s,u^{\epsilon,N}(s)). v^{\epsilon,N}(s)\bigr]&\,ds=\frac{1}{\epsilon}\int_0^t\E\bigl[\Theta_1^{(N)}\bigl(t,s,u^{\epsilon,N}(s),v^{\epsilon,N}(s)\bigr)\bigr]\,ds\\
&=-\frac{1}{\epsilon}\int_0^t\E\bigl[\mathcal{L}^{(N)}\Theta_1^{(N)}\bigl(t,s,u^{\epsilon,N}(s),v^{\epsilon,N}(s)\bigr)\bigr]\,ds\\
&=\epsilon\Big(\E[\Theta^{(N)}_1(t,0,u^{\epsilon,N}(0),v^{\epsilon,N}(0))]-\E[\Theta^{(N)}_1(t,t,u^{\epsilon,N}(t),v^{\epsilon,N}(t))]\Big)\\
&+\epsilon\int_0^t\E\big[\partial_s\Theta^{(N)}_1(t,s,u^{\epsilon,N}(s),v^{\epsilon,N}(s))\big]\,ds\\
&+\int_0^t\E\big[\mathcal{A}^{(N)}\Theta^{(N)}_1(t,s,u^{\epsilon,N}(s),v^{\epsilon,N}(s))\big]\,ds.
\end{align*}
As a consequence, the error term $E_1^{\epsilon,N}(t)$ is decomposed as
\[
E_1^{\epsilon,N}(t)=E_{1,1}^{\epsilon,N}(t)+E_{1,2}^{\epsilon,N}(t),
\]
where the error terms $E_{1,1}^{\epsilon,N}(t)$ and $E_{1,2}^{\epsilon,N}(t)$ are defined as
\begin{align}
E^{\epsilon,N}_{1,1}(t)&=\int_0^t\E\big[\mathcal{A}^{(N)}(D_u\Phi^{(0,N)}(t-s,u^{\epsilon,N}(s)).v^{\epsilon,N}(s))\big]\,ds-\int_0^t\!\!\E[\partial_s\Phi^{(0,N)}(t\!-\!s,u^{\epsilon,N}(s))]\,ds,\label{eq:defEepsN11}\\
 E^{\epsilon,N}_{1,2}(t)&=  \epsilon\Big(\E[D_u\Phi^{(0,N)}(t,u^{\epsilon,N}(0)).v^{\epsilon,N}(0)]\!-\!\E[D_u\Phi^{(0,N)}(0,u^{\epsilon,N}(t)).v^{\epsilon,N}(t)]\Big)\label{eq:defEepsN12}\\
 &~~~\quad +\epsilon\int_0^t\E\big[\partial_sD_u\Phi^{(0,N)}(t-s,u^{\epsilon,N}(s)).v^{\epsilon,N}(s)\big]\,ds.\nonumber
\end{align}
The second error term $E^{\epsilon,N}_{1,2}(t)$ is in a suitable form for the proof of upper bounds, however the first error term $E^{\epsilon,N}_{1,1}(t)$ needs additional work.

Recall that the mapping $\Phi^{(0,N)}$ is the solution to the Kolmogorov equation~\eqref{eq:kolmogorovPhiN}. Owing to the definitions~\eqref{eq:LN} and~\eqref{eq:genAN} of the operators $\mathbb{L}^{(N)}$ and $\mathcal{A}^{(N)}$, one has
\begin{align*}
 E^{\epsilon,N}_{1,1}(t)&=\int_0^t\E\big[\mathcal{A}^{(N)}(D_u\Phi^{(0,N)}(t-s,u^{\epsilon,N}(s)).v^{\epsilon,N}(s))-\mathbb{L}^{(N)}\Phi^{(0,N)}(t-s,u^{\epsilon,N}(s))\big]\,ds\nonumber\\
 &=\int_0^t\E\Big[D^2_{uu}\Phi^{(0,N)}(t-s,u^{\epsilon,N}(s)).(v^{\epsilon,N}(s), v^{\epsilon,N}(s))-\frac{1}{2}\sum_{n=1}^{N}q_n D^2_{uu}\Phi^{(0,N)}(t-s,u^{\epsilon,N}(s)).(e_n,e_n)\Big]\,ds\\
&=\int_0^t\E\Big[\sum_{n,m=1}^ND^2_{uu}\Phi^{(0,N)}(t-s,u^{\epsilon,N}(s)).(e_n,e_m)\times\Bigl(\langle v^{\epsilon,N}(s), e_n\rangle~\langle v^{\epsilon,N}(s), e_m\rangle-\frac{q_n}{2}\mathds{1}_{n=m}\Bigr)\Big]\,ds.
\end{align*}

For all $n,m\in\{1,\ldots,N\}$, set
\begin{equation}\label{eq:defPsi_nm}
\Psi_{n,m}(u,v)=\langle v, e_n\rangle\langle v, e_m\rangle-\frac{q_n}{2}\mathds{1}_{n=m},\qquad \forall~x=(u,v)\in H_N\times H_N.
\end{equation}
It is straightforward to check that one has $\mathcal{L}^{(N)}\Psi_{n,m}=-2\Psi_{n,m}$.

For all $(s,u,v)\in [0,t]\times H_N\times H_N$, set 
\begin{align}\label{def:Theta2}
\Theta^{(N)}_{2}(t,s,u,v)=\sum_{n,m=1}^ND^2_{uu}\Phi^{(0,N)}(t-s,u).(e_n,e_m)\times\Psi_{n,m}(u,v).
\end{align}
Then one has $\mathcal{L}^{(N)}\Theta^{(N)}_2(t,s,u,v)=-2\Theta^{(N)}_2(t,s,u,v)$ and
\[
E_{1,1}^{\epsilon,N}(t)=\int_0^t\E\Big[\Theta_2^{(N)}\bigl(t,s,u^{\epsilon,N}(s),v^{\epsilon,N}(s)\bigr)\Big]\,ds=-\frac12\int_0^t\E\Big[\mathcal{L}^{(N)}\Theta_2^{(N)}\bigl(t,s,u^{\epsilon,N}(s),v^{\epsilon,N}(s)\bigr)\Big]\,ds.
\]
Applying the It\^o's formula and the definition \eqref{def:operator-dampedwave} of the operator $\mathcal{L}^{\epsilon,N}$, one obtains
\begin{align*}
&~\quad \E[\Theta^{(N)}_{2}(t,t,u^{\epsilon,N}(t),v^{\epsilon,N}(t))]-\E[\Theta^{(N)}_{2}(t,0,u^{\epsilon,N}(0),v^{\epsilon,N}(0))]\nonumber\\
&=\int_0^t\E\big[\partial_s\Theta^{(N)}_{2}(t,s,u^{\epsilon,N}(s),v^{\epsilon,N}(s))\big]\,ds+\int_0^t\E\big[\mathcal{L}^{\epsilon,N}\Theta^{(N)}_{2}(t,s,u^{\epsilon,N}(s),v^{\epsilon,N}(s))\big]\,ds\nonumber\\
	&=\int_0^t\E\big[\partial_s\Theta^{(N)}_{2}(t,s,u^{\epsilon,N}(s),v^{\epsilon,N}(s))\big]\,ds+\frac{1}{\epsilon}\int_0^t\E\big[\mathcal{A}^{(N)}\Theta^{(N)}_{2}(t,s,u^{\epsilon,N}(s),v^{\epsilon,N}(s))\big]\,ds\nonumber\\
	&~~~\quad +\frac{1}{\epsilon^2}\int_0^t\E\big[\mathcal{L}^{(N)}\Theta^{(N)}_{2}(t,s,u^{\epsilon,N}(s),v^{\epsilon,N}(s))\big]\,ds.
\end{align*}
Therefore, one obtains
\begin{align*}
 E^{\epsilon,N}_{1,1}(t)&=-\frac12\int_0^t \E\bigl[\mathcal{L}^{(N)}\Theta^{(N)}_{2}(t,s,u^{\epsilon,N}(s),v^{\epsilon,N}(s))\bigr]\,ds\nonumber\\
 &=\frac{\epsilon^2}{2}\Bigl(\E[\Theta^{(N)}_{2}(t,0,u^{\epsilon,N}(0),v^{\epsilon,N}(0))]\!-\!\E[\Theta^{(N)}_{2}(t,t,u^{\epsilon,N}(t),v^{\epsilon,N}(t))]\Bigr)\nonumber\\
 &\quad+\frac{\epsilon^2}{2}\int_0^t\E\big[\partial_s\Theta^{(N)}_{2}(t,s,u^{\epsilon,N}(s),v^{\epsilon,N}(s))\big]\,ds+\frac{\epsilon}{2}\int_0^t\E\big[\mathcal{A}^{(N)}\Theta^{(N)}_{2}(t,s,u^{\epsilon,N}(s),v^{\epsilon,N}(s))\big]\,ds.
\end{align*}
Finally, for the error term $E^{\epsilon,N}_{1,1}(t)$ one obtains the following decomposition
\begin{equation}\label{eq:decompose-E11}
 E^{\epsilon,N}_{1,1}(t)=E^{\epsilon,N}_{1,1,1}(t)+E^{\epsilon,N}_{1,1,2}(t)+E^{\epsilon,N}_{1,1,3}(t)
\end{equation}
where the error terms $E^{\epsilon,N}_{1,1,1}(t)$ $E^{\epsilon,N}_{1,1,2}(t)$ and $E^{\epsilon,N}_{1,1,3}(t)$ are defined by
\begin{align}
E^{\epsilon,N}_{1,1,1}(t)&=\frac{\epsilon^2}{2}\sum_{n,m=1}^N\E\bigl[D^2_{uu}\Phi^{(0,N)}(t,u^{\epsilon,N}(0)).(e_n,e_m)\Psi_{n,m}(X^{\epsilon,N}(0))\bigr]\nonumber\\
 &-\frac{\epsilon^2}{2}\sum_{n,m=1}^N\E\bigl[D^2_{uu}\Phi^{(0,N)}(0,u^{\epsilon,N}(t)).(e_n,e_m)\Psi_{n,m}(X^{\epsilon,N}(t))\bigr],\label{eq:defEepsN111}\\
E^{\epsilon,N}_{1,1,2}(t)&=\frac{\epsilon^2}{2}\sum_{n,m=1}^N\int_0^t\E\Big[\partial_s\Big(D^2_{uu}\Phi^{(0,N)}(t-s,u^{\epsilon,N}(s)).(e_n,e_m)\Psi_{n,m}(X^{\epsilon,N}(s))\Big)\Big]\,ds,\label{eq:defEepsN112}\\
 E^{\epsilon,N}_{1,1,3}(t)&=	\frac{\epsilon}{2}\sum_{n,m=1}^N\int_0^t\E\Big[\mathcal{A}^{(N)}\Big(D^2_{uu}\Phi^{(0,N)}(t-s,u^{\epsilon,N}(s)).(e_n,e_m)\Psi_{n,m}(X^{\epsilon,N}(s))\Big)\Big]\,ds.\label{eq:defEepsN113}
\end{align}

To summarize, for all $\epsilon\in(0,1)$, $N\in\N$ and $t\in[0,T]$ one obtains the decompositions
\begin{align*}
\E[\varphi(u^{\epsilon,N}(t))]-\E[\varphi(u^{0,N}(t))]&=E^{\epsilon,N}(t)\\
&=E_1^{\epsilon,N}(t)+E_2^{\epsilon,N}(t)\\
&=E_{1,1}^{\epsilon,N}(t)+E_{1,2}^{\epsilon,N}(t)+E_2^{\epsilon,N}(t)\\
&=E^{\epsilon,N}_{1,1,1}(t)+E^{\epsilon,N}_{1,1,2}(t)+E^{\epsilon,N}_{1,1,3}(t)+E_{1,2}^{\epsilon,N}(t)+E_2^{\epsilon,N}(t).
\end{align*}
It then remains to state and prove upper bounds for the five error terms appearing in the last line above. Error terms are grouped according to the assumptions and techniques required for the proofs of the upper bounds. More precisely, the terms $E^{\epsilon,N}_{1,2}(t), E^{\epsilon,N}_{1,1,3}(t)$ appear to be of order $\epsilon$ and are thus treated together. Similarly, the terms $E^{\epsilon,N}_{1,1,1}(t), E^{\epsilon,N}_{1,1,2}(t)$ appear to be of order $\epsilon^2$ and are also treated together. The term $E^{\epsilon,N}_2(t)$ is treated separately.

\begin{lemma}\label{lem:error-T1}
For all $T\in(0,\infty)$, $\alpha\in[0,\beta)$ and any mapping $\varphi\colon H\to\R$ of class $\mathcal{C}^1$, with bounded derivatives up to order $1$, there exists $C_\alpha(T,\varphi)\in(0,\infty)$ such that for all $N\in\N$ and $\epsilon\in(0,1)$ one has
\begin{equation}\label{eq:error-T1}
\underset{t\in[0,T]}\sup~|E^{\epsilon,N}_{2}(t)|\le C_\alpha(T,\varphi)\epsilon^{\min(2\alpha,1)}.      
\end{equation}
\end{lemma}

\begin{lemma}\label{lem:error-T2}
For all $T\in(0,\infty)$, $\alpha\in[0,\beta)$, and any mapping $\varphi\colon H\to\R$ of class $\mathcal{C}^3$, with bounded derivatives up to order $3$, there exists $C_{\alpha}(T,\varphi)\in(0,\infty)$ such that for all $t\in[0,T]$ one has
\begin{align}
|E^{\epsilon,N}_{1,2}(t)|&\le C_{\alpha}(T,\varphi) \epsilon\lambda_N^{\max(\frac12-\alpha,0)},\label{eq:error-T2a}\\
|E^{\epsilon,N}_{1,1,3}(t)|&\le C_{\alpha}(T,\varphi) \epsilon\lambda_N^{\max(\frac12-\alpha,0)}.\label{eq:error-T2b}
\end{align}
\end{lemma}

\begin{lemma}\label{lem:error-T3}
For all $T\in(0,\infty)$, $\alpha\in[0,\beta)$, and any mapping $\varphi\colon H\to\R$ of class $\mathcal{C}^4$, with bounded derivatives up to order $4$, there exists $C_{\alpha}(T,\varphi)\in(0,\infty)$ such that for all $t\in[0,T]$ one has
\begin{align}
|E^{\epsilon,N}_{1,1,1}(t)|&\le C_{\alpha}(T,\varphi)\epsilon^2\lambda_N^{1-\alpha}\label{eq:error-T3a}\\
|E^{\epsilon,N}_{1,1,2}(t)|&\le C_{\alpha}(T,\varphi)\epsilon^2\lambda_N^{1-\alpha}\label{eq:error-T3b}.
\end{align}
\end{lemma}

The proof of Theorem~\ref{theo:weak-N} is then a straightforward consequence of the three auxiliary results stated above.
\begin{proof}[Proof of Theorem~\ref{theo:weak-N}]
It suffices to recall that for all $\epsilon\in(0,1)$, $N\in\N$ and $t\in[0,T]$ one has
\[
\E[\varphi(u^{\epsilon,N}(t))]-\E[\varphi(u^{0,N}(t))]=E^{\epsilon,N}_{1,1,1}(t)+E^{\epsilon,N}_{1,1,2}(t)+E^{\epsilon,N}_{1,1,3}(t)+E_{1,2}^{\epsilon,N}(t)+E_2^{\epsilon,N}(t)
\]
and to apply the error estimates from Lemmas~\ref{lem:error-T1},~\ref{lem:error-T2} and~\ref{lem:error-T3} to obtain~\eqref{eq:theo-weak-N}.
\end{proof}

The proofs of Lemmas~\ref{lem:error-T1},~\ref{lem:error-T2} and~\ref{lem:error-T3} are technical and are given in Section~\ref{sec:proof-theo-weak-errorterms}.

\subsection{Proof of technical upper bounds on the auxiliary error terms}\label{sec:proof-theo-weak-errorterms}

\begin{proof}[Proof of Lemma~\ref{lem:error-T1}]
Recall that the error term $E^{\epsilon,N}_2(t)$ is defined by~\eqref{eq:defEepsN2}. Owing to the mean value theorem and applying the regularity property~\eqref{eq:propo-regularity-PhiN} from Proposition~\ref{propo:regularity-PhiN} (with $k=1$ and $\alpha_1=0$), there exists $C(T)\in(0,\infty)$ such that for all $t\in[0,T]$ one has
\begin{align*}
|E^{\epsilon,N}_{2}(t)|&=\bigl|\E[\Phi^{(0,N)}(t,P_Nu^{\epsilon}_0)]-\E[\Phi^{(0,N)}(t,P_Nu_0^0)]\bigr|\\
&\le \Big|\int_0^1 \E\Bigl[D_u\Phi^{(0,N)}(t,P_Nu^\epsilon_0+\theta P_{N}(u^0_0-u^\epsilon_0)). P_{N}(u^0_0-u^\epsilon_0)\Bigr]\,d\theta\Big|\\
&\le C(T)\vvvert\varphi\vvvert_1 \E[\|u^0_0-u^\epsilon_0\|_{H}].
\end{align*}
Finally, applying the condition~\eqref{eq:init-weak-error} from Assumption~\ref{ass:init-weak} on the initial value $u_0^\epsilon$, one thus obtains the inequality~\eqref{eq:error-T1} and the proof of Lemma~\ref{lem:error-T1} is completed.
\end{proof}

\begin{proof}[Proof of the inequality~\eqref{eq:error-T2a} from Lemma~\ref{lem:error-T2}]

The error term $E^{\epsilon,N}_{1,2}(t)$ defined by~\eqref{eq:defEepsN12} is decomposed as
\begin{equation}\label{eq:decompose-E12}
E^{\epsilon,N}_{1,2}(t)=E^{\epsilon,N}_{1,2,1}(t)+E^{\epsilon,N}_{1,2,2}(t)+E^{\epsilon,N}_{1,2,3}(t),
\end{equation}
where the error terms $E^{\epsilon,N}_{1,2,1}(t)$, $E^{\epsilon,N}_{1,2,2}(t)$ and $E^{\epsilon,N}_{1,2,3}(t)$ are defined as
\begin{align}
E^{\epsilon,N}_{1,2,1}(t)&=\epsilon\E[D_u\Phi^{(0,N)}(t,u^{\epsilon,N}(0)).v^{\epsilon,N}(0)],\label{eq:defEepsN121}\\
E^{\epsilon,N}_{1,2,2}(t)&=-\epsilon\E[D_u\Phi^{(0,N)}(0,u^{\epsilon,N}(t)).v^{\epsilon,N}(t)],\label{eq:defEepsN122}\\
E^{\epsilon,N}_{1,2,3}(t)&=\epsilon\int_0^t\E\big[D_u\bigl(\partial_s\Phi^{(0,N)}(t-s,u^{\epsilon,N}(s))\bigr).v^{\epsilon,N}(s)\big]\,ds.\label{eq:defEepsN123}
\end{align}

$\bullet$ \textbf{Treatment of the error term $E^{\epsilon,N}_{1,2,1}(t)$}

Applying the inequality~\eqref{eq:propo-regularity-PhiN} from Proposition~\ref{propo:regularity-PhiN} with $k=1$ and $\alpha_1=0$, and recalling that $X^{\epsilon,N}(0)=\IP_Nx_0^{\epsilon}$, one has
\[
\big|E^{\epsilon,N}_{1,2,1}(t)\big|\le C(T)\epsilon \vvvert \varphi\vvvert_1\E[\|v^{\epsilon,N}(0)\|_H]=C(T)\epsilon \vvvert \varphi\vvvert_1\E[\|x_0^{\epsilon,N}\|_{\mathcal{H}^1}].
\]
Applying the inverse inequality~\eqref{eq:inverseinequality} if $\alpha\in[0,1/2)$, and taking into account the condition~\eqref{eq:init-weak-bound} from  Assumption~\ref{ass:init-weak} on the initial value $x_0^\epsilon$, one obtains the upper bounds
\begin{equation}\label{eq:boundIPNx0HH1}
\E[\|\IP_Nx_0^{\epsilon}\|_{\HH^1}]\le \E[\|x_0^{\epsilon}\|_{\HH^1}]\mathds{1}_{\alpha\ge 1/2}+\lambda_N^{\frac12-\alpha}\E[\|x_0^\epsilon\|_{\HH^{2\alpha}}]\mathds{1}_{\alpha < 1/2}\le C_\alpha \lambda_N^{\max(\frac12-\alpha,0)}.
\end{equation}
One thus obtains the following upper bound for the error term $E^{\epsilon,N}_{1,2,1}(t)$: there exists $C_\alpha(T,\varphi)\in(0,\infty)$ such that for all $t\in[0,T]$, all $\epsilon\in(0,1)$ and all $N\in\N$ one has
\begin{equation}\label{eq:boundEepsN121}
\big|E^{\epsilon,N}_{1,2,1}(t)\big|\le C_\alpha(T,\varphi)\epsilon\lambda_N^{\max(\frac12-\alpha,0)}.
\end{equation}

$\bullet$ \textbf{Treatment of the error term $E^{\epsilon,N}_{1,2,2}(t)$}.

Owing to~\eqref{eq:PhiN} one has $\Phi^{(0,N)}(0,\cdot)=\varphi(\cdot)$. Noting that $v^{\epsilon,N}(t)=\Pi_vX^{\epsilon,N}(t)$ and using the expression~\eqref{eq:mild-sdwe-galerkin} of $X^{\epsilon,N}(t)$, the error term $E^{\epsilon,N}_{1,2,2}(t)$ is decomposed as
\begin{equation}\label{eq:decompose-E122}
E^{\epsilon,N}_{1,2,2}(t)=E^{\epsilon,N}_{1,2,2,1}(t)+E^{\epsilon,N}_{1,2,2,2}(t)+E^{\epsilon,N}_{1,2,2,3}(t)
\end{equation}
where the error terms $E^{\epsilon,N}_{1,2,2,1}(t)$, $E^{\epsilon,N}_{1,2,2,2}(t)$ and $E^{\epsilon,N}_{1,2,2,3}(t)$ are defined as
\begin{align}
E^{\epsilon,N}_{1,2,2,1}(t)&=-\epsilon\E[D_u\varphi(u^{\epsilon,N}(t)).\Pi_ve^{tA_{\epsilon}}\IP_Nx^\epsilon_0],\label{eq:defEepsN1221}\\
E^{\epsilon,N}_{1,2,2,2}(t)&=-\E\Bigl[D_u\varphi(u^{\epsilon,N}(t)). \int_0^t\Pi_ve^{(t-s)A_{\epsilon}}F_N(X^{\epsilon,N}(s))\,ds\Bigr],\label{eq:defEepsN1222}\\
E^{\epsilon,N}_{1,2,2,3}(t)&=-\E\Bigl[D_u\varphi(u^{\epsilon,N}(t)). \int_0^t\Pi_v e^{(t-s)A_{\epsilon}}\IP_N\,d\IW^Q(s)\Bigr].\label{eq:defEepsN1223}
\end{align}

{\it Treatment of the error term $E^{\epsilon,N}_{1,2,2,1}(t)$}.
Applying the inequality~\eqref{eq:lem_expAeps-bound} from Lemma~\ref{lem:expAeps-bound} with $\alpha=1$, for all $t\in[0,T]$ one has
\[
\big|E^{\epsilon,N}_{1,2,2,1}(t)\big|\le C(T)\epsilon\vvvert\varphi\vvvert_1 \E[\|\Pi_ve^{tA_{\epsilon}}\IP_Nx^\epsilon_0\|]\le C(T,\varphi)\epsilon \E[\|e^{tA_{\epsilon}}\IP_Nx^\epsilon_0\|_{\HH^1}]\le C(T,\varphi)\epsilon \E[\|\IP_Nx^\epsilon_0\|_{\HH^1}].
\]
Applying the inequality~\eqref{eq:boundIPNx0HH1} (see the treatment of the error term $E^{\epsilon,N}_{1,2,1}(t)$ above), one thus obtains the following upper bound for the error term $E^{\epsilon,N}_{1,2,2,1}(t)$: there exists $C_\alpha(T,\varphi)\in(0,\infty)$ such that for all $t\in[0,T]$, all $\epsilon\in(0,1)$ and all $N\in\N$ one has
\begin{equation}\label{eq:boundEepsN1221}
\big|E^{\epsilon,N}_{1,2,2,1}(t)\big|\le C_\alpha(T,\varphi)\epsilon\lambda_N^{\max(\frac12-\alpha,0)}.
\end{equation}

{\it Treatment of the error term $E^{\epsilon,N}_{1,2,2,2}(t)$}.
First, for all $t\in[0,T]$ one has
\begin{align*}
|E^{\epsilon,N}_{1,2,2,2}(t)|&\le C(T)\vvvert\varphi\vvvert_1\int_0^t\E\bigl[\bigl\|\Pi_ve^{(t-s)A_{\epsilon}}F_N(X^{\epsilon,N}(s))\bigr\|_H\bigr]\,ds\\
&\le C(T,\varphi)\int_0^t\E\bigl[\bigl\|e^{(t-s)A_{\epsilon}}(0,f_N(u^{\epsilon,N}(s)))\bigr\|_{\HH^1}\bigr]\,ds.
\end{align*}
Applying the inequality~\eqref{eq:lem_expAeps-smoothing1} from Lemma~\ref{lem:expAeps-smoothing1} with $\bigl(\alpha,\delta,\rho\bigr)=\bigl(1,1/2,1\bigr)$ and the linear growth property of the mapping $f$, one obtains
\begin{align*}
|E^{\epsilon,N}_{1,2,2,2}(t)|&\le C(T,\varphi)\epsilon  \int_0^t(t-s)^{-\frac12}(1+e^{-\frac{t-s}{2\epsilon^2}})\times \E[\|f_N(u^{\epsilon,N}(s)\|_H]\,ds\\
&\le C(T,\varphi)\epsilon \bigl(1+\underset{t\in[0,T]}\sup~\E[\|u^{\epsilon,N}(t)\|_{\HH}]\bigr).
\end{align*}
Applying the moment bounds~\eqref{eq:propo-sdwe-momentbounds-Galerkin} from Proposition~\ref{propo:sdwe-momentbounds} (with $\alpha=0$) and taking into account the condition~\eqref{eq:init-weak-bound} from Assumption~\ref{ass:init-weak} on the initial value $x_0^\epsilon$ (with $\alpha=0$), one obtains the following result: there exists $C(T,\varphi)\in(0,\infty)$ such that for all $t\in[0,T]$, all $\epsilon\in(0,1)$ and all $N\in\N$ one has 
\begin{equation}\label{eq:boundEepsN1222}
\big|E^{\epsilon,N}_{1,2,2,2}(t)\big|\le C(T,\varphi)\epsilon.
\end{equation}

{\it Treatment of the error term $E^{\epsilon,N}_{1,2,2,3}(t)$}. The treatment is more delicate and requires to employ the Malliavin calculus techniques presented in Section~\ref{sec:Malliavin}. Applying the identity~\eqref{prop:Malliavinderivative-W}, one obtains
\begin{align*}
E^{\epsilon,N}_{1,2,2,3}(t)&=-\E\Bigl[D\varphi(u^{\epsilon,N}(t)). \int_0^t\Pi_v e^{(t-s)A_{\epsilon}}\IP_N\,d\IW^Q(s)\Bigr]\\
&=-\sum_{n=1}^N\int_0^t\E\bigl[D^2\varphi(u^{\epsilon,N}(t)).\bigl(\mathcal{D}^{e_n}_su^{\epsilon,N}(t),\Pi_v e^{(t-s)A_{\epsilon}}(0,\sqrt{q_n}e_n)\bigr)\bigr]\,ds.
\end{align*}
The second order derivative of $\varphi$ is assumed to be bounded, therefore one has
\[
\big|E^{\epsilon,N}_{1,2,2,3}(t)\big|\le \vvvert\varphi\vvvert_2\sum_{n=1}^{N}\sqrt{q_n}\int_{0}^{t}\E[\|\mathcal{D}^{e_n}_su^{\epsilon,N}(t)\|]\|e^{(t-s)A_{\epsilon}}(0,e_n)\|_{\HH^1}\,ds.
\]
On the one hand, letting $\kappa\in[0,1)$ and applying the inequality~\eqref{eq:lemMalliavinderivative-moment} from Lemma~\ref{lem:Malliavinderivative-moment} on the Malliavin derivative $\mathcal{D}^{e_n}_su^{\epsilon,N}(t)$, one has
\[
\E[\|\mathcal{D}^{e_n}_su^{\epsilon,N}(t)\|]\le C_\kappa(T)(t-s)^{-\frac{\kappa}{2}}\|Q^{\frac12}e_n\|_{H^{-\kappa}}=C_\kappa(T)(t-s)^{-\frac{\kappa}{2}}\sqrt{q_n}\lambda_n^{-\frac{\kappa}{2}}.
\]
On the other hand, applying the inequality~\eqref{eq:lem_expAeps-smoothing1} from Lemma~\ref{lem:expAeps-smoothing1} with $\bigl(\alpha,\rho,\delta\bigr)=\bigl(1,1,1/2\bigr)$, one has
\[
\|e^{(t-s)A_{\epsilon}}(0,e_n)\|_{\HH^1}\le C\epsilon(t-s)^{-\frac12}.
\]
Therefore one obtains the upper bound
\[
\big|E^{\epsilon,N}_{1,2,2,3}(t)\big|\le C_\kappa(T,\varphi)\epsilon\int_{0}^{t}(t-s)^{-\frac{1+\kappa}{2}}\,ds \sum_{n=1}^{N}q_n\lambda_n^{-\frac{\kappa}{2}},
\]
and the treatment of the sum on the right-hand side above requires to be treated differently in two cases.

First, assume that $\beta\in(1/2,1]$. Set $\kappa=\frac32-\beta$. Note that $\kappa\in(0,1)$ and that $\alpha=1-\frac{\kappa}{2}=\frac14+\frac{\beta}{2}<\beta$. As a result, owing to the condition~\eqref{eq:condition_beta} from Assumption~\ref{ass:noise} on the covariance operator $Q$, one has
\[
\sum_{n=1}^{N}q_n\lambda_n^{-\frac{\kappa}{2}}=\sum_{n=1}^{N}q_n\lambda_n^{\alpha-1}\le \sum_{n=1}^{\infty}q_n\lambda_n^{\alpha-1}=C_\alpha.
\]

Second, assume that $\beta\in(0,1/2]$ and let $\alpha\in(0,\beta)$. Set $\kappa=1-(\beta-\alpha)$. Note that $\kappa\in(0,1)$, and that one has
\[
\sum_{n=1}^{N}q_n\lambda_n^{-\frac{\kappa}{2}}=\sum_{n=1}^{N}q_n\lambda_n^{\alpha-\frac12-\frac{\kappa}{2}}\lambda_n^{\frac12-\alpha}\le \lambda_{N}^{\frac12-\alpha}\sum_{n=1}^{N}q_n\lambda_n^{\frac{\alpha+\beta}{2}-1}.
\]
As a result, since $\frac{\alpha+\beta}{2}<\beta$, owing to the condition~\eqref{eq:condition_beta} from Assumption~\ref{ass:noise} on the covariance operator $Q$, one has
\[
\sum_{n=1}^{N}q_n\lambda_n^{-\frac{\kappa}{2}}\le \lambda_N^{\frac12-\alpha}\sum_{n=1}^{\infty}q_n\lambda_n^{\frac{\alpha+\beta}{2}-1}=C_\alpha\lambda_N^{\frac12-\alpha}.
\]

Gathering the upper bounds, one obtains the following result: for all $\alpha\in(0,\beta)$, there exists $C_\alpha(T,\varphi)\in(0,\infty)$ such that for all $t\in[0,T]$, all $\epsilon\in(0,1)$ and all $N\in\N$ one has
\begin{equation}\label{eq:boundEepsN1223}
\big|E^{\epsilon,N}_{1,2,2,3}(t)\big|\le C_\alpha(T,\varphi)\epsilon\lambda_N^{\max(\frac12-\alpha,0)}.
\end{equation}

{\it Conclusion}.
Recalling the decomposition~\eqref{eq:decompose-E122} of the error term and gathering the error bounds~\eqref{eq:boundEepsN1221},~\eqref{eq:boundEepsN1222} and~\eqref{eq:boundEepsN1223}, one thus obtains the following error bound for the error term $E^{\epsilon,N}_{1,2,2}(t)$: there exists $C_\alpha(T,\varphi)\in(0,\infty)$ such that for all $t\in[0,T]$, all $\epsilon\in(0,1)$ and all $N\in\N$ one has
\begin{equation}\label{eq:boundEepsN122}
\big|E^{\epsilon,N}_{1,2,2}(t)\big|\le C_\alpha(T,\varphi)\epsilon\lambda_N^{\max(\frac12-\alpha,0)}.
\end{equation}

$\bullet$ \textbf{Treatment of the error term $E^{\epsilon,N}_{1,2,3}(t)$}.

Since the mapping $\Phi^{(0,N)}$ is solution to the Kolmogorov equation~\eqref{eq:kolmogorovPhiN}, the error term $E^{\epsilon,N}_{1,2,3}$ can be written as
\[
E^{\epsilon,N}_{1,2,3}(t)=-\epsilon\int_0^t\E\big[D_u\bigl(\mathbb{L}^{(N)}\Phi^{(0,N)}\bigr)(t-s,u^{\epsilon,N}(s)).v^{\epsilon,N}(s)\bigr]\,ds.
\]
Recalling the definition~\eqref{eq:LN} of the infinitesimal generator $\mathbb{L}^{(N)}$, the error term $E^{\epsilon,N}_{1,2,3}(t)$ is decomposed as
\begin{equation}\label{eq:decompose-E123}
E^{\epsilon,N}_{1,2,3}(t)=E^{\epsilon,N}_{1,2,3,1}(t)+E^{\epsilon,N}_{1,2,3,2}(t)+E^{\epsilon,N}_{1,2,3,3}(t)+E^{\epsilon,N}_{1,2,3,4}(t)+E^{\epsilon,N}_{1,2,3,5}(t),
\end{equation}
where the error terms $E^{\epsilon,N}_{1,2,3,1}(t)$, $E^{\epsilon,N}_{1,2,3,2}(t)$, $E^{\epsilon,N}_{1,2,3,3}(t)$, $E^{\epsilon,N}_{1,2,3,4}(t)$ and $E^{\epsilon,N}_{1,2,3,5}(t)$ are defined as
\begin{align}
E^{\epsilon,N}_{1,2,3,1}(t)&=\epsilon\int_0^t\E\bigl[D_u\Phi^{(0,N)}(t-s,u^{\epsilon,N}(s)).\bigl(\Lambda v^{\epsilon,N}(s)\bigr)\bigr]\,ds,\label{eq:defEepsN1231}\\
E^{\epsilon,N}_{1,2,3,2}(t)&=-\epsilon\int_0^t\E\bigl[D_u\Phi^{(0,N)}(t-s,u^{\epsilon,N}(s)).\bigl(Df_N(u^{\epsilon,N}(s)).v^{\epsilon,N}(s)\bigr)\bigr]\,ds,\label{eq:defEepsN1232}\\
E^{\epsilon,N}_{1,2,3,3}(t)&=\epsilon\int_0^t\E\bigl[D_{uu}^2\Phi^{(0,N)}(t-s,u^{\epsilon,N}(s)).\bigl(v^{\epsilon,N}(s),\Lambda u^{\epsilon,N}(s)\bigr)\bigr]\,ds,\label{eq:defEepsN1233}\\
E^{\epsilon,N}_{1,2,3,4}(t)&=-\epsilon\int_0^t\E\bigl[D_{uu}^2\Phi^{(0,N)}(t-s,u^{\epsilon,N}(s)).\bigl(v^{\epsilon,N}(s),f_N(u^{\epsilon,N}(s))\bigr)\bigr]\,ds,\label{eq:defEepsN1234}\\
E^{\epsilon,N}_{1,2,3,5}(t)&=-\frac{\epsilon}{2}\int_0^t\sum_{n=1}^{N}q_n\times\E\bigl[D^3_{u}\Phi^{(0,N)}(t-s,u^{\epsilon,N}(s)).(e_n,e_n,v^{\epsilon,N}(s))\bigr]\,ds.\label{eq:defEepsN1235}
\end{align}

{\it Treatment of the error term $E^{\epsilon,N}_{1,2,3,1}(t)$}.

Writing $v^{\epsilon,N}(s)=\Pi_vX^{\epsilon,N}(s)$ and using  the expression~\eqref{eq:mild-sdwe} of $X^{\epsilon,N}(s)$, one obtains the following decomposition of the error term $E^{\epsilon,N}_{1,2,3,1}(t)$:
\begin{equation}\label{eq:decompose-E1231}
E^{\epsilon,N}_{1,2,3,1}(t)=E^{\epsilon,N}_{1,2,3,1,1}(t)+E^{\epsilon,N}_{1,2,3,1,2}(t)+E^{\epsilon,N}_{1,2,3,1,3}(t)
\end{equation}
where the error terms $E^{\epsilon,N}_{1,2,3,1,1}(t)$, $E^{\epsilon,N}_{1,2,3,1,2}(t)$ and $E^{\epsilon,N}_{1,2,3,1,3}(t)$ are defined as
\begin{align}
E^{\epsilon,N}_{1,2,3,1,1}(t)
&=\epsilon\int_0^t \E\bigl[D_u\Phi^{(0,N)}(t-s,u^{\epsilon,N}(s)).(\IL\Pi_ve^{sA_\epsilon}\IP_Nx_0^\epsilon)\bigr]\,ds\label{eq:defEepsN12311}\\
E^{\epsilon,N}_{1,2,3,1,2}(t)
&=\int_0^t\E \bigl[D_u\Phi^{(0,N)}(t-s,u^{\epsilon,N}(s)).\int_0^s\IL\Pi_ve^{(s-r)A_\epsilon}F_N(X^{\epsilon,N}(r))\,dr\bigr]\,ds\label{eq:defEepsN12312}\\
E^{\epsilon,N}_{1,2,3,1,3}(t)
&=\int_0^t\E\bigl[D_u\Phi^{(0,N)}(t-s,u^{\epsilon,N}(s)).\int_0^s\IL\Pi_ve^{(s-r)A_\epsilon}\IP_N\,d\IW^Q(r)\bigr]\,ds
\label{eq:defEepsN12313}
\end{align}

Consider first the error term $E^{\epsilon,N}_{1,2,3,1,1}(t)$. Let $\alpha\in(0,\beta)$, and let $\kappa\in(0,1)$. Applying the inequality~\eqref{eq:propo-regularity-PhiN} from Proposition~\ref{propo:regularity-PhiN} with $k=1$ and $\alpha_1=2(1-\kappa)$, one has
\[
|E^{\epsilon,N}_{1,2,3,1,1}(t)|\le C_\kappa(T)\epsilon\vvvert\varphi\vvvert_1\int_{0}^{t}(t-s)^{-1+\kappa}\E[\|\IL\Pi_ve^{sA_\epsilon}\IP_Nx_0^\epsilon\|_{H^{-2+2\kappa}}]\,ds.
\]
Moreover, one has the upper bounds
\[
\|\IL\Pi_ve^{sA_\epsilon}\IP_Nx_0^\epsilon\|_{H^{-2+2\kappa}}=\|\Pi_ve^{sA_\epsilon}\IP_Nx_0^\epsilon\|_{H^{2\kappa}}\le \|e^{sA_\epsilon}\IP_Nx_0^\epsilon\|_{\HH^{1+2\kappa}},
\]
and applying the inequality~\eqref{eq:lem_expAeps-bound} from Lemma~\ref{lem:expAeps-bound} one obtains the upper bound
\[
|E^{\epsilon,N}_{1,2,3,1,1}(t)|\le C_\kappa(T,\varphi)\epsilon \E[\|\IP_Nx_0^\epsilon\|_{\HH^{1+2\kappa}}].
\]
First, assume that $\beta\in(0,1/2]$. Choosing $\kappa=(\beta-\alpha)/2$ and applying the inverse inequality~\eqref{eq:inverseinequality}, one obtains
\[
\E[\|\IP_Nx_0^\epsilon\|_{\HH^{1+2\kappa}}]\le \lambda_N^{\frac12-\alpha}\E[\|x_0^\epsilon\|_{\HH^{2(\alpha+\kappa)}}]=\lambda_N^{\frac12-\alpha}\E[\|x_0^\epsilon\|_{\HH^{\alpha+\beta}}].
\]
Owing to the condition~\eqref{eq:init-weak-bound} from Assumption~\ref{ass:init-weak} on the initial value $x_0^\epsilon$, one then obtains
\[
\underset{\epsilon\in(0,1)}\sup~\E[\|\IP_Nx_0^\epsilon\|_{\HH^{1+2\kappa}}]\le C_\alpha \lambda_N^{\frac12-\alpha}.
\]
Second, assume that $\beta\in(1/2,1]$. Choosing $\kappa=(2\beta-1)/4$, one has $1+2\kappa=\beta+1/2<2\beta$. Owing to the condition~\eqref{eq:init-weak-bound} from Assumption~\ref{ass:init-weak} on the initial value $x_0^\epsilon$, one then obtains
\[
\underset{\epsilon\in(0,1)}\sup~\E[\|\IP_Nx_0^\epsilon\|_{\HH^{1+2\kappa}}]<\infty.
\]
Combining the results from the two cases, one thus obtains the following inequality for the error term $E^{\epsilon,N}_{1,2,3,1,1}(t)$: for all $\alpha\in(0,\beta)$, there exists $C_\alpha(T,\varphi)\in(0,\infty)$ such that for all $t\in[0,T]$, all $\epsilon\in(0,1)$ and all $N\in\N$ one has
\begin{equation}\label{eq:boundEepsN12311}
\big|E^{\epsilon,N}_{1,2,3,1,1}(t)\big|\le C_\alpha(T,\varphi)\epsilon\lambda_N^{\max(\frac12-\alpha,0)}.
\end{equation}

Consider next the error term $E^{\epsilon,N}_{1,2,3,1,2}(t)$. Let $\alpha\in(0,\beta)$ and $\kappa\in(0,\alpha/2)$. Applying the inequality~\eqref{eq:propo-regularity-PhiN} from Proposition~\ref{propo:regularity-PhiN} with $k=1$ and $\alpha_1=2-\kappa$, one has
\[
|E^{\epsilon,N}_{1,2,3,1,2}(t)|\le C_\kappa(T)\epsilon\vvvert\varphi\vvvert_1\int_{0}^{t}(t-s)^{-1+\frac{\kappa}{2}}\int_0^s\E[\|\IL\Pi_ve^{(s-r)A_\epsilon}F_N(X^{\epsilon,N}(r))\|_{H^{-2+\kappa}}]\,dr\,ds.
\]
Moreover, one has the upper bounds
\[
\|\IL\Pi_ve^{(s-r)A_\epsilon}F_N(X^{\epsilon,N}(r))\|_{H^{-2+\kappa}}=\|\Pi_ve^{(s-r)A_\epsilon}F_N(X^{\epsilon,N}(r))\|_{H^{\kappa}}\le \|e^{(s-r)A_\epsilon}\bigl(0,f_N(u^{\epsilon,N}(r))\bigr)\|_{\HH^{1+\kappa}}.
\]
Applying the inequality~\eqref{eq:lem_expAeps-smoothing1} from Lemma~\ref{lem:expAeps-smoothing1} with $\bigl(\alpha,\delta,\rho\bigr)=\bigl(1+\kappa,1/2,1\bigr)$ and the inequality~\eqref{eq:falpha1} from Assumption~\ref{ass:mapping-f} on the mapping $f$ with $\alpha=\kappa$, one obtains
\begin{align*}
\int_0^s\E[\|\IL\Pi_ve^{(s-r)A_\epsilon}F_N(X^{\epsilon,N}(r))\|_{H^{-2+\kappa}}]\,dr&\le C_\kappa\int_{0}^{s}(s-r)^{-\frac12}\E[\|f_N(u^{\epsilon,N}(r))\|_{H^{\kappa}}]\,dr\\
&\le  C_\kappa\int_{0}^{s}(s-r)^{-\frac12}\bigl(1+\E[\|u^{\epsilon,N}(r)\|_{H^{2\kappa}}]\bigr)\,dr\\
&\le C_\kappa(T)\bigl(1+\underset{r\in[0,T]}\sup~\E[\|X^{\epsilon,N}(r)\|_{H^{\alpha}}]\bigr).
\end{align*}
Owing to the moment bounds~\eqref{eq:propo-sdwe-momentbounds-Galerkin} from Proposition~\ref{propo:sdwe-momentbounds} and to the condition~\eqref{eq:init-weak-bound} from Assumption~\ref{ass:init-weak} on the initial value $x_0^\epsilon$, one thus obtains the following inequality for the error term $E^{\epsilon,N}_{1,2,3,1,2}(t)$: for all $\alpha\in(0,\beta)$, there exists $C_\alpha(T,\varphi)\in(0,\infty)$ such that for all $t\in[0,T]$, all $\epsilon\in(0,1)$ and all $N\in\N$ one has
\begin{equation}\label{eq:boundEepsN12312}
\big|E^{\epsilon,N}_{1,2,3,1,2}(t)\big|\le C_\alpha(T,\varphi)\epsilon\lambda_N^{\max(\frac12-\alpha,0)}.
\end{equation}

It remains to deal with the error term $E^{\epsilon,N}_{1,2,3,1,3}(t)$. The treatment is more delicate and requires to employ the Malliavin calculus techniques presented in Section~\ref{sec:Malliavin}. Applying the identity~\eqref{prop:Malliavinderivative-W}, one obtains
\[
E^{\epsilon,N}_{1,2,3,1,3}(t)=-\sum_{n=1}^{N}\sqrt{q_n}\int_{0}^{t}\int_{0}^{s}\E\bigl[D_{uu}^2\Phi^{(0,N)}(t-s,u^{\epsilon,N}(s)).\Bigl(\mathcal{D}_r^{e_n}u^{\epsilon,N}(s),\IL\Pi_ve^{(s-r)A_\epsilon}\bigl(0,e_n)\Bigr) \bigr]\,dr\,ds.
\]
Let $\kappa\in(0,1/2]$. Applying the inequality~\eqref{eq:propo-regularity-PhiN} from Proposition~\ref{propo:regularity-PhiN} with $k=2$, $\alpha_1=0$ and $\alpha_2=2(1-\kappa)$, one has
\[
|E^{\epsilon,N}_{1,2,3,1,3}(t)|\le C_\kappa(T)\vvvert\varphi\vvvert_2\sum_{n=1}^{N}\sqrt{q_n}\int_{0}^{t}\int_{0}^{s}(t-s)^{-1+\kappa}\E[\|\mathcal{D}_r^{e_n}u^{\epsilon,N}(s)\|_H]\|\IL\Pi_ve^{(s-r)A_\epsilon}\bigl(0,e_n)\|_{H^{-2+2\kappa}}\,dr\,ds.
\]
On the one hand, applying the inequality~\eqref{eq:lemMalliavinderivative-moment} from Lemma~\ref{lem:Malliavinderivative-moment} on the Malliavin derivative, one obtains the upper bound
\[
\E[\|\mathcal{D}_r^{e_n}u^{\epsilon,N}(s)\|_H]\le C_\kappa(T)(s-r)^{-\frac12+\kappa}\|Q^{\frac12}e_n\|_{H^{-1+2\kappa}}= C_\kappa(T)(s-r)^{-\frac12+\kappa}\sqrt{q_n}\lambda_n^{-\frac12+\kappa}.
\]
On the other hand, applying the inequality~\eqref{eq:lem_expAeps-smoothing1} from Lemma~\ref{lem:expAeps-smoothing1} with $\bigl(\alpha,\delta,\rho\bigr)=\bigl(1+2\kappa,1/2,1\bigr)$, one obtains
\begin{align*}
\|\IL\Pi_ve^{(s-r)A_\epsilon}\bigl(0,e_n)\|_{H^{-2+2\kappa}}&\le \|\Pi_ve^{(s-r)A_\epsilon}\bigl(0,e_n)\|_{H^{2\kappa}}\le \|e^{(s-r)A_\epsilon}\bigl(0,e_n)\|_{\HH^{1+2\kappa}}\\
&\le C_\kappa\epsilon(s-r)^{-\frac12}\|e_n\|_{H^{2\kappa}}=C_\kappa\epsilon(s-r)^{-\frac12}\lambda_n^\kappa.
\end{align*}
Gathering the upper bounds, one thus obtains the following: there exists $C_\kappa(T,\varphi)\in(0,\infty)$ such that
\[
|E^{\epsilon,N}_{1,2,3,1,3}(t)|\le C_\kappa(T,\varphi)\sum_{n=1}^{N}q_n\lambda_n^{-\frac12+2\kappa}.
\]
First, assume that $\beta\in(0,1/2]$. Let $\alpha\in(0,\beta)$ and choose $\kappa=(\beta-\alpha)/4$. Then one obtains
\[
\sum_{n=1}^{N}q_n\lambda_n^{-\frac12+2\kappa}\le \lambda_N^{\frac12-\alpha}\sum_{n=1}^{N}q_n\lambda_n^{\alpha-1+2\kappa}\le C_\alpha\lambda_N^{\frac12-\alpha}
\]
with $C_\alpha=\sum_{n=1}^{\infty}\lambda_n^{\frac{\alpha+\beta}{2}-1}<\infty$.

Second, assume that $\beta\in(1/2,1]$ and choose $\kappa=(2\beta-1)/8$. Then one obtains
\[
\sum_{n=1}^{N}q_n\lambda_n^{-\frac12+2\kappa}=\sum_{n=1}^{N}q_n\lambda_n^{\beta-2\kappa-1}\le \sum_{n=1}^{\infty}q_n\lambda_n^{\beta-2\kappa-1}<\infty.
\]
Finally, one thus obtains the following inequality for the error term $E^{\epsilon,N}_{1,2,3,1,3}(t)$: for all $\alpha\in(0,\beta)$, there exists $C_\alpha(T,\varphi)\in(0,\infty)$ such that for all $t\in[0,T]$, all $\epsilon\in(0,1)$ and all $N\in\N$ one has
\begin{equation}\label{eq:boundEepsN12313}
\big|E^{\epsilon,N}_{1,2,3,1,3}(t)\big|\le C_\alpha(T,\varphi)\epsilon\lambda_N^{\max(\frac12-\alpha,0)}.
\end{equation}

Recalling the decomposition~\eqref{eq:decompose-E1231} of the error term~\eqref{eq:defEepsN1231}, and combining the upper bounds~\eqref{eq:boundEepsN12311},~\eqref{eq:boundEepsN12312} and~\eqref{eq:boundEepsN12313}, one thus obtains the following inequality for the error term $E^{\epsilon,N}_{1,2,3,1}(t)$: for all $\alpha\in(0,\beta)$, there exists $C_\alpha(T,\varphi)\in(0,\infty)$ such that for all $t\in[0,T]$, all $\epsilon\in(0,1)$ and all $N\in\N$ one has
\begin{equation}\label{eq:boundEepsN1231}
\big|E^{\epsilon,N}_{1,2,3,1}(t)\big|\le C_\alpha(T,\varphi)\epsilon\lambda_N^{\max(\frac12-\alpha,0)}.
\end{equation}

{\it Treatment of the error term $E^{\epsilon,N}_{1,2,3,2}(t)$}.
Applying the inequality~\eqref{eq:propo-regularity-PhiN} from Proposition~\ref{propo:regularity-PhiN} with $k=1$ and $\alpha_1=1$ and the Cauchy--Schwarz inequality, one has
\[
|E^{\epsilon,N}_{1,2,3,2}(t)|\le C(T)\epsilon\vvvert\varphi\vvvert_1\int_{0}^{t}(t-s)^{-\frac12}\E[\|Df_N(u^{\epsilon,N}(s)).v^{\epsilon,N}(s)\|_{H^{-1}}]\,ds.
\]
First, assume that $\beta\in(0,1/2]$ and let $\alpha\in(0,\beta)$. Owing to the condition~\eqref{eq:DfalphaBis} from Assumption~\ref{ass:mapping-f} on the mapping $f$ and applying the Cauchy--Schwarz inequality, one has
\begin{align*}
\E[\|Df_N(u^{\epsilon,N}(s)).v^{\epsilon,N}(s)\|_{H^{-1}}]&\le C_\alpha\E\bigl[\bigl(1+\|u^{\epsilon,N}(s)\|_{H^\alpha}\bigr)\|v^{\epsilon,N}(s)\|_{H^{-\alpha}}\bigr]\\
&\le C_\alpha\Bigl(1+\bigl(\E[\|X^{\epsilon,N}(s)\|_{\HH^\alpha}^2]\bigr)^{\frac12}\Bigr)\bigl(\E[\|X^{\epsilon,N}(s)\|_{\HH^{1-\alpha}}^2]\bigr)^{\frac12}.
\end{align*}
Applying the inverse inequality~\eqref{eq:inverseinequality}, one thus obtains the upper bound
\[
|E^{\epsilon,N}_{1,2,3,2}(t)|\le C_\alpha(T,\varphi)\epsilon \lambda_N^{\frac12-\alpha}\bigl(1+\underset{s\in[0,T]}\sup~\E[\|X^{\epsilon,N}(s)\|_{\HH^\alpha}^2]\bigr).
\]
Second, assume that $\beta\in(1/2,1]$. Applying the inequality~\eqref{eq:DfalphaBis2} on the mapping $f$ with $\beta=1/2$, one has
\begin{align*}
\E[\|Df_N(u^{\epsilon,N}(s)).v^{\epsilon,N}(s)\|_{H^{-1}}]&\le C\E\bigl[\bigl(1+\|u^{\epsilon,N}(s)\|_{H^{\frac12}}\bigr)\|v^{\epsilon,N}(s)\|_{H^{-\frac12}}\bigr]\\
&\le C\bigl(1+\E[\|X^{\epsilon,N}(s)\|_{\HH^{\frac12}}^2]\bigr).
\end{align*}
One thus obtains the upper bound
\[
|E^{\epsilon,N}_{1,2,3,2}(t)|\le C_{\alpha}(T,\varphi)\epsilon\lambda_N^{\max(\frac12-\alpha,0)} \bigl(1+\underset{s\in[0,T]}\sup~\E[\|X^{\epsilon,N}(s)\|_{\HH^{\frac12}}^2]\bigr).
\]
Owing to the moment bounds~\eqref{eq:propo-sdwe-momentbounds-Galerkin} from Proposition~\ref{propo:sdwe-momentbounds} and to the condition~\eqref{eq:init-weak-bound} from Assumption~\ref{ass:init-weak} on the initial value $x_0^\epsilon$, and combining the two cases, one thus obtains the following inequality for the error term $E^{\epsilon,N}_{1,2,3,2}(t)$: for all $\alpha\in(0,\beta)$, there exists $C_\alpha(T,\varphi)\in(0,\infty)$ such that for all $t\in[0,T]$, all $\epsilon\in(0,1)$ and all $N\in\N$ one has
\begin{equation}\label{eq:boundEepsN1232}
\big|E^{\epsilon,N}_{1,2,3,2}(t)\big|\le C_\alpha(T,\varphi)\epsilon\lambda_N^{\max(\frac12-\alpha,0)}.
\end{equation}

{\it Treatment of the error term $E^{\epsilon,N}_{1,2,3,3}(t)$}.
Let $\alpha\in(0,\beta)$, and set $\kappa=(\beta-\alpha)/2$. Applying the inequality~\eqref{eq:propo-regularity-PhiN} from Proposition~\ref{propo:regularity-PhiN} with $k=2$, $\alpha_1=1-\beta+\kappa$, and $\alpha_2=1+\beta-2\kappa$ and the Cauchy--Schwarz inequality, one has
\begin{align*}
|E^{\epsilon,N}_{1,2,3,3}(t)|
&\le C_{\kappa}(T)\epsilon\vvvert\varphi\vvvert_2\int_0^t(t-s)^{-1+\frac{\kappa}{2}}\E\bigl[\|v^{\epsilon,N}(s)\|_{H^{-1+\beta-\kappa}}\|\Lambda u^{\epsilon,N}(s)\|_{H^{-1-\beta+2\kappa}}\bigr]\,ds\\
&\le C_{\kappa}(T,\varphi)\epsilon\int_0^t(t-s)^{-1+\frac{\kappa}{2}}\times\bigl(\E[\|X^{\epsilon,N}(s)\|_{\HH^{\beta-\kappa}}^2]\bigr)^{\frac12}\bigl(\E[\|X^{\epsilon,N}(s)\|_{\HH^{1-\beta+2\kappa}}^2]\bigr)^{\frac12}\,ds.
\end{align*}
On the other hand, note that $\beta-\kappa=\beta-\frac{\beta-\alpha}{2}=\frac{\alpha+\beta}{2}\in(0,\beta)$. Owing to the moment bounds~\eqref{eq:propo-sdwe-momentbounds-Galerkin} from Proposition~\ref{propo:Galerkin} and to the condition~\eqref{eq:init-weak-bound} from Assumption~\ref{ass:init-weak} on the initial value $x_0^\epsilon$, there exists $C_\alpha(T)\in(0,\infty)$ such that for all $N\in\N$ one has
\[
\underset{s\in[0,T]}\sup~\underset{\epsilon\in(0,1)}\sup~\E[\|X^{\epsilon,N}(s)\|_{\HH^{\frac{\alpha+\beta}{2}}}^2]\le C_\alpha(T).
\]
On the other hand, note that $1-\beta+2\kappa=1-\alpha$. If $\alpha<1/2$, owing to the inverse inequality~\eqref{eq:inverseinequality}, for all $N\in\N$ one has
\[
\|x_N\|_{\HH^{1-\alpha}}\le \lambda_N^{\frac12-\alpha}\|x_N\|_{\HH^\alpha},\quad \forall~x_N\in \HH_N.
\]
Therefore, for all $\alpha\in(0,1)$ and for all $N\in\N$ one has
\begin{equation}\label{eq:inverseinequality2}
\|x_N\|_{\HH^{1-\alpha}}\le \lambda_N^{\max(\frac12-\alpha,0)}\|x_N\|_{\HH^{\min(1-\alpha,\alpha)}},\quad \forall~x_N\in \HH_N.
\end{equation}
Given that $\alpha\in(0,\beta)$, one has $\min(1-\alpha,\alpha)\in(0,\beta)$. Owing to the moment bounds~\eqref{eq:propo-sdwe-momentbounds-Galerkin} from Proposition~\ref{propo:Galerkin} and to the condition~\eqref{eq:init-weak-bound} from Assumption~\ref{ass:init-weak} on the initial value $x_0^\epsilon$, there exists $C_\alpha(T)\in(0,\infty)$ such that for all $N\in\N$ one has
\begin{equation}\label{eq:auxmomentbounds}
\underset{s\in[0,T]}\sup~\underset{\epsilon\in(0,1)}\sup~\E[\|X^{\epsilon,N}(s)\|_{\HH^{1-\alpha}}^2]\le \lambda_N^{2\max(\frac12-\alpha,0)}\underset{s\in[0,T]}\sup~\underset{\epsilon\in(0,1)}\sup~\E[\|X^{\epsilon,N}(s)\|_{\HH^{\min(1-\alpha,\alpha)}}^2] \le C_\alpha(T)\lambda_N^{2\max(\frac12-\alpha,0)}.
\end{equation}
Combining the upper bounds, one thus obtains the following inequality for the error term $E^{\epsilon,N}_{1,2,3,3}(t)$: for all $\alpha\in(0,\beta)$, there exists $C_\alpha(T,\varphi)\in(0,\infty)$ such that for all $t\in[0,T]$, all $\epsilon\in(0,1)$ and all $N\in\N$ one has
\begin{equation}\label{eq:boundEepsN1233}
\big|E^{\epsilon,N}_{1,2,3,3}(t)\big|\le C_\alpha(T,\varphi)\epsilon\lambda_N^{\max(\frac12-\alpha,0)}.
\end{equation}

{\it Treatment of the error term $E^{\epsilon,N}_{1,2,3,4}(t)$}.
Let $\alpha\in[0,\beta)$. Applying the the inequality~\eqref{eq:propo-regularity-PhiN} from Proposition~\ref{propo:regularity-PhiN} with $k=2$, $\alpha_1=1-\alpha$ and $\alpha_2=0$, the linear growth property of the mapping $f$ and the Cauchy--Schwarz inequality, one has
\begin{align*}
|E^{\epsilon,N}_{1,2,3,4}(t)|
&\le C_{\alpha}(T)\epsilon\vvvert\varphi\vvvert_2\int_0^t(t-s)^{-\frac{1-\alpha}{2}}\E\bigl[\|v^{\epsilon,N}(s)\|_{H^{-1+\alpha}}\|f_N(u^{\epsilon,N}(s))\|_H\bigr]\,ds\\
&\le C_{\alpha}(T,\varphi)\epsilon\int_0^t(t-s)^{-\frac{1-\alpha}{2}}\bigl(\E[\|X^{\epsilon,N}(s)\|_{\HH^{\alpha}}^2]\bigr)^{\frac12} \bigl(1+ \E[\|X^{\epsilon,N}(s)\|_{\HH}^2] \bigr)^{\frac12}\,ds.
\end{align*}
Applying the moment bounds~\eqref{eq:propo-sdwe-momentbounds-Galerkin} from Proposition~\ref{propo:Galerkin}, one obtains
\[
\underset{s\in[0,T]}\sup~\bigl(\E[\|X^{\epsilon,N}(s)\|_{\HH^{\alpha}}^2]\bigr)^{\frac12} \bigl(1+ \E[\|X^{\epsilon,N}(s)\|_{\HH}^2] \bigr)^{\frac12}\le C_\alpha(T)\bigl(\E[\|x_0^{\epsilon}\|_{\HH^{\alpha}}^2]\bigr)^{\frac12} \bigl(1+ \E[\|x_0^{\epsilon}\|_{\HH}^2] \bigr)^{\frac12}.
\]
Finally, owing to the condition~\eqref{eq:init-weak-bound} from Assumption~\ref{ass:init-weak} on the initial value $x_0^\epsilon$, one thus obtains the following inequality for the error term $E^{\epsilon,N}_{1,2,3,4}(t)$: for all $\alpha\in(0,\beta)$, there exists $C_\alpha(T,\varphi)\in(0,\infty)$ such that for all $t\in[0,T]$, all $\epsilon\in(0,1)$ and all $N\in\N$ one has
\begin{equation}\label{eq:boundEepsN1234}
\big|E^{\epsilon,N}_{1,2,3,4}(t)\big|\le C_\alpha(T,\varphi)\epsilon.
\end{equation}

{\it Treatment of the error term $E^{\epsilon,N}_{1,2,3,5}(t)$}.
Let $\alpha\in[0,\beta)$ and set $\kappa=(\beta-\alpha)/2$. Applying the the inequality~\eqref{eq:propo-regularity-PhiN} from Proposition~\ref{propo:regularity-PhiN} with $k=3$, $\alpha_1=\alpha_2=1-\alpha-\kappa$ and $\alpha_3=2\alpha$, one has
\begin{align*}
\big|E^{\epsilon,N}_{1,2,3,5}(t)\big|
&\le C_\alpha(T)\epsilon\vvvert\varphi\vvvert_3\int_0^t(t-s)^{-1+\kappa} \sum_{n=1}^{N}q_n\|e_n\|_{H^{-1+\alpha+\kappa}}^2\E[\|v^{\epsilon,N}(s)\|_{H^{-2\alpha}}]\,ds\\
&\le C_\alpha(T,\varphi)\epsilon \sum_{n=1}^{N}q_n\lambda_n^{\frac{\alpha+\beta}{2}-1}\int_0^t(t-s)^{-1+\kappa}\E[\|X^{\epsilon,N}(s)\|_{\HH^{1-2\alpha}}]\,ds.
\end{align*}
Owing to the condition~\eqref{eq:condition_beta} from Assumption~\ref{ass:noise}, one has
\[
\underset{N\in\N}\sup~\sum_{n=1}^{N}q_n\lambda_n^{\frac{\alpha+\beta}{2}-1}<\infty.
\]
In addition, applying the inequalities~\eqref{eq:comparenorms} and~\eqref{eq:auxmomentbounds}, one has
\[
\underset{s\in[0,T]}\sup~\underset{\epsilon\in(0,1)}\sup~\E[\|X^{\epsilon,N}(s)\|_{\HH^{1-2\alpha}}]\le \underset{s\in[0,T]}\sup~\underset{\epsilon\in(0,1)}\sup~\E[\|X^{\epsilon,N}(s)\|_{\HH^{1-\alpha}}]\le  C_\alpha(T)\lambda_N^{\max(\frac12-\alpha,0)}.
\]
Combining the upper bounds, one thus obtains the following inequality for the error term $E^{\epsilon,N}_{1,2,3,5}(t)$: for all $\alpha\in(0,\beta)$, there exists $C_\alpha(T,\varphi)\in(0,\infty)$ such that for all $t\in[0,T]$, all $\epsilon\in(0,1)$ and all $N\in\N$ one has
\begin{equation}\label{eq:boundEepsN1235}
\big|E^{\epsilon,N}_{1,2,3,5}(t)\big|\le C_\alpha(T,\varphi)\epsilon\lambda_N^{\max(\frac12-\alpha,0)}.
\end{equation}

{\it Conclusion}.
Recalling the decomposition~\eqref{eq:decompose-E123} of the error term $E^{\epsilon,N}_{1,2,3}(t)$ and combining the upper bounds~\eqref{eq:boundEepsN1231},~\eqref{eq:boundEepsN1232},
~\eqref{eq:boundEepsN1233},~\eqref{eq:boundEepsN1234} and~\eqref{eq:boundEepsN1235}, one thus obtains the following inequality for the error term $E^{\epsilon,N}_{1,2,3}(t)$: for all $\alpha\in(0,\beta)$, there exists $C_\alpha(T,\varphi)\in(0,\infty)$ such that for all $t\in[0,T]$, all $\epsilon\in(0,1)$ and all $N\in\N$ one has
\begin{equation}\label{eq:boundEepsN123}
\big|E^{\epsilon,N}_{1,2,3}(t)\big|\le C_\alpha(T,\varphi)\epsilon\lambda_N^{\max(\frac12-\alpha,0)}.
\end{equation}

$\bullet$ \textbf{Conclusion}.
Recalling the decomposition~\eqref{eq:decompose-E12} of the error term $E^{\epsilon,N}_{1,2}(t)$ and combining the upper bounds~\eqref{eq:boundEepsN121},~\eqref{eq:boundEepsN122} and~\eqref{eq:boundEepsN123}, one obtains the inequality~\eqref{eq:error-T2a} from Lemma~\ref{lem:error-T2}. The proof is thus completed.
\end{proof}

\begin{proof}[Proof of the inequality~\eqref{eq:error-T2b} from Lemma~\ref{lem:error-T2}]
The error term $E^{\epsilon,N}_{1,1,3}(t)$ is defined by~\eqref{eq:defEepsN113}. Owing to the definition~\eqref{eq:defPsi_nm} of the auxiliary mappings $\Psi_{n,m}$, it is decomposed as
\begin{equation}\label{eq:decompose-E113}
E^{\epsilon,N}_{1,1,3}(t)=E^{\epsilon,N}_{1,1,3,1}(t)+E^{\epsilon,N}_{1,1,3,2}(t),
\end{equation}
where the error terms $E^{\epsilon,N}_{1,1,3,1}(t)$ and $E^{\epsilon,N}_{1,1,3,2}(t)$ are defined as
\begin{align}
E^{\epsilon,N}_{1,1,3,1}(t)
&=\epsilon\int_0^t\E\Big[\mathcal{A}^{(N)}\Big( D^2_{uu}\Phi^{(0,N)}(t-s,u^{\epsilon,N}(s)).(v^{\epsilon,N}(s),v^{\epsilon,N}(s))\Big)\Big]\,ds,\label{eq:defEepsN1131}\\
E^{\epsilon,N}_{1,1,3,2}(t)
&=\frac{\epsilon}{2}\sum_{n=1}^Nq_n\int_0^t\E\Big[\mathcal{A}^{(N)}\Big( D^2_{uu}\Phi^{(0,N)}(t-s,u^{\epsilon,N}(s)).(e_n,e_n)\Big)\Big]\,ds.\label{eq:defEepsN1132}
\end{align}

$\bullet$ {\bf Treatment of the error term $E^{\epsilon,N}_{1,1,3,1}(t)$}.
Recalling that the operator $\mathcal{A}^{(N)}$ is defined by~\eqref{eq:genAN}, the error term $E^{\epsilon,N}_{1,1,3,1}(t)$ is decomposed as
\begin{equation}\label{eq:decompose-E1131}
E^{\epsilon,N}_{1,1,3,1}(t)=E^{\epsilon,N}_{1,1,3,1,1}(t)+E^{\epsilon,N}_{1,1,3,1,2}(t)+E^{\epsilon,N}_{1,1,3,1,3}(t),
\end{equation}
where the error terms $E^{\epsilon,N}_{1,1,3,1,1}(t)$, $E^{\epsilon,N}_{1,1,3,1,2}(t)$ and $E^{\epsilon,N}_{1,1,3,1,3}(t)$ are defined as
\begin{align}
E^{\epsilon,N}_{1,1,3,1,1}(t)
&=\epsilon\int_0^t\E\big[ D^3_{u}\Phi^{(0,N)}(t-s,u^{\epsilon,N}(s)).(v^{\epsilon,N}(s),v^{\epsilon,N}(s)), v^{\epsilon,N}(s)\bigr)\big]\,ds,\label{eq:defEepsN11311}\\
E^{\epsilon,N}_{1,1,3,1,2}(t)
&=-2\epsilon\int_0^t\E\big[ D^2_{uu}\Phi^{(0,N)}(t-s,u^{\epsilon,N}(s)).\bigr(v^{\epsilon,N}(s), \Lambda u^{\epsilon,N}(s)\bigr)\big]\,ds,\label{eq:defEepsN11312}\\
E^{\epsilon,N}_{1,1,3,1,3}(t)
&=2\epsilon\int_0^t\E\big[ D^2_{uu}\Phi^{(0,N)}(t-s,u^{\epsilon,N}(s)).\bigr(v^{\epsilon,N}(s), f_N(u^{\epsilon,N}(s))\bigr)\big]\,ds.\label{eq:defEepsN11313}
\end{align}

\textit{Treatment of the error term $E^{\epsilon,N}_{1,1,3,1,1}(t)$}.
Let $\alpha\in(0,\beta)$ and set $\kappa=(\beta-\alpha)/2$. Applying the inequality~\eqref{eq:propo-regularity-PhiN} from Proposition~\ref{propo:regularity-PhiN} with $k=3$, $\alpha_1=\alpha_2=1-\alpha-\kappa$ and $\alpha_3=2\alpha$, one has
\begin{align*}
|E^{\epsilon,N}_{1,1,3,1,1}(t)|
&\le C_{\alpha}(T)\vvvert\varphi\vvvert_3\epsilon\int_0^t(t-s)^{-1+\kappa}\E\bigl[\|v^{\epsilon,N}(s)\|_{H^{-1+(\alpha+\kappa)}}^2\|v^{\epsilon,N}(s)\|_{H^{-2\alpha}}\bigr]\,ds\\
&\le C_{\alpha}(T,\varphi)\epsilon\int_0^t(t-s)^{-1+\kappa}\E\bigl[\|X^{\epsilon,N}(s)\|_{\HH^{\frac{\alpha+\beta}{2}}}^2\|X^{\epsilon,N}(s)\|_{\HH^{1-2\alpha}}\bigr]\,ds.
\end{align*}
Applying the inequality~\eqref{eq:comparenorms} if $\alpha\ge 1/2$ and the inverse inequality~\eqref{eq:inverseinequality} if $\alpha\le 1/2$, one has the upper bound 
\[
\|X^{\epsilon,N}(s)\|_{\HH^{1-2\alpha}}\le C_\alpha \lambda_N^{\max(\frac12-\alpha,0)}\|X^{\epsilon,N}(s)\|_{\HH}.
\]
Applying the Cauchy--Schwarz inequality, owing to the moment bounds~\eqref{eq:propo-sdwe-momentbounds-Galerkin} from Proposition~\ref{propo:sdwe-momentbounds} and to the condition~\eqref{eq:init-weak-bound} from Assumption~\ref{ass:init-weak} on the initial value $x_0^\epsilon$, one thus obtains the following inequality for the error term $E^{\epsilon,N}_{1,2,3,1,1}(t)$: for all $\alpha\in(0,\beta)$, there exists $C_\alpha(T,\varphi)\in(0,\infty)$ such that for all $t\in[0,T]$, all $\epsilon\in(0,1)$ and all $N\in\N$ one has
\begin{equation}\label{eq:boundEepsN11311}
\big|E^{\epsilon,N}_{1,1,3,1,1}(t)\big|\le C_\alpha(T,\varphi)\epsilon\lambda_N^{\max(\frac12-\alpha,0)}.
\end{equation}

\textit{Treatment of the error term $E^{\epsilon,N}_{1,1,3,1,2}(t)$}.
Let $\alpha\in(0,\beta)$ and set $\kappa=(\beta-\alpha)/2$. Applying the inequality~\eqref{eq:propo-regularity-PhiN} from Proposition~\ref{propo:regularity-PhiN} with $k=2$, $\alpha_1=\alpha$ and $\alpha_2=2-\alpha-\kappa$, one has
\begin{align*}
|E^{\epsilon,N}_{1,1,3,1,2}(t)|
&\le C_\alpha(T)\vvvert\varphi\vvvert_2\epsilon\int_0^t(t-s)^{-1+\frac{\kappa}{2}}\E[\|v^{\epsilon,N}(s)\|_{H^{-\alpha}}\|\IL u^{\epsilon,N}(s)\|_{H^{-2+\alpha+\kappa}}]\,ds\\
&\le C_\alpha(T,\varphi)\epsilon\int_0^t(t-s)^{-1+\frac{\kappa}{2}}\E[\|X^{\epsilon,N}(s)\|_{\HH^{1-\alpha}}\|X^{\epsilon,N}(s)\|_{\HH^{\frac{\alpha+\beta}{2}}}]\,ds.
\end{align*}
Applying the inequality~\eqref{eq:inverseinequality2} and the Cauchy--Schwarz inequality, one obtains
\[
|E^{\epsilon,N}_{1,1,3,1,2}(t)|\le C_\alpha(T,\varphi)\epsilon\lambda_N^{\max(\frac12-\alpha,0)}\underset{s\in[0,T]}\sup~\bigl(\E[\|X^{\epsilon,N}(s)\|_{\HH^{\min(\alpha,1-\alpha)}}]\bigr)^{\frac12}\underset{s\in[0,T]}\sup~\bigl(\E[\|X^{\epsilon,N}(s)\|_{\HH^{\frac{\alpha+\beta}{2}}}]\bigr)^{\frac12}.
\]
Owing to the moment bounds~\eqref{eq:propo-sdwe-momentbounds-Galerkin} from Proposition~\ref{propo:sdwe-momentbounds} and to the condition~\eqref{eq:init-weak-bound} from Assumption~\ref{ass:init-weak} on the initial value $x_0^\epsilon$, one thus obtains the following inequality for the error term $E^{\epsilon,N}_{1,2,3,1,2}(t)$: for all $\alpha\in(0,\beta)$, there exists $C_\alpha(T,\varphi)\in(0,\infty)$ such that for all $t\in[0,T]$, all $\epsilon\in(0,1)$ and all $N\in\N$ one has
\begin{equation}\label{eq:boundEepsN11312}
\big|E^{\epsilon,N}_{1,1,3,1,2}(t)\big|\le C_\alpha(T,\varphi)\epsilon\lambda_N^{\max(\frac12-\alpha,0)}.
\end{equation}

\textit{Treatment of the error term $E^{\epsilon,N}_{1,1,3,1,3}(t)$}.
Let $\alpha\in(0,\beta)$. Applying the inequality~\eqref{eq:propo-regularity-PhiN} from Proposition~\ref{propo:regularity-PhiN} with $k=2$, $\alpha_1=2(1-\alpha)$ and $\alpha_2=0$, and the linear growth property of the mapping $f$, one has
\begin{align*}
|E^{\epsilon,N}_{1,1,3,1,3}(t)|
&\le C_\alpha(T)\vvvert\varphi\vvvert_2\epsilon\int_0^t(t-s)^{-1+\alpha}\E[\|v^{\epsilon,N}(s)\|_{H^{-2+2\alpha}}\|f_N(u^{\epsilon,N}(s))\|_{H}]\,ds\\
&\le C_\alpha(T,\varphi)\epsilon\int_0^t(t-s)^{-1+\alpha}\E[\|X^{\epsilon,N}(s)\|_{\HH^{2\alpha-1}}\bigl(1+\|X^{\epsilon,N}(s)\|_{\HH}\bigr)]\,ds.
\end{align*}
Moreover, applying the inequality~\eqref{eq:comparenorms} if $\alpha\ge 1/2$ and inverse inequality~\eqref{eq:inverseinequality} if $\alpha\le 1/2$, one obtains the upper bound
\[
|E^{\epsilon,N}_{1,1,3,1,3}(t)|\le C_\alpha(T,\varphi)\epsilon\lambda_N^{\max(\frac12-\alpha,0)}\bigl(1+\underset{s\in[0,T]}\sup~\E[\|X^{\epsilon,N}(s)\|_{\HH}^2]\bigr).
\]
Owing to the moment bounds~\eqref{eq:propo-sdwe-momentbounds-Galerkin} from Proposition~\ref{propo:sdwe-momentbounds} and to the condition~\eqref{eq:init-weak-bound} from Assumption~\ref{ass:init-weak} on the initial value $x_0^\epsilon$, one thus obtains the following inequality for the error term $E^{\epsilon,N}_{1,1,3,1,3}(t)$: for all $\alpha\in(0,\beta)$, there exists $C_\alpha(T,\varphi)\in(0,\infty)$ such that for all $t\in[0,T]$, all $\epsilon\in(0,1)$ and all $N\in\N$ one has
\begin{equation}\label{eq:boundEepsN11313}
\big|E^{\epsilon,N}_{1,1,3,1,3}(t)\big|\le C_\alpha(T,\varphi)\epsilon\lambda_N^{\max(\frac12-\alpha,0)}.
\end{equation}

{\it Conclusion}.
Recalling the decomposition~\eqref{eq:decompose-E1131} of the error term $E^{\epsilon,N}_{1,1,3,1}(t)$ and combining the upper bounds~\eqref{eq:boundEepsN11311},~\eqref{eq:boundEepsN11312} and~\eqref{eq:boundEepsN11313}, one thus obtains the following inequality for the error term $E^{\epsilon,N}_{1,1,3,1}(t)$: for all $\alpha\in(0,\beta)$, there exists $C_\alpha(T,\varphi)\in(0,\infty)$ such that for all $t\in[0,T]$, all $\epsilon\in(0,1)$ and all $N\in\N$ one has
\begin{equation}\label{eq:boundEepsN1131}
\big|E^{\epsilon,N}_{1,1,3,1}(t)\big|\le C_\alpha(T,\varphi)\epsilon\lambda_N^{\max(\frac12-\alpha,0)}.
\end{equation}

$\bullet$ {\bf Treatment of the error term $E^{\epsilon,N}_{1,1,3,2}(t)$}.
Recalling that the operator $\mathcal{A}^{(N)}$ is defined by~\eqref{eq:genAN}, the error term $E^{\epsilon,N}_{1,1,3,2}(t)$ is written as
\[
E^{\epsilon,N}_{1,1,3,2}(t)
=\frac{\epsilon}{2}\sum_{n=1}^Nq_n\int_0^t\E\Big[\Big( D^3_{u}\Phi^{(0,N)}(t-s,u^{\epsilon,N}(s)).\bigl(e_n,e_n,v^{\epsilon,N}(s)\bigr)\Big)\Big]\,ds.
\]
Let $\alpha\in(0,\beta)$ and set $\kappa=(\beta-\alpha)/2$. Applying the inequality~\eqref{eq:propo-regularity-PhiN} from Proposition~\ref{propo:regularity-PhiN} with $k=3$, $\alpha_1=\alpha_2=1/2$ and $\alpha_3=1-\alpha-\kappa$, one has
\begin{align*}
|E^{\epsilon,N}_{1,1,3,2}(t)|
&\le C_\alpha(T)\vvvert\varphi\vvvert_3\epsilon\sum_{n=1}^{N}q_n\|e_n\|_{H^{-\frac12}}^2\int_{0}^{t}(t-s)^{-1+\frac{\alpha+\kappa}{2}}\E[\|v^{\epsilon,N}(s)\|_{H^{-1+\alpha+\kappa}}]\,ds\\
&\le C_\alpha(T,\varphi)\epsilon\sum_{n=1}^{N}q_n\lambda_n^{-\frac12}\underset{s\in[0,T]}\sup~\E[\|X^{\epsilon,N}(s)\|_{\HH^{\frac{\alpha+\beta}{2}}}].
\end{align*}
Moreover, note that one has
\[
\sum_{n=1}^{N}q_n\lambda_n^{-\frac12}=\sum_{n=1}^{N}q_n\lambda_n^{\alpha-1}\lambda_n^{\frac12-\alpha}\le C_\alpha\lambda_N^{\max(\frac12-\alpha,0)}\sum_{n=1}^{\infty}q_n\lambda_n^{\alpha-1}
\]
with $\sum_{n=1}^{\infty}q_n\lambda_n^{\alpha-1}<\infty$ owing to the condition~\eqref{eq:condition_beta} from Assumption~\ref{ass:noise}.

Owing to the moment bounds~\eqref{eq:propo-sdwe-momentbounds-Galerkin} from Proposition~\ref{propo:sdwe-momentbounds} and to the condition~\eqref{eq:init-weak-bound} from Assumption~\ref{ass:init-weak} on the initial value $x_0^\epsilon$, one thus obtains the following inequality for the error term $E^{\epsilon,N}_{1,1,3,2}(t)$: for all $\alpha\in(0,\beta)$, there exists $C_\alpha(T,\varphi)\in(0,\infty)$ such that for all $t\in[0,T]$, all $\epsilon\in(0,1)$ and all $N\in\N$ one has
\begin{equation}\label{eq:boundEepsN1132}
\big|E^{\epsilon,N}_{1,1,3,2}(t)\big|\le C_\alpha(T,\varphi)\epsilon\lambda_N^{\max(\frac12-\alpha,0)}.
\end{equation}

$\bullet$ {\bf Conclusion}.
Recalling the decomposition~\eqref{eq:decompose-E113} of the error term $E^{\epsilon,N}_{1,1,3}(t)$ and combining the upper bounds~\eqref{eq:boundEepsN1131} and~\eqref{eq:boundEepsN1132}, one obtains the inequality~\eqref{eq:error-T2b} from Lemma~\ref{lem:error-T2}. The proof is thus completed.

\end{proof}

\begin{proof}[Proof of the inequality~\eqref{eq:error-T3a} from Lemma~\ref{lem:error-T3}]
Owing to the definition~\eqref{eq:defPsi_nm} of the mappings $\Psi_{n,m}$, the error term $E^{\epsilon,N}_{1,1,1}(t)$ defined by~\eqref{eq:defEepsN111} is decomposed as
\begin{equation}\label{eq:decompose-E111}
E^{\epsilon,N}_{1,1,1}(t)=E^{\epsilon,N}_{1,1,1,1}(t)+E^{\epsilon,N}_{1,1,1,2}(t),
\end{equation}
where the error terms $E^{\epsilon,N}_{1,1,1,1}(t)$ and $E^{\epsilon,N}_{1,1,1,2}(t)$ are defined as
\begin{align}
E^{\epsilon,N}_{1,1,1,1}(t)&=\frac{\epsilon^2}{2}\E\bigl[D^2_{uu}\Phi^{(0,N)}(t,u^{\epsilon,N}(0)).\bigl(v^{\epsilon,N}(0),v^{\epsilon,N}(0)\bigr)-D^2_{uu}\Phi^{(0,N)}(0,u^{\epsilon,N}(t)).\bigl(v^{\epsilon,N}(t),v^{\epsilon,N}(t)\bigr)\bigr],\label{eq:defEepsN1111}\\
E^{\epsilon,N}_{1,1,1,2}(t)&=\frac{\epsilon^2}{4}\sum_{n=1}^Nq_n\E\bigl[D^2_{uu}\Phi^{(0,N)}(0,u^{\epsilon,N}(t)).(e_n,e_n)-D^2_{uu}\Phi^{(0,N)}(t,u^{\epsilon,N}(0)).(e_n,e_n)\bigr].\label{eq:defEepsN1112}
\end{align}

$\bullet$ {\bf Treatment of the error term $E^{\epsilon,N}_{1,1,1,1}(t)$}.
Applying the the inequality~\eqref{eq:propo-regularity-PhiN} from Proposition~\ref{propo:regularity-PhiN} with $k=2$ and $\alpha_1=\alpha_2=0$, one has
\begin{align*}
|E^{\epsilon,N}_{1,1,1,1}(t)|&\le C(T)\vvvert\varphi\vvvert_2\epsilon^2\bigl(\E[\|v^{\epsilon,N}(0)\|_{H}^{2}]+\E[\|v^{\epsilon,N}(t)\|_{H}^{2}]\bigr)\\
&\le C(T,\varphi)\epsilon^2\underset{s\in[0,T]}\sup~\E[\|X^{\epsilon,N}(t)\|_{\HH^1}^{2}].
\end{align*}
Let $\alpha\in(0,\beta)$. Applying the inequality~\eqref{eq:comparenorms} if $\alpha\ge 1/2$ or the inverse inequality~\eqref{eq:inverseinequality} if $\alpha\le 1/2$, one obtains
\[
\underset{s\in[0,T]}\sup~\E[\|X^{\epsilon,N}(t)\|_{\HH^1}^{2}]\le \lambda_n^{1-\alpha}\underset{s\in[0,T]}\sup~\E[\|X^{\epsilon,N}(t)\|_{\HH^\alpha}^{2}].
\]
Owing to the moment bounds~\eqref{eq:propo-sdwe-momentbounds-Galerkin} from Proposition~\ref{propo:sdwe-momentbounds} and to the condition~\eqref{eq:init-weak-bound} from Assumption~\ref{ass:init-weak} on the initial value $x_0^\epsilon$, one thus obtains the following inequality for the error term $E^{\epsilon,N}_{1,1,1,1}(t)$: for all $\alpha\in(0,\beta)$, there exists $C_\alpha(T,\varphi)\in(0,\infty)$ such that for all $t\in[0,T]$, all $\epsilon\in(0,1)$ and all $N\in\N$ one has
\begin{equation}\label{eq:boundEepsN1111}
\big|E^{\epsilon,N}_{1,1,1,1}(t)\big|\le C_\alpha(T,\varphi)\epsilon^2\lambda_N^{1-\alpha}.
\end{equation}

$\bullet$ {\bf Treatment of the error term $E^{\epsilon,N}_{1,1,1,2}(t)$}.
Applying the the inequality~\eqref{eq:propo-regularity-PhiN} from Proposition~\ref{propo:regularity-PhiN} with $k=2$ and $\alpha_1=\alpha_2=0$, one has
\begin{align*}
|E^{\epsilon,N}_{1,1,1,2}(t)|&\le C(T)\vvvert\varphi\vvvert_2\epsilon^2\sum_{n=1}^{N}q_n.
\end{align*}
Let $\alpha\in(0,\beta)$, and note that
\[
\sum_{n=1}^{N}q_n=\sum_{n=1}^{N}q_n\lambda_n^{\alpha-1}\lambda_n^{1-\alpha}\le \lambda_{N}^{1-\alpha}\sum_{n=1}^{\infty}q_n\lambda_n^{\alpha-1},
\]
with $\sum_{n=1}^{\infty}q_n\lambda_n^{\alpha-1}<\infty$ owing to the condition~\eqref{eq:condition_beta} from Assumption~\ref{ass:noise}. One thus obtains the following inequality for the error term $E^{\epsilon,N}_{1,1,1,2}(t)$: for all $\alpha\in(0,\beta)$, there exists $C_\alpha(T,\varphi)\in(0,\infty)$ such that for all $t\in[0,T]$, all $\epsilon\in(0,1)$ and all $N\in\N$ one has
\begin{equation}\label{eq:boundEepsN1112}
\big|E^{\epsilon,N}_{1,1,1,2}(t)\big|\le C_\alpha(T,\varphi)\epsilon^2\lambda_N^{1-\alpha}.
\end{equation}

$\bullet$ {\bf Conclusion}.
Recalling the decomposition~\eqref{eq:decompose-E111} of the error term $E^{\epsilon,N}_{1,1,1}(t)$ and combining the upper bounds~\eqref{eq:boundEepsN1111} and~\eqref{eq:boundEepsN1112}, one obtains the inequality~\eqref{eq:error-T3a} from Lemma~\ref{lem:error-T3}. The proof is thus completed.
\end{proof}

\begin{proof}[Proof of the inequality~\eqref{eq:error-T3b} from Lemma~\ref{lem:error-T3}]
Owing to the definition~\eqref{eq:defPsi_nm} of the mappings $\Psi_{n,m}$, the error term $E^{\epsilon,N}_{1,1,2}(t)$ defined by~\eqref{eq:defEepsN112} is decomposed as
\begin{equation}\label{eq:decompose-E112}
E^{\epsilon,N}_{1,1,2}(t)=E^{\epsilon,N}_{1,1,2,1}(t)+E^{\epsilon,N}_{1,1,2,2}(t),
\end{equation}
where the error terms $E^{\epsilon,N}_{1,1,2,1}(t)$ and $E^{\epsilon,N}_{1,1,2,2}(t)$ are defined as
\begin{align}
E^{\epsilon,N}_{1,1,2,1}(t)
&=\frac{\epsilon^2}{2}\int_{0}^{t}\E\bigl[\partial_sD^2_{uu}\Phi^{(0,N)}(t-s,u^{\epsilon,N}(s)).\bigl(v^{\epsilon,N}(s),v^{\epsilon,N}(s)\bigr)\bigr]\,ds,\label{eq:defEepsN1121}\\
E^{\epsilon,N}_{1,1,2,2}(t)
&=-\frac{\epsilon^2}{4}\sum_{n=1}^Nq_n\int_{0}^{t}\E\bigl[\partial_sD^2_{uu}\Phi^{(0,N)}(t-s,u^{\epsilon,N}(s)).(e_n,e_n)\bigr]\,ds.\label{eq:defEepsN1122}
\end{align}

Note that owing to the Kolmogorov equation \eqref{eq:kolmogorovPhiN}, the definition \eqref{eq:LN} of the infinitesimal generator $\mathbb{L}^{(N)}$, and the Leibniz formula, for all $s\ge 0$ and all $u,v\in H_N$, one has
\begin{align}
\partial_sD_{uu}^2\Phi^{(0,N)}(s,u).(v,v)&=D_u\Phi^{(0,N)}(s,u).\bigl(D^2f_N(u).(v,v)\bigr)+2D_{uu}^2\Phi^{(0,N)}(s,u).\bigl(-\IL v+Df_N(u).v,v\bigr)\label{eq:partial_sD2}\\
&+D_u^3\Phi^{(0,N)}(u,s).\bigl(-\IL u+f_N(u),v,v\bigr)+\sum_{n=1}^{N}\frac{q_n}{2}D_u^4\Phi^{(0,N)}(s,u).\bigl(e_n,e_n,v,v\bigr).\nonumber
\end{align}

$\bullet$ {\bf Treatment of the error term $E^{\epsilon,N}_{1,1,2,1}(t)$}. Owing to the expression~\eqref{eq:partial_sD2} for $\partial_sD_{uu}^2\Phi^{(0,N)}$, the error term $E^{\epsilon,N}_{1,1,2,1}(t)$ defined by~\eqref{eq:defEepsN1121} is decomposed as
\begin{equation}\label{eq:decompose-E1121}
E^{\epsilon,N}_{1,1,2,1}(t)=E^{\epsilon,N}_{1,1,2,1,1}(t)+E^{\epsilon,N}_{1,1,2,1,2}(t)+E^{\epsilon,N}_{1,1,2,1,3}(t)+E^{\epsilon,N}_{1,1,2,1,4}(t),
\end{equation}
where the error terms $E^{\epsilon,N}_{1,1,2,1,1}(t)$, $E^{\epsilon,N}_{1,1,2,1,2}(t)$, $E^{\epsilon,N}_{1,1,2,1,3}(t)$ and $E^{\epsilon,N}_{1,1,2,1,4}(t)$ are defined as
\begin{align}
E^{\epsilon,N}_{1,1,2,1,1}(t)
&=\frac{\epsilon^2}{2}\int_{0}^{t}\E\bigl[D_u\Phi^{(0,N)}(t-s,u^{\epsilon,N}(s)).\bigl(D^2f_N(u^{\epsilon,N}(s)).(v^{\epsilon,N}(s),v^{\epsilon,N}(s))\bigr)\bigr]\,ds,\label{eq:defEepsN11211}
\\
E^{\epsilon,N}_{1,1,2,1,2}(t)
&=\epsilon^2\int_{0}^{t}\E\bigl[D_{uu}^2\Phi^{(0,N)}(t-s,u^{\epsilon,N}(s)).\bigl(-\IL v^{\epsilon,N}(s)+Df_N(u^{\epsilon,N}(s)).v^{\epsilon,N}(s),v^{\epsilon,N}(s)\bigr)\bigr]\,ds,\label{eq:defEepsN11212}
\\
E^{\epsilon,N}_{1,1,2,1,3}(t)
&=\frac{\epsilon^2}{2}\int_{0}^{t}\E\bigl[D_u^3\Phi^{(0,N)}(t-s,u^{\epsilon,N}(s)).\bigl(-\IL u^{\epsilon,N}(s)+f_N(u^{\epsilon,N}(s)),v^{\epsilon,N}(s),v^{\epsilon,N}(s)\bigr)\bigr]\,ds,\label{eq:defEepsN11213}
\\
E^{\epsilon,N}_{1,1,2,1,4}(t)
&=\frac{\epsilon^2}{4}\int_{0}^{t}\sum_{n=1}^{N}q_n\E\bigl[D_u^4\Phi^{(0,N)}(t-s,u^{\epsilon,N}(s)).\bigl(e_n,e_n,v^{\epsilon,N}(s),v^{\epsilon,N}(s)\bigr)\bigr]\,ds.\label{eq:defEepsN11214}
\end{align}

\textit{Treatment of the error term $E^{\epsilon,N}_{1,1,2,1,1}(t)$}.
Let $\alpha\in(0,\beta)$. Applying the inequality \eqref{eq:propo-regularity-PhiN} from Proposition~\ref{propo:regularity-PhiN} with $k=1$ and $\alpha_1=0$, and Assumption~\ref{ass:mapping-f}, one has
\begin{align*}
|E^{\epsilon,N}_{1,1,2,1,1}(t)|&\le C(T)\vvvert\varphi\vvvert_1\epsilon^2 \int_0^t \E\bigl[ \bigl\|D^2f_N(u^{\epsilon,N}(s)).(v^{\epsilon,N}(s),v^{\epsilon,N}(s))\bigr\|_H \bigr] \,ds\\
&\le C(T,\varphi)\epsilon^2\int_0^t \E\bigl[\|v^{\epsilon,N}(s)\|^2_H\bigr]\,ds\\
&\le C(T,\varphi)\epsilon^2\underset{s\in[0,T]}\sup~\E\bigl[\|X^{\epsilon,N}(s)\|^2_{\HH^1}\bigr].
\end{align*}
Applying the inverse inequality~\eqref{eq:inverseinequality} and the moment bounds~\eqref{eq:propo-sdwe-momentbounds-Galerkin} from Proposition~\ref{propo:Galerkin}, one has
\[
\underset{s\in[0,T]}\sup~\E\bigl[\|X^{\epsilon,N}(s)\|^2_{\HH^1}\bigr]\le  \lambda_N^{1-\alpha}\underset{s\in[0,T]}\sup~\E\bigl[\|X^{\epsilon,N}(s)\|^2_{\HH^{\alpha}}\bigr]\le C_\alpha(T)\lambda_N^{1-\alpha}.
\]
One thus obtains the following upper bound for the error term $E^{\epsilon,N}_{1,1,2,1,1}(t)$: there exists $C_\alpha(T,\varphi)\in(0,\infty)$ such that for all $t\in[0,T]$, all $\epsilon\in(0,1)$ and all $N\in\N$ one has
\begin{equation}\label{eq:boundEepsN11211}
\big|E^{\epsilon,N}_{1,1,2,1,1}(t)\big|\le C_\alpha(T,\varphi)\epsilon^2\lambda_N^{1-\alpha}.
\end{equation}

\textit{Treatment of the error term $E^{\epsilon,N}_{1,1,2,1,2}(t)$}.
The error term $E^{\epsilon,N}_{1,1,2,1,2}(t)$ defined by~\eqref{eq:defEepsN11212} is decomposed as
\[
E^{\epsilon,N}_{1,1,2,1,2}(t)=E^{\epsilon,N}_{1,1,2,1,2,1}(t)+E^{\epsilon,N}_{1,1,2,1,2,2}(t),
\]
where the error terms $E^{\epsilon,N}_{1,1,2,1,2,1}(t)$ and $E^{\epsilon,N}_{1,1,2,1,2,2}(t)$ are defined as
\begin{align}
E^{\epsilon,N}_{1,1,2,1,2,1}(t)
&=-\epsilon^2\int_{0}^{t}\E\bigl[D_{uu}^2\Phi^{(0,N)}(t-s,u^{\epsilon,N}(s)).\bigl(\IL v^{\epsilon,N}(s),v^{\epsilon,N}(s)\bigr)\bigr]\,ds,\label{eq:defEepsN112121}\\
E^{\epsilon,N}_{1,1,2,1,2,2}(t)
&=\epsilon^2\int_{0}^{t}\E\bigl[D_{uu}^2\Phi^{(0,N)}(t-s,u^{\epsilon,N}(s)).\bigl(Df_N(u^{\epsilon,N}(s)).v^{\epsilon,N}(s),v^{\epsilon,N}(s)\bigr)\bigr]\,ds.\label{eq:defEepsN112122}
\end{align}
Let $\alpha\in(0,\beta)$ and $\kappa=(\beta-\alpha)/2\in(0,1)$.

On the one hand,  applying the inequality \eqref{eq:propo-regularity-PhiN} from Proposition~\ref{propo:regularity-PhiN} with $k=2$, $\alpha_1=-2+\kappa$ and $\alpha_2=0$, one has
\begin{align*}
|E^{\epsilon,N}_{1,1,2,1,2,1}(t)|
&\le C_\kappa(T)\vvvert\varphi\vvvert_2\epsilon^2\int_{0}^{t}(t-s)^{-1+\frac{\kappa}{2}}\E[\|\IL v^{\epsilon,N}(s)\|_{H^{-2+\kappa}}\|v^{\epsilon,N}(s)\|_{H}]\,ds\\
&\le C_{\kappa}(T,\varphi)\epsilon^2\int_{0}^{t}(t-s)^{-1+\frac{\kappa}{2}}\E[\|v^{\epsilon,N}(s)\|_{H^\kappa}\|v^{\epsilon,N}(s)\|_{H}]\,ds\\
&\le C_{\kappa}(T,\varphi)\epsilon^2\underset{s\in[0,T]}\sup~\E[\|X^{\epsilon,N}(s)\|_{\HH^{1+\kappa}}^2].
\end{align*}

On the other hand, applying the inequality \eqref{eq:propo-regularity-PhiN} from Proposition~\ref{propo:regularity-PhiN} with $k=2$, $\alpha_1=\alpha_2=0$, and the Lipschitz continuity condition on the mapping $f$ from Assumption~\ref{ass:fLip}, one has
\begin{align*}
|E^{\epsilon,N}_{1,1,2,1,2,2}(t)|
&\le C(T)\vvvert\varphi\vvvert_2\epsilon^2\int_{0}^{t}\E[\|Df_N(u^{\epsilon,N}(s)).v^{\epsilon,N}(s)\|_{H}\|v^{\epsilon,N}(s)\|_{H}]\,ds\\
&\le C(T,\varphi)\epsilon^2\int_{0}^{t}\E[\|v^{\epsilon,N}(s)\|_{H}^2]\,ds\\
&\le C(T,\varphi)\epsilon^2\underset{s\in[0,T]}\sup~\E[\|X^{\epsilon,N}(s)\|_{\HH^{1}}^2]\\
&\le C(T,\varphi)\epsilon^2\underset{s\in[0,T]}\sup~\E[\|X^{\epsilon,N}(s)\|_{\HH^{1+\kappa}}^2].
\end{align*}
Applying the inverse inequality~\eqref{eq:inverseinequality}, one has
\[
\underset{s\in[0,T]}\sup~\E[\|X^{\epsilon,N}(s)\|_{\HH^{1+\kappa}}^2]\le \lambda_N^{1-\alpha}\underset{s\in[0,T]}\sup~\E[\|X^{\epsilon,N}(s)\|_{\HH^{\alpha+\kappa}}^2]=\lambda_N^{1-\alpha}\underset{s\in[0,T]}\sup~\E[\|X^{\epsilon,N}(s)\|_{\HH^{\frac{\alpha+\beta}{2}}}^2].
\]
Owing to the moment bounds~\eqref{eq:propo-sdwe-momentbounds-Galerkin} from Proposition~\ref{propo:sdwe-momentbounds} and to the condition~\eqref{eq:init-weak-bound} from Assumption~\ref{ass:init-weak} on the initial value $x_0^\epsilon$, one thus obtains the following inequality for the error term $E^{\epsilon,N}_{1,1,2,1,2}(t)$: for all $\alpha\in(0,\beta)$, there exists $C_\alpha(T,\varphi)\in(0,\infty)$ such that for all $t\in[0,T]$, all $\epsilon\in(0,1)$ and all $N\in\N$ one has
\begin{equation}\label{eq:boundEepsN11212}
\big|E^{\epsilon,N}_{1,1,2,1,2}(t)\big|\le C_\alpha(T,\varphi)\epsilon^2\lambda_N^{1-\alpha}.
\end{equation}

\textit{Treatment of the error term $E^{\epsilon,N}_{1,1,2,1,3}(t)$}.
The error term $E^{\epsilon,N}_{1,1,2,1,3}(t)$ defined by~\eqref{eq:defEepsN11212} is decomposed as
\[
E^{\epsilon,N}_{1,1,2,1,3}(t)=E^{\epsilon,N}_{1,1,2,1,3,1}(t)+E^{\epsilon,N}_{1,1,2,1,3,2}(t),
\]
where the error terms $E^{\epsilon,N}_{1,1,2,1,3,1}(t)$ and $E^{\epsilon,N}_{1,1,2,1,3,2}(t)$ are defined as
\begin{align}
E^{\epsilon,N}_{1,1,2,1,3,1}(t)
&=-\frac{\epsilon^2}{2}\int_{0}^{t}\E\bigl[D_u^3\Phi^{(0,N)}(t-s,u^{\epsilon,N}(s)).\bigl(\IL u^{\epsilon,N}(s),v^{\epsilon,N}(s),v^{\epsilon,N}(s)\bigr)\bigr]\,ds,\label{eq:defEepsN112131}\\
E^{\epsilon,N}_{1,1,2,1,3,2}(t)
&=\frac{\epsilon^2}{2}\int_{0}^{t}\E\bigl[D_u^3\Phi^{(0,N)}(t-s,u^{\epsilon,N}(s)).\bigl(f_N(u^{\epsilon,N}(s)),v^{\epsilon,N}(s),v^{\epsilon,N}(s)\bigr)\bigr]\,ds.\label{eq:defEepsN112132}
\end{align}
Let $\alpha\in(0,\beta)$.

On the one hand, applying the inequality \eqref{eq:propo-regularity-PhiN} from Proposition~\ref{propo:regularity-PhiN} with $k=3$, $\alpha_1=2-\alpha$, $\alpha_2=\alpha_3=0$, one has
\begin{align*}
|E^{\epsilon,N}_{1,1,2,1,3,1}(t)|
&\le C_\alpha(T)\vvvert\varphi\vvvert_3\epsilon^2\int_{0}^{t}(t-s)^{-1+\frac{\alpha}{2}}\E[\|\IL u^{\epsilon,N}(s)\|_{H^{-2+\alpha}}\|v^{\epsilon,N}(s)\|_{H}^2]\,ds\\
&\le C_\alpha(T,\varphi)\epsilon^2\int_{0}^{t}(t-s)^{-1+\frac{\alpha}{2}}\E[\|u^{\epsilon,N}(s)\|_{H^{\alpha}}\|v^{\epsilon,N}(s)\|_{H}^2]\,ds\\
&\le C_\alpha(T,\varphi)\epsilon^2\underset{s\in[0,T]}\sup~\E[\|X^{\epsilon,N}(s)\|_{\HH^{\alpha}}\|X^{\epsilon,N}(s)\|_{\HH^1}^2].
\end{align*}

On the other hand, applying the inequality \eqref{eq:propo-regularity-PhiN} from Proposition~\ref{propo:regularity-PhiN} with $k=3$, $\alpha_1=\alpha_2=\alpha_3=0$, and recalling that the mapping $f$ grows at most linearly owing to Assumption~\ref{ass:fLip}, one has
\begin{align*}
|E^{\epsilon,N}_{1,1,2,1,3,2}(t)|
&\le C(T)\vvvert\varphi\vvvert_3\epsilon^2\int_{0}^{t}\E[\|f_N(u^{\epsilon,N}(s))\|_{H}\|v^{\epsilon,N}(s)\|_{H}^2]\,ds\\
&\le C(T,\varphi)\epsilon^2\int_{0}^{t}\E[\bigl(1+\|u^{\epsilon,N}(s)\|_{H}\bigr)\|v^{\epsilon,N}(s)\|_{H}^2]\,ds\\
&\le C(T,\varphi)\epsilon^2\underset{s\in[0,T]}\sup~\E\bigl[\bigl(1+\|X^{\epsilon,N}(s)\|_{\HH}\bigr)\|X^{\epsilon,N}(s)\|_{\HH^1}^2\bigr].
\end{align*}
Applying the inverse inequality~\eqref{eq:inverseinequality}, one obtains the upper bound
\[
|E^{\epsilon,N}_{1,1,2,1,3,1}(t)|+|E^{\epsilon,N}_{1,1,2,1,3,2}(t)|\le C_\alpha(T,\varphi)\epsilon^2\lambda_N^{1-\alpha}\underset{s\in[0,T]}\sup~\E\bigl[\bigl(1+\|X^{\epsilon,N}(s)\|_{\HH^{\alpha}}^3\bigr)\bigr].
\]
Owing to the moment bounds~\eqref{eq:propo-sdwe-momentbounds-Galerkin} from Proposition~\ref{propo:sdwe-momentbounds} and to the condition~\eqref{eq:init-weak-bound} from Assumption~\ref{ass:init-weak} on the initial value $x_0^\epsilon$, one thus obtains the following inequality for the error term $E^{\epsilon,N}_{1,1,2,1,3}(t)$: for all $\alpha\in(0,\beta)$, there exists $C_\alpha(T,\varphi)\in(0,\infty)$ such that for all $t\in[0,T]$, all $\epsilon\in(0,1)$ and all $N\in\N$ one has
\begin{equation}\label{eq:boundEepsN11213}
\big|E^{\epsilon,N}_{1,1,2,1,3}(t)\big|\le C_\alpha(T,\varphi)\epsilon^2\lambda_N^{1-\alpha}.
\end{equation}

\textit{Treatment of the error term $E^{\epsilon,N}_{1,1,2,1,4}(t)$}.
Let $\alpha\in(0,\beta)$. Applying the inequality \eqref{eq:propo-regularity-PhiN} from Proposition~\ref{propo:regularity-PhiN} with $k=4$, $\alpha_1=\alpha_2=1-\alpha$, $\alpha_3=\alpha_4=0$, one has
\begin{align*}
|E^{\epsilon,N}_{1,1,2,1,2,4}(t)|
&\le C_\alpha(T)\vvvert\varphi\vvvert_4\epsilon^2\int_{0}^{t}(t-s)^{-1+\alpha}\sum_{n=1}^{N}q_n\|e_n\|_{H^{\alpha-1}}^2\E[\|v^{\epsilon,N}(s)\|_H^2]\,ds\\
&\le C_\alpha(T,\varphi)\epsilon^2\sum_{n=1}^{N}q_n\lambda_n^{\alpha-1} \underset{s\in[0,T]}\sup~\E[\|X^{\epsilon,N}(s)\|_{\HH^1}^2].
\end{align*}
Applying the condition~\eqref{eq:condition_beta} from Assumption~\ref{ass:noise} and applying the inverse inequality~\eqref{eq:inverseinequality}, one obtains the upper bound
\[
|E^{\epsilon,N}_{1,1,2,1,4}(t)|\le C_\alpha(T,\varphi)\epsilon^2\lambda_N^{1-\alpha}\underset{s\in[0,T]}\sup~\E[\|X^{\epsilon,N}(s)\|_{\HH^\alpha}^2].
\]
Owing to the moment bounds~\eqref{eq:propo-sdwe-momentbounds-Galerkin} from Proposition~\ref{propo:sdwe-momentbounds} and to the condition~\eqref{eq:init-weak-bound} from Assumption~\ref{ass:init-weak} on the initial value $x_0^\epsilon$, one thus obtains the following inequality for the error term $E^{\epsilon,N}_{1,1,2,1,4}(t)$: for all $\alpha\in(0,\beta)$, there exists $C_\alpha(T,\varphi)\in(0,\infty)$ such that for all $t\in[0,T]$, all $\epsilon\in(0,1)$ and all $N\in\N$ one has
\begin{equation}\label{eq:boundEepsN11214}
\big|E^{\epsilon,N}_{1,1,2,1,4}(t)\big|\le C_\alpha(T,\varphi)\epsilon^2\lambda_N^{1-\alpha}.
\end{equation}

{\it Conclusion}.
Recalling the decomposition~\eqref{eq:decompose-E1121} of the error term $E^{\epsilon,N}_{1,1,2,1}(t)$ and combining the upper bounds~\eqref{eq:boundEepsN11211},~\eqref{eq:boundEepsN11212}
,~\eqref{eq:boundEepsN11213} and~\eqref{eq:boundEepsN11214}, one thus obtains the following inequality for the error term $E^{\epsilon,N}_{1,1,2,1}(t)$: for all $\alpha\in(0,\beta)$, there exists $C_\alpha(T,\varphi)\in(0,\infty)$ such that for all $t\in[0,T]$, all $\epsilon\in(0,1)$ and all $N\in\N$ one has
\begin{equation}\label{eq:boundEepsN1121}
\big|E^{\epsilon,N}_{1,1,2,1}(t)\big|\le C_\alpha(T,\varphi)\epsilon^2\lambda_N^{1-\alpha}.
\end{equation}

$\bullet$ {\bf Treatment of the error term $E^{\epsilon,N}_{1,1,2,2}(t)$}.
Owing to the expression~\eqref{eq:partial_sD2} for $\partial_sD_{uu}^2\Phi^{(0,N)}$, the error term $E^{\epsilon,N}_{1,1,2,2}(t)$ defined by~\eqref{eq:defEepsN1122} is decomposed as
\begin{equation}\label{eq:decompose-E1122}
E^{\epsilon,N}_{1,1,2,2}(t)=E^{\epsilon,N}_{1,1,2,2,1}(t)+E^{\epsilon,N}_{1,1,2,2,2}(t)+E^{\epsilon,N}_{1,1,2,2,3}(t)+E^{\epsilon,N}_{1,1,2,2,4}(t),
\end{equation}
where the error terms $E^{\epsilon,N}_{1,1,2,2,1}(t)$, $E^{\epsilon,N}_{1,1,2,2,2}(t)$, $E^{\epsilon,N}_{1,1,2,2,3}(t)$ and $E^{\epsilon,N}_{1,1,2,2,4}(t)$ are defined as
\begin{align}
E^{\epsilon,N}_{1,1,2,2,1}(t)
&=-\frac{\epsilon^2}{4}\sum_{n=1}^{N}q_n\int_{0}^{t}\E\bigl[D_u\Phi^{(0,N)}(t-s,u^{\epsilon,N}(s)).\bigl(D^2f_N(u^{\epsilon,N}(s)).(e_n,e_n)\bigr)\bigr]\,ds,\label{eq:defEepsN11221}
\\
E^{\epsilon,N}_{1,1,2,2,2}(t)
&=-\frac{\epsilon^2}{2}\sum_{n=1}^{N}q_n\int_{0}^{t}\E\bigl[D_{uu}^2\Phi^{(0,N)}(t-s,u^{\epsilon,N}(s)).\bigl(-\IL e_n+Df_N(u^{\epsilon,N}(s)).e_n,e_n\bigr)\bigr]\,ds,\label{eq:defEepsN11222}
\\
E^{\epsilon,N}_{1,1,2,2,3}(t)
&=-\frac{\epsilon^2}{4}\int_{0}^{t}\E\bigl[D_u^3\Phi^{(0,N)}(t-s,u^{\epsilon,N}(s)).\bigl(-\IL u^{\epsilon,N}(s)+f_N(u^{\epsilon,N}(s)),e_n,e_n\bigr)\bigr]\,ds,\label{eq:defEepsN11223}
\\
E^{\epsilon,N}_{1,1,2,2,4}(t)
&=-\frac{\epsilon^2}{8}\int_{0}^{t}\sum_{n,m=1}^{N}q_nq_m\E\bigl[D_u^4\Phi^{(0,N)}(t-s,u^{\epsilon,N}(s)).\bigl(e_n,e_n,e_m,e_m\bigr)\bigr]\,ds.\label{eq:defEepsN11224}
\end{align}

\textit{Treatment of the error term $E^{\epsilon,N}_{1,1,2,2,1}(t)$}.
Applying the inequality \eqref{eq:propo-regularity-PhiN} from Proposition~\ref{propo:regularity-PhiN} with $k=1$ and $\alpha_1=0$, and Assumption~\ref{ass:mapping-f}, one has
\begin{align*}
|E^{\epsilon,N}_{1,1,2,2,1}(t)|&\le C(T)\vvvert\varphi\vvvert_1\epsilon^2\sum_{n=1}^{N}q_n\int_0^t \E\bigl[ \|D^2f_N(u^{\epsilon,N}(s)).(e_n,e_n)\|_H \bigr] \,ds\\
&\le C(T,\varphi)\epsilon^2\sum_{n=1}^Nq_{n}\|e_n\|_H^2\\
&\le C(T,\varphi)\epsilon^2\sum_{n=1}^Nq_{n}.
\end{align*}
Let $\alpha\in(0,\beta)$, then one has
\[
\sum_{n=1}^Nq_{n}\le \lambda_N^{1-\alpha}\sum_{n=1}^{N}q_n\lambda_{n}^{\alpha-1}\le C_\alpha\lambda_{N}^{1-\alpha},
\]
with $C_\alpha=\underset{N\in\N}\sup~\sum_{n=1}^{N}q_n\lambda_{n}^{\alpha-1}<\infty$ owing to the condition~\eqref{eq:condition_beta} from Assumption~\ref{ass:noise}.

One thus obtains the following inequality for the error term $E^{\epsilon,N}_{1,1,2,2,1}(t)$: for all $\alpha\in(0,\beta)$, there exists $C_\alpha(T,\varphi)\in(0,\infty)$ such that for all $t\in[0,T]$, all $\epsilon\in(0,1)$ and all $N\in\N$ one has
\begin{equation}\label{eq:boundEepsN11221}
\big|E^{\epsilon,N}_{1,1,2,2,1}(t)\big|\le C_\alpha(T,\varphi)\epsilon^2\lambda_N^{1-\alpha}.
\end{equation}

\textit{Treatment of the error term $E^{\epsilon,N}_{1,1,2,2,2}(t)$}.
The error term $E^{\epsilon,N}_{1,1,2,2,2}(t)$ defined by~\eqref{eq:defEepsN11222} is decomposed as
\[
E^{\epsilon,N}_{1,1,2,2,2}(t)=E^{\epsilon,N}_{1,1,2,2,2,1}(t)+E^{\epsilon,N}_{1,1,2,2,2,2}(t),
\]
where the error terms $E^{\epsilon,N}_{1,1,2,2,2,1}(t)$ and $E^{\epsilon,N}_{1,1,2,2,2,2}(t)$ are defined as
\begin{align}
E^{\epsilon,N}_{1,1,2,2,2,1}(t)
&=\frac{\epsilon^2}{2}\sum_{n=1}^{N}q_n\int_{0}^{t}\E\bigl[D_{uu}^2\Phi^{(0,N)}(t-s,u^{\epsilon,N}(s)).\bigl(\IL e_n,e_n\bigr)\bigr]\,ds,\label{eq:defEepsN112221}\\
E^{\epsilon,N}_{1,1,2,2,2,2}(t)
&=-\frac{\epsilon^2}{2}\sum_{n=1}^{N}\int_{0}^{t}\E\bigl[D_{uu}^2\Phi^{(0,N)}(t-s,u^{\epsilon,N}(s)).\bigl(Df_N(u^{\epsilon,N}(s)).e_n,e_n\bigr)\bigr]\,ds.\label{eq:defEepsN112222}
\end{align}
Let $\alpha\in(0,\beta)$ and $\kappa=(\beta-\alpha)/2\in(0,1)$.

On the one hand, applying the inequality \eqref{eq:propo-regularity-PhiN} from Proposition~\ref{propo:regularity-PhiN} with $k=2$, $\alpha_1=2-\kappa$ and $\alpha_2=0$, one has
\begin{align*}
|E^{\epsilon,N}_{1,1,2,2,2,1}(t)|&\le C_\kappa(T)\vvvert\varphi\vvvert_2\epsilon^2\sum_{n=1}^{N}q_n\int_{0}^{t}(t-s)^{-1+\frac{\kappa}{2}}\|\IL e_n\|_{H^{-2+\kappa}}\|e_n\|_H\,ds\\
&\le C_\kappa(T,\varphi)\epsilon^2\sum_{n=1}^{N}q_n\lambda_n^{\kappa}.
\end{align*}

On the other hand, applying the inequality \eqref{eq:propo-regularity-PhiN} from Proposition~\ref{propo:regularity-PhiN} with $k=2$ and $\alpha_1=\alpha_2=0$, and recalling that the mapping $f$ grows at most linearly owing to Assumption~\ref{ass:fLip}, one has
\begin{align*}
|E^{\epsilon,N}_{1,1,2,2,2,2}(t)|&\le C_\kappa(T)\vvvert\varphi\vvvert_2\epsilon^2\sum_{n=1}^{N}q_n\int_{0}^{t}\E[\|Df_N(u^{\epsilon,N}(s)).e_n\|]\|e_n\|_H\,ds\\
&\le C_\kappa(T,\varphi)\epsilon^2\sum_{n=1}^{N}q_n.
\end{align*}
Moreover, note that
\[
\lambda_1^{\kappa}\sum_{n=1}^{N}q_n+\sum_{n=1}^{N}q_n\lambda_n^{\kappa}\le 2\sum_{n=1}^{N}q_n\lambda_n^{\frac{\beta-\alpha}{2}}\le 2\lambda_N^{1-\alpha}\sum_{n=1}^{N}q_n\lambda_n^{\frac{\alpha+\beta}{2}-1}\le C_\alpha\lambda_N^{1-\alpha},
\]
with $C_\alpha=\underset{N\in\N}\sup~\sum_{n=1}^{N}q_n\lambda_{n}^{\frac{\alpha+\beta}{2}-1}<\infty$ owing to the condition~\eqref{eq:condition_beta} from Assumption~\ref{ass:noise}.

One thus obtains the following inequality for the error term $E^{\epsilon,N}_{1,1,2,2,2}(t)$: for all $\alpha\in(0,\beta)$, there exists $C_\alpha(T,\varphi)\in(0,\infty)$ such that for all $t\in[0,T]$, all $\epsilon\in(0,1)$ and all $N\in\N$ one has
\begin{equation}\label{eq:boundEepsN11222}
\big|E^{\epsilon,N}_{1,1,2,2,2}(t)\big|\le C_\alpha(T,\varphi)\epsilon^2\lambda_N^{1-\alpha}.
\end{equation}

\textit{Treatment of the error term $E^{\epsilon,N}_{1,1,2,2,3}(t)$}.
The error term $E^{\epsilon,N}_{1,1,2,2,3}(t)$ defined by~\eqref{eq:defEepsN11223} is decomposed as
\[
E^{\epsilon,N}_{1,1,2,2,3}(t)=E^{\epsilon,N}_{1,1,2,2,3,1}(t)+E^{\epsilon,N}_{1,1,2,2,3,2}(t),
\]
where the error terms $E^{\epsilon,N}_{1,1,2,2,3,1}(t)$ and $E^{\epsilon,N}_{1,1,2,2,3,2}(t)$ are defined as
\begin{align}
E^{\epsilon,N}_{1,1,2,2,3,1}(t)
&=\frac{\epsilon^2}{4}\sum_{n=1}^{N}q_n\int_{0}^{t}\E\bigl[D_{u}^3\Phi^{(0,N)}(t-s,u^{\epsilon,N}(s)).\bigl(\IL u^{\epsilon,N}(s), e_n,e_n\bigr)\bigr]\,ds,\label{eq:defEepsN112231}\\
E^{\epsilon,N}_{1,1,2,2,3,2}(t)
&=-\frac{\epsilon^2}{4}\sum_{n=1}^{N}\int_{0}^{t}\E\bigl[D_{u}^3\Phi^{(0,N)}(t-s,u^{\epsilon,N}(s)).\bigl(f_N(u^{\epsilon,N}(s)),e_n,e_n\bigr)\bigr]\,ds.\label{eq:defEepsN112232}
\end{align}
Let $\alpha\in(0,\beta)$.

On the one hand, applying the inequality \eqref{eq:propo-regularity-PhiN} from Proposition~\ref{propo:regularity-PhiN} with $k=3$, $\alpha_1=2-\alpha$ and $\alpha_2=\alpha_3=0$, one has
\begin{align*}
|E^{\epsilon,N}_{1,1,2,2,3,1}(t)|&\le C_\alpha(T)\vvvert\varphi\vvvert_3\epsilon^2\sum_{n=1}^{N}q_n\int_{0}^{t}(t-s)^{-1+\frac{\alpha}{2}}\E[\|\IL u^{\epsilon,N}(s)\|_{H^{-2+\alpha}}]\|e_n\|_H^2\,ds\\
&\le C_\alpha(T,\varphi)\epsilon^2\sum_{n=1}^{N}q_n\underset{s\in[0,T]}\sup~\E[\|X^{\epsilon,N}(s)\|_{\HH^\alpha}].
\end{align*}

On the other hand, applying the inequality \eqref{eq:propo-regularity-PhiN} from Proposition~\ref{propo:regularity-PhiN} with $k=3$ and $\alpha_1=\alpha_2=\alpha_3=0$, and recalling that the mapping $f$ grows at most linearly owing to Assumption~\ref{ass:fLip}, one has
\begin{align*}
|E^{\epsilon,N}_{1,1,2,2,3,2}(t)|&\le C_\alpha(T)\vvvert\varphi\vvvert_3\epsilon^2\sum_{n=1}^{N}q_n\int_{0}^{t}\E[\|f_N(u^{\epsilon,N}(s))\|_H]\|e_n\|_H^2\,ds\\
&\le C_\alpha(T,\varphi)\epsilon^2\sum_{n=1}^{N}q_n\bigl(1+\underset{s\in[0,T]}\sup~\E[\|X^{\epsilon,N}(s)\|_{\HH}]\bigr).
\end{align*}
As above, one has
\[
\sum_{n=1}^Nq_{n}\le \lambda_N^{1-\alpha}\sum_{n=1}^{N}q_n\lambda_{n}^{\alpha-1}\le C_\alpha\lambda_{N}^{1-\alpha},
\]
with $C_\alpha=\underset{N\in\N}\sup~\sum_{n=1}^{N}q_n\lambda_{n}^{\alpha-1}<\infty$ owing to the condition~\eqref{eq:condition_beta} from Assumption~\ref{ass:noise}.

Owing to the moment bounds~\eqref{eq:propo-sdwe-momentbounds-Galerkin} from Proposition~\ref{propo:sdwe-momentbounds} and to the condition~\eqref{eq:init-weak-bound} from Assumption~\ref{ass:init-weak} on the initial value $x_0^\epsilon$, one thus obtains the following inequality for the error term $E^{\epsilon,N}_{1,1,2,2,3}(t)$: for all $\alpha\in(0,\beta)$, there exists $C_\alpha(T,\varphi)\in(0,\infty)$ such that for all $t\in[0,T]$, all $\epsilon\in(0,1)$ and all $N\in\N$ one has
\begin{equation}\label{eq:boundEepsN11223}
\big|E^{\epsilon,N}_{1,1,2,2,3}(t)\big|\le C_\alpha(T,\varphi)\epsilon^2\lambda_N^{1-\alpha}.
\end{equation}

\textit{Treatment of the error term $E^{\epsilon,N}_{1,1,2,2,4}(t)$}.
Let $\alpha\in[0,\beta)$. Applying the inequality \eqref{eq:propo-regularity-PhiN} from Proposition~\ref{propo:regularity-PhiN} with $k=4$, $\alpha_1=\alpha_2=1-\alpha$, and $\alpha_3=\alpha_4=0$, one has
\begin{align*}
|E^{\epsilon,N}_{1,1,2,2,4}(t)|&\le C_{\alpha}(T)\vvvert\varphi\vvvert_4\epsilon^2 \sum_{n,m=1}^Nq_nq_m\int_0^t(t-s)^{-1+\alpha}\|e_n\|_{H^{-1+\alpha}}^2\|e_m\|_H^2\,ds\\
&\le C_\alpha(T,\varphi)\epsilon^2\sum_{n=1}^{N}q_n\lambda_n^{\alpha-1}\sum_{m=1}^{N}q_m.
\end{align*}
As above, one has
\[
\sum_{m=1}^Nq_{m}\le \lambda_N^{1-\alpha}\sum_{m=1}^{N}q_m\lambda_{m}^{\alpha-1}\le C_\alpha\lambda_{N}^{1-\alpha},
\]
with $C_\alpha=\underset{N\in\N}\sup~\sum_{m=1}^{N}q_m\lambda_{m}^{\alpha-1}<\infty$ owing to the condition~\eqref{eq:condition_beta} from Assumption~\ref{ass:noise}.

One thus obtains the following inequality for the error term $E^{\epsilon,N}_{1,1,2,2,4}(t)$: for all $\alpha\in(0,\beta)$, there exists $C_\alpha(T,\varphi)\in(0,\infty)$ such that for all $t\in[0,T]$, all $\epsilon\in(0,1)$ and all $N\in\N$ one has
\begin{equation}\label{eq:boundEepsN11224}
\big|E^{\epsilon,N}_{1,1,2,2,4}(t)\big|\le C_\alpha(T,\varphi)\epsilon^2\lambda_N^{1-\alpha}.
\end{equation}

{\it Conclusion}.
Recalling the decomposition~\eqref{eq:decompose-E1122} of the error term $E^{\epsilon,N}_{1,1,2,2}(t)$ and combining the upper bounds~\eqref{eq:boundEepsN11221},~\eqref{eq:boundEepsN11222}
,~\eqref{eq:boundEepsN11223} and~\eqref{eq:boundEepsN11224}, one thus obtains the following inequality for the error term $E^{\epsilon,N}_{1,1,2,2}(t)$: for all $\alpha\in(0,\beta)$, there exists $C_\alpha(T,\varphi)\in(0,\infty)$ such that for all $t\in[0,T]$, all $\epsilon\in(0,1)$ and all $N\in\N$ one has
\begin{equation}\label{eq:boundEepsN1122}
\big|E^{\epsilon,N}_{1,1,2,2}(t)\big|\le C_\alpha(T,\varphi)\epsilon^2\lambda_N^{1-\alpha}.
\end{equation}

$\bullet$ {\bf Conclusion}.
Recalling the decomposition~\eqref{eq:decompose-E112} of the error term $E^{\epsilon,N}_{1,1,2}(t)$ and combining the upper bounds~\eqref{eq:boundEepsN1121} and~\eqref{eq:boundEepsN1122}, one obtains the inequality~\eqref{eq:error-T3b} from Lemma~\ref{lem:error-T3}. The proof is thus completed.

\end{proof}

\subsection{Proof of Theorem~\ref{theo:weak}}\label{sec:proof-theo-weak-final}

As explained at the beginning of Section~\ref{sec:proof-theo-weak}, Theorem~\ref{theo:weak} is obtained by combining Theorem~\ref{theo:weak-N} (which depends on the parameter $N\in\N$ of the auxiliary spectral Galerkin approximation) and Lemmas~\ref{lem:weakerror-0-galerkin}  and~\ref{lem:weakerror-epsilon-galerkin}.

\begin{proof}[Proof of Theorem~\ref{theo:weak}]
Owing to Theorem~\ref{theo:weak-N} and Lemmas~\ref{lem:weakerror-0-galerkin}  and~\ref{lem:weakerror-epsilon-galerkin}, for all $t\in[0,T]$, and for all $\epsilon\in(0,1)$ and $N\in\N$, one has
\begin{align*}
\big|\E[\varphi(u^{\epsilon}(t))]-\E[\varphi(u^{0}(t))]\big|&\le\big|\E[\varphi(u^{\epsilon,N}(t))]-\E[\varphi(u^{0,N}(t))]\big|\\
&+\big|\E[\varphi(u^{0,N}(t))]-\E[\varphi(u^{0}(t))]\big|\\
&+\big|\E[\varphi(u^{\epsilon}(t))]-\E[\varphi(u^{\epsilon,N}(t))]\big|\\
&\le C_{\alpha}(T,\varphi)\Bigl(\epsilon^{\min(2\alpha,1)}+\epsilon\lambda_N^{\max(\frac12-\alpha,0)}+\epsilon^2\lambda_N^{1-\alpha}+\lambda_N^{-\alpha}\Bigr).
\end{align*}

The weak error appearing in the left-hand side above does not depend on the auxiliary parameter $N\in\N$. If the parameter $N$ is chosen to depend on $\epsilon$ such that $\lambda_N\sim \epsilon^{-2}$, then one has
\[
\epsilon\lambda_N^{\max(\frac12-\alpha,0)}+\epsilon^2\lambda_N^{1-\alpha}+\lambda_N^{-\alpha}\sim 3\epsilon^{2\alpha}.
\]
As a result, for all $\epsilon\in(0,1)$ one obtains
\[
\underset{t\in[0,T]}\sup~\big|\E[\varphi(u^{\epsilon}(t))]-\E[\varphi(u^{0}(t))]\big|\le C_{\alpha}(T,\varphi)\epsilon^{\min(2\alpha,1)}.
\]
The proof of the weak error estimates~\eqref{eq:theo-weak} is thus completed. This concludes the proof of Theorem~\ref{theo:weak}.
\end{proof}


\begin{thebibliography}{10}

\bibitem{MR4632567}
C.-E. Br\'ehier.
\newblock Uniform strong and weak error estimates for numerical schemes applied
  to multiscale {SDE}s in a {S}moluchowski-{K}ramers diffusion approximation
  regime.
\newblock {\em J. Comput. Dyn.}, 10(3):387--424, 2023.

\bibitem{MR4942808}
Z.~a. Brze\'zniak and S.~Cerrai.
\newblock Stochastic wave equations with constraints: well-posedness and
  {S}moluchowski-{K}ramers diffusion approximation.
\newblock {\em Comm. Math. Phys.}, 406(9):Paper No. 223, 59, 2025.

\bibitem{MR4993799}
S.~Cerrai and A.~Debussche.
\newblock Smoluchowski-{K}ramers diffusion approximation for systems of
  stochastic damped wave equations with nonconstant friction.
\newblock {\em Ann. Appl. Probab.}, 35(6):4106--4171, 2025.

\bibitem{MR2240691}
S.~Cerrai and M.~Freidlin.
\newblock On the {S}moluchowski-{K}ramers approximation for a system with an
  infinite number of degrees of freedom.
\newblock {\em Probab. Theory Related Fields}, 135(3):363--394, 2006.

\bibitem{MR2267703}
S.~Cerrai and M.~Freidlin.
\newblock Smoluchowski-{K}ramers approximation for a general class of {SPDE}s.
\newblock {\em J. Evol. Equ.}, 6(4):657--689, 2006.

\bibitem{MR3583470}
S.~Cerrai, M.~Freidlin, and M.~Salins.
\newblock On the {S}moluchowski-{K}ramers approximation for {SPDE}s and its
  interplay with large deviations and long time behavior.
\newblock {\em Discrete Contin. Dyn. Syst.}, 37(1):33--76, 2017.

\bibitem{MR4056993}
S.~Cerrai and N.~Glatt-Holtz.
\newblock On the convergence of stationary solutions in the
  {S}moluchowski-{K}ramers approximation of infinite dimensional systems.
\newblock {\em J. Funct. Anal.}, 278(8):108421, 38, 2020.

\bibitem{MR3245077}
S.~Cerrai and M.~Salins.
\newblock Smoluchowski-{K}ramers approximation and large deviations for
  infinite dimensional gradient systems.
\newblock {\em Asymptot. Anal.}, 88(4):201--215, 2014.

\bibitem{MR3531676}
S.~Cerrai and M.~Salins.
\newblock Smoluchowski-{K}ramers approximation and large deviations for
  infinite-dimensional nongradient systems with applications to the exit
  problem.
\newblock {\em Ann. Probab.}, 44(4):2591--2642, 2016.

\bibitem{MR3575542}
S.~Cerrai and M.~Salins.
\newblock On the {S}moluchowski-{K}ramers approximation for a system with
  infinite degrees of freedom exposed to a magnetic field.
\newblock {\em Stochastic Process. Appl.}, 127(1):273--303, 2017.

\bibitem{MR4142946}
S.~Cerrai, J.~Wehr, and Y.~Zhu.
\newblock An averaging approach to the {S}moluchowski-{K}ramers approximation
  in the presence of a varying magnetic field.
\newblock {\em J. Stat. Phys.}, 181(1):132--148, 2020.

\bibitem{MR4413207}
S.~Cerrai and G.~Xi.
\newblock A {S}moluchowski-{K}ramers approximation for an infinite dimensional
  system with state-dependent damping.
\newblock {\em Ann. Probab.}, 50(3):874--904, 2022.

\bibitem{MR4657218}
S.~Cerrai and M.~Xie.
\newblock On the small noise limit in the {S}moluchowski-{K}ramers
  approximation of nonlinear wave equations with variable friction.
\newblock {\em Trans. Amer. Math. Soc.}, 376(11):7651--7689, 2023.

\bibitem{MR4759624}
S.~Cerrai and M.~Xie.
\newblock On the small-mass limit for stationary solutions of stochastic wave
  equations with state dependent friction.
\newblock {\em Appl. Math. Optim.}, 90(1):Paper No. 7, 48, 2024.

\bibitem{MR4865023}
S.~Cerrai and M.~Xie.
\newblock The small-mass limit for some constrained wave equations with
  nonlinear conservative noise.
\newblock {\em Electron. J. Probab.}, 30:Paper No. 25, 27, 2025.

\bibitem{MR3014105}
M.~Freidlin and W.~Hu.
\newblock Smoluchowski-{K}ramers approximation in the case of variable
  friction.
\newblock volume 179, pages 184--207. 2011.
\newblock Problems in mathematical analysis. No. 61.

\bibitem{GV25}
B.~Guelmame and J.~Vovelle.
\newblock A smoluchowski-kramers approximation for the stochastic variational
  wave equation, 2025.

\bibitem{MR3324144}
S.~Hottovy, A.~McDaniel, G.~Volpe, and J.~Wehr.
\newblock The {S}moluchowski-{K}ramers limit of stochastic differential
  equations with arbitrary state-dependent friction.
\newblock {\em Comm. Math. Phys.}, 336(3):1259--1283, 2015.

\bibitem{MR2916095}
S.~Hottovy, G.~Volpe, and J.~Wehr.
\newblock Noise-induced drift in stochastic differential equations with
  arbitrary friction and diffusion in the {S}moluchowski-{K}ramers limit.
\newblock {\em J. Stat. Phys.}, 146(4):762--773, 2012.

\bibitem{LLX26}
S.~Liu, W.~Liu, and L.~Xu.
\newblock Long term convergence rate of smoluchowski-kramers approximation by
  stein's method, 2026.

\bibitem{LQW24}
X.~Liu, Q.~Jiang, and W.~Wang.
\newblock The smoluchowski-kramers approximation for a system with arbitrary
  friction depending on both state and distribution, 2024.

\bibitem{D.2006The}
D.~Nualart.
\newblock {\em The {M}alliavin calculus and related topics}.
\newblock Probability and its Applications (New York). Springer-Verlag, Berlin,
  second edition, 2006.

\bibitem{DE2018}
D.~Nualart and E.~Nualart.
\newblock {\em Introduction to {M}alliavin calculus}, volume~9 of {\em
  Institute of Mathematical Statistics Textbooks}.
\newblock Cambridge University Press, Cambridge, 2018.

\bibitem{MR4054345}
M.~Rousset, Y.~Xu, and P.-A. Zitt.
\newblock A weak overdamped limit theorem for {L}angevin processes.
\newblock {\em ALEA Lat. Am. J. Probab. Math. Stat.}, 17(1):1--21, 2020.

\bibitem{MR3916264}
M.~Salins.
\newblock Smoluchowski-{K}ramers approximation for the damped stochastic wave
  equation with multiplicative noise in any spatial dimension.
\newblock {\em Stoch. Partial Differ. Equ. Anal. Comput.}, 7(1):86--122, 2019.

\bibitem{SLW24}
C.~Shi, Y.~Lv, and W.~Wang.
\newblock The smoluchowski-kramers approximation for a mckean-vlasov equation
  subject to environmental noise with state-dependent friction, 2024.

\bibitem{MR4235249}
C.~Shi and W.~Wang.
\newblock Small mass limit and diffusion approximation for a generalized
  {L}angevin equation with infinite number degrees of freedom.
\newblock {\em J. Differential Equations}, 286:645--675, 2021.

\bibitem{SW24}
C.~Shi and W.~Wang.
\newblock The smoluchowski-kramers approximation with distribution-dependent
  potential and highly oscillating force, 2024.

\bibitem{MR4713516}
T.~C. Son, D.~Q. Le, and M.~H. Duong.
\newblock Rate of convergence in the {S}moluchowski-{K}ramers approximation for
  mean-field stochastic differential equations.
\newblock {\em Potential Anal.}, 60(3):1031--1065, 2024.

\bibitem{MR4102448}
N.~V. Tan and N.~T. Dung.
\newblock A {B}erry-{E}sseen bound in the {S}moluchowski-{K}ramers
  approximation.
\newblock {\em J. Stat. Phys.}, 179(4):871--884, 2020.

\bibitem{MR4419570}
L.~Xie and L.~Yang.
\newblock The {S}moluchowski-{K}ramers limits of stochastic differential
  equations with irregular coefficients.
\newblock {\em Stochastic Process. Appl.}, 150:91--115, 2022.

\bibitem{ZLW25}
Q.~Zhao, X.~Liu, and W.~Wang.
\newblock Smoluchowski--kramers approximation with state-dependent friction in
  rough path topology, 2025.

\bibitem{ZW24}
Q.~Zhao and W.~Wang.
\newblock Convergence rate of smoluchowski–kramers approximation with stable
  levy noise, 2024.

\bibitem{MR4908783}
Y.~Zine.
\newblock Smoluchowski-{K}ramers approximation for singular stochastic wave
  equations in two dimensions.
\newblock {\em Electron. J. Probab.}, 30:Paper No. 88, 49, 2025.

\end{thebibliography}
\end{document}